\documentclass[12pt,reqno]{amsart}
\usepackage{amsmath,amssymb,amsfonts,amsthm,mathtools}
\usepackage{enumerate}
\usepackage[paperwidth=8.26in,paperheight=11.69in,
left=1in, right=1in, bottom=1.2in]{geometry}

\usepackage{hyperref}
\usepackage{multicol}
\usepackage{microtype}

\usepackage{thmtools}
\declaretheoremstyle[
 headfont=\normalfont\bfseries,
 headindent= 0pt,
 bodyfont=\em,
 spaceabove=8pt,
 spacebelow=8pt
]{thm}
\declaretheoremstyle[
 headfont=\normalfont\em,
 headindent= 0pt,
 spaceabove=8pt,
 spacebelow=8pt
]{remark}
\declaretheoremstyle[
 headfont=\normalfont\bfseries,
 headindent= 0pt,
 spaceabove=8pt,
 spacebelow=8pt
]{example}
\declaretheoremstyle[
 headfont=\normalfont\bfseries,
 headindent= 0pt,
 spaceabove=8pt,
 spacebelow=8pt
]{definition}
\declaretheorem[name=Theorem,style=thm,numberwithin=section,
]{thm}

\declaretheorem[name=Lemma,style=thm,sibling=thm]{lem}
\declaretheorem[name=Corollary,style=thm,sibling=thm]{cor}
\declaretheorem[name=Example,style=example,sibling=thm]{example}
\declaretheorem[name=Definition,style=thm,sibling=thm]{defn}
\declaretheorem[name=Remark,style=remark]{rem}

\usepackage{cleveref}
\crefname{thm}{Theorem}{Theorems}
\crefname{prop}{Proposition}{Propositions}
\crefname{lem}{Lemma}{Lemmas}
\crefname{cor}{Corollary}{Corollaries}
\crefname{example}{Example}{Examples}
\crefname{defn}{Definition}{Definitions}
\crefname{rem}{Remark}{Remarks}

\crefname{equation}{}{}

\numberwithin{equation}{section}
\newcommand{\RR}{\mathbb{R}}
\newcommand{\CC}{\mathbb{C}}
\newcommand{\NN}{\mathbb{N}}

\newcommand{\gba}{\bar{g}}
\newcommand{\fba}{\bar{f}}

\newcommand{\wba}{\bar{w}}

\newcommand{\zba}{\bar{z}}
\newcommand{\Zba}{\overline{Z}}
\newcommand{\Lba}{\overline{L}}
\newcommand{\heis}{\mathbb{H}^3}
\newcommand{\lc}{\mathcal{X}}
\newcommand{\tfour}{{D^{\mathrm{IV}}_3}}
\newcommand{\tfourN}{{D^{\mathrm{IV}}_{n+1}}}
\newcommand{\tfourM}{{D^{\mathrm{IV}}_{m}}}
\newcommand{\mr}{{\mathcal{R}}}

\newcommand{\lba}{\bar{\lambda}}
\newcommand{\pba}{\bar{\phi}}
\newcommand{\vr}{\varrho}

\newcommand{\invar}{(\phi g + i f^2)}

\DeclareMathOperator{\Aut}{Aut}
\DeclareMathOperator{\trace}{tr}

\renewcommand{\Re}{\operatorname{Re}}
\renewcommand{\Im}{\operatorname{Im}}

\begin{document}
\title{On CR maps from the sphere into the tube over the future light cone}
\author{Michael Reiter}
\address{Fakult\"at f\"ur Mathematik, Universit\"at Wien, Oskar-Morgenstern-Platz 1, 1090 Wien, Austria}
\email{m.reiter@univie.ac.at}
\author{Duong Ngoc Son}
\address{Fakult\"at f\"ur Mathematik, Universit\"at Wien, Oskar-Morgenstern-Platz 1, 1090 Wien, Austria}
\email{son.duong@univie.ac.at}
\begin{abstract}
	We determine all local smooth or formal CR maps from the\break unit sphere $\mathbb{S}^3\subset \mathbb{C}^2$ into the tube $\mathcal{T}:= \mathcal{C} \times i\mathbb{R}^3 \subset \mathbb{C}^3$ over the future light cone $\mathcal{C}:= \left\{x\in \mathbb{R}^3\colon x_1^2+x_2^2 = x_3^2, \ x_3 > 0\right\}$. This result leads to a complete classification of proper holomorphic maps from the unit ball in $\mathbb{C}^2$ into Cartan's classical domain of type IV in $\mathbb{C}^3$ that extend smoothly to some boundary point. Up to composing with CR automorphisms of the source and target, the classification consists of four algebraic maps. Two maps among them were known earlier in the literature, which were shown to be ``rigid'' in the higher dimensional case in a recent paper by Xiao and Yuan. Two newly discovered quadratic polynomial maps provide counterexamples to a conjecture appeared in the same paper for the case of dimension two.
\end{abstract}

\date{June 11, 2021}
\thanks{The  authors were supported by the Austrian Science Fund (FWF): Projekt I3472 (Internationale Projekte-RSF/Russland 2017).}
\maketitle
\section{Introduction}
The study of proper holomorphic maps between balls goes back to Poincar\'e \cite{poincare1907fonctions} and Alexander \cite{alexander1977proper} who showed that the proper holomorphic self-maps of a ball in $\mathbb{C}^N$, $N\geqslant 2$, are precisely the  automorphisms which have explicit parametrizations. When the target is a ball of higher dimension, the discovery of inner functions implies that there are proper holomorphic maps between balls that do not extend to any boundary point \cite{aleksandrov1982existence,hakim1982fonctions,low1982construction,low1985embeddings}. On the other hand, by a well-known result of Forstneri\v{c} \cite[Theorem~1.4]{forstneric1989extending}, proper holomorphic maps between balls, which extend sufficiently smooth to a boundary point, extend to rational maps with no poles on the boundary and induce global CR maps between spheres. 

There is an interesting phenomenon for CR maps of spheres (i.e. \emph{sphere maps}) depending on the size of the \emph{codimension} (i.e. the difference of the CR dimensions of the source and target sphere.) When the codimension is ``low,'' the maps are ``rigid,'' which means that there is only the linear map up to composition with automorphisms, see \cite{Webster79b, faran1986linearity, Huang99}. When the codimension is not too large, the sphere maps may be completely classified \cite{faran1982maps,faran1986linearity,huang2001mapping,HJX06,HJY14}. In case there is no restriction on the codimension, the collection of sphere maps is known to be huge \cite{dangelo1988proper, Hamada05}. A particular interesting case is when the source is $\mathbb{S}^3\subset \mathbb{C}^2$. In this case, the smooth CR maps into $\mathbb{S}^5 \subset \mathbb{C}^3$ have been classified by Faran \cite{faran1982maps} as four ``spherical equivalence'' classes of maps. See also \cite{CS89, Ji10, Reiter16a} for different proofs. However, a classification of the maps from $\mathbb{S}^3$ into $\mathbb{S}^7$ is only available for monomial maps~\cite{dangelo1988proper}.

Besides the study of proper holomorphic maps of balls, there is a large literature devoted to the study of proper holomorphic maps between various models such as the complex ellipsoids, Reinhardt domains, and the classical domains. In particular, there are many works devoted to the study of proper holomorphic maps as well as  isometric holomorphic embeddings from one bounded symmetric domain into another by Mok, Ng, Xiao, Yuan and others, see, e.g. \cite{mok2012germs,xiao2020holomorphic}. In \cite{xiao2020holomorphic}, Xiao--Yuan give complete classifications of proper maps from $\mathbb{B}^n$ into $\tfourM \subset \CC^m$ (the Cartan's classical domain of type IV of $m$ dimension) in the ``low'' codimensional case (i.e., $4\leqslant n \leqslant m-1 \leqslant 2n-4$) and for the holomorphic isometries from $\mathbb{B}^n$ into $\tfourN  \subset \CC^{n+1}$ for all $n\geqslant 2$. When the codimension is higher, there are many more explicit examples of proper holomorphic maps given in \cite{xiao2020holomorphic}. However, in the case $n=2$ and $m=3$, only two equivalence classes of proper holomorphic maps were known (and they are all isometries), yet the classification problem was left open in these particular dimensions. We refer the reader to, e.g., \cite{Mok2011, Mok2018, mok2012germs,xiao2020holomorphic} and their references for more detailed information.

In this paper, we study local CR maps from $\mathbb{S}^3$ into the well-known \emph{tube over the future light cone} $\mathcal{T}$ of real dimension $5$ which has been of interest in many papers, e.g., \cite{ebenfelt1998normal, kaup2006local}. This is an everywhere Levi-degenerate 2-nondegenerate CR manifold (cf. \cite[Example~4.2.1]{ebenfelt1998normal}) which is homogeneous and has a ``large'' stability group; see, e.g., \cite{kaup2006local,fels2007cr}. Moreover, $\mathcal{T}$ is locally CR equivalent to the smooth boundary part $\mathcal{R}$ of $\tfour$ and our study provides a complete classification of proper holomorphic maps from $\mathbb{B}^2$ into $\tfour$ which extend smoothly to a boundary point. 
\begin{defn}
	Let $M$ and $M'$ be CR manifolds and let $H \colon (M, p) \to M'$, $\tilde H \colon (M, \tilde p)\to M'$ be germs of smooth CR maps at $p$ and $\tilde{p}$, respectively. We say that $H$ and $\tilde{H}$ are \emph{equivalent} if there exist germs of local CR diffeomorphisms $\gamma \colon (M, p) \to (M, \tilde p) $ and $\psi \colon (M', \tilde{H}(\tilde{p})) \to (M',H(p))$, such that
	\[
	H = \psi \circ \tilde H \circ \gamma^{-1}.
	\]
	We say that CR maps are \emph{equivalent} if they represent two equivalent germs at some interior points in their domains of definition. 
\end{defn}

In our main theorem below, we classify the germs of CR maps from $\mathbb{S}^3$ into $\mathcal{T}$ as five equivalence classes of map germs. Their representing maps have simple formulas when being expressed in the Heisenberg model $\heis := \left\{(z,w)\in \mathbb{C}^2\colon \Im w = |z|^2\right\}$ for $\mathbb{S}^3$ and the rational model $\lc$ for the tube given explicitly as follows.
\[
\lc : = \left\{(z,w)\in \mathbb{C}^2 \colon v = \frac{|z|^2 + \Re (z^2 \bar{\zeta})}{1 - |\zeta|^2}, \ |\zeta|^2 <1\right\}, \quad w = u+ iv.
\]
An explicit local equivalence of $\lc$ and $\mathcal{T}$ is given by Fels and Kaup in \cite[Proposition~4.16]{fels2007cr}. 
\begin{thm}\label{thm:main1}
Let $U$ be an open subset of $\heis$ and $H \colon U \to \lc$ a smooth CR map. Then the following hold:
\begin{enumerate}[(a)]
\setlength{\itemsep}{3pt}
\item If $H$ is CR transversal at some point $p\in U$, then $H$ is CR transversal on $U$, the germs $(H,q)$, $q\in U$, are mutually equivalent and are equivalent to exactly one of the following four pairwise inequivalent germs at the origin:
\begin{enumerate}[(i)]
	\setlength{\itemsep}{3pt}
	\item $\ell(z,w) = (z,0,w)$,
	\item $r_1(z,w) = \left(\dfrac{z(1+i w)}{1-w^2},\dfrac{2 z^2}{1-w^2},\dfrac{w}{1-w^2}\right)$,
	\item $r_{-1}(z,w) = \left(\dfrac{z(1-i w)}{1-w^2},\dfrac{-2 z^2}{1-w^2},\dfrac{w}{1-w^2}\right) $,
	\item $\iota (z,w) = \left(2 z,\  2 w,\ 2 w\right)\bigl/ \left(1+\sqrt{1-4 w^2-4 i z^2}\right)$.
\end{enumerate}
Consequently, $H$ extends as an algebraic map from $\CC^2$ to $\CC^3$ sending $\heis$ into $\lc$.
\item If $H$ is nowhere CR transversal, then for each $q$, the germ $(H,q)$ is equivalent to the germ at the origin of a map $t_q \colon (z,w) \mapsto (0, \phi_q(z,w),0)$ for a smooth CR function $\phi_q$ defined in a neighbourhood of the origin.
\end{enumerate}
\end{thm}

\begin{rem}
The smoothness assumption in \cref{thm:main1} can be weakened to $C^2$-smoothness by a result of Kossovskiy--Lamel--Xiao \cite{KLX21}. Regularity of CR maps into the tube over the future light cone has been studied in \cite{Mir17, Xiao20, Greilhuber20}. Similar statements hold for germs of formal holomorphic maps sending $\heis$ into $\lc$; we omit the details.
\end{rem}

As briefly mentioned above, this theorem implies a complete classification of proper holomorphic maps from $\mathbb{B}^2$ to the classical domain $\tfour$ that extend smoothly to some smooth boundary point. Recall \cite[\S V]{cartan1935domaines} that the Cartan's classical domain of type IV in dimension $m$, denoted by $\tfourM$, is the domain of $m$-dimensional vectors 
\[
Z = (z_1,z_2,\dots, z_m ) \in \mathbb{C}^m
\] 
satisfying the inequalities
\[
1- 2 Z \Zba^t + \left|Z Z^t\right|^2 > 0, \quad \left| ZZ ^t \right| < 1,
\]
where $Z^t$ is the $m\times 1$ transpose of $Z$. This domain is also called the \textit{Lie ball} and the \textit{complex sphere} in the literature. This domain is homogeneous: the global autormophism group $\Aut(\tfour)$ acts transitively on it. We say that two proper holomorphic maps $H,\widetilde{H}\colon \mathbb{B}^2\to \tfour$ are \emph{equivalent} if 
there exist $\psi\in \Aut(\tfour)$ and $\gamma \in \Aut(\mathbb{B}^2)$ such that $\widetilde{H}:=\psi \circ H\circ \gamma^{-1}$. Then we obtain a classification of proper holomorphic maps from $\mathbb{B}^2$ to $\tfour$ as follows.
\begin{cor}\label{cor:balld4}
Let $H\colon \mathbb{B}^2 \to \tfour$ be a proper holomorphic map which extends smoothly to some boundary point $p\in \partial \mathbb{B}^2$. 
Then $H$ is equivalent to exactly one of the following four pairwise inequivalent maps:
\begin{enumerate}[(i)]
	\setlength{\itemsep}{3pt}
	\item $R_0(z,w) = \left(\dfrac{z}{\sqrt{2}},\dfrac{2 w^2+2 w-z^2}{4 (w+1)},\dfrac{i \left(2 w^2+2 w+z^2\right)}{4 (w+1)}\right)$,
	\item $P_1(z,w) = \left(zw, \dfrac{z^2-w^2}{2}, \dfrac{i(z^2+w^2)}{2}\right)$,
	\item $P_{-1}(z,w) = \left(z, \dfrac{w^2}{2}, \dfrac{iw^2
	}{2}\right)$,
	\item $I(z,w) = \left(z,w,1-\sqrt{1-z^2-w^2}\,\right)\bigl/ \sqrt{2}$.
\end{enumerate}
\end{cor}
\begin{rem}\label{rem2}
The maps $R_0$ and $I$ were known earlier in the classification of isometric embeddings, studied in, e.g.,  \cite{chan2017holomorphic,upmeier2019holomorphic}, by methods that are very different from ours. They also appeared in \cite{xiao2020holomorphic} which studies the rigidity of transversal holomorphic maps sending a piece of the sphere into the smooth boundary part of a type IV domain in higher dimension. They are essentially the only isometric embeddings (up to a normalizing constant) of the respective Bergman metrics on $\mathbb{B}^2$ and $\tfour$. The quadratic polynomial maps $P_{\pm 1}$ are not isometric and provide counterexamples to Conjecture 2.9 in \cite{xiao2020holomorphic} of Xiao and Yuan in the case the source has dimension $n= 2$ (the case $n=3$ and the case of $C^1$-extendability are still open). In fact, there is an 1-parameter smooth family of rational proper maps $R_{\alpha}$, $\alpha \in \RR$, such that $R_{\alpha} \sim P_1$ for $\alpha > 0$ and $R_{\alpha}  \sim P_{-1}$ for $\alpha <0$, yet $R_0$ is not equivalent to $R_{\alpha}$ for $\alpha\ne 0$. In essence, the map $R_0$ is not ``locally rigid'' (see \cite{della2020sufficient} for a recent discussion on the local rigidity of CR maps).  This exhibits an important difference of our setting from the case $n\geqslant 4$ and the case of isometries treated in \cite{xiao2020holomorphic,upmeier2019holomorphic,chan2017holomorphic} in which the rigidity holds.

It is immediate to see that these four maps are pairwise inequivalent. Indeed, $R_0$ and $I$ (which is irrational) have singularities on the sphere, the image of $\mathbb{S}^3$ via $P_1$ lies entirely in the smooth part $\mr \subset\partial \tfour$, while $P_{-1}$ sends the circle $\{(e^{it},0)\} \subset \mathbb{S}^3$ into the singular part of $\partial\tfour$ and sends $\mathbb{S}^3 \setminus \{(e^{it},0)\}$ transversally into $\mr$. Hence, the inequivalences are evident.

Our proof of Corollary~\ref{cor:balld4} is independent from that of the classification of proper holomorphic maps and isometric embeddings in \cite{chan2017holomorphic,upmeier2019holomorphic,xiao2020holomorphic}.
\end{rem}
\begin{rem}
Corollary~\ref{cor:balld4} is analogous to Faran's well-known classification of maps from $\mathbb{B}^2$ to $\mathbb{B}^3$ as a list of four equivalence classes of maps \cite{faran1982maps} (see also \cite{Reiter16a,Ji10}). We point out, however, that there is a difference when considering \emph{local} equivalences of map germs. For each map $H\in \{R_0, P_1, P_{-1}, I\}$ and each $p, p'\in \mathbb{S}^3$ such that $H(p), H(p') \in \mr$, the germs $(H,p)$ and $(H,p')$ are equivalent. In contrast, the sphere map $(z,w)\mapsto(z,zw,w^2)$ from $\mathbb{B}^2$ to $\mathbb{B}^3$ represents inequivalent germs when the base point $p$ varies on and off the circle of ``umbilical points'' $\{(0,e^{i\theta}) \in \mathbb{S}^3\colon t\in \mathbb{R}\}$. The sphere map $(z,w) \mapsto (z^3,\sqrt{3} zw, w^3)$ also exhibits a similar phenomenon, see  \cite{Reiter16a} and \cite[Table~1]{lamel2019cr}. 
\end{rem}
\begin{rem}	
There is a convenient way to construct proper holomorphic maps from a ball into a type IV domain of certain dimension as first noticed in \cite{xiao2020holomorphic}. In essence, it relies on an equation of the form
\begin{equation}\label{xy}
2\sum_{j=1}^m |h_j|^2 - \left|\sum_{j=1} ^mh_j^2\right| ^2 = \sum _{j=1}^{m'} |f_j|^2,
\end{equation}
where $H =(h_1, h_2, \dots, h_m)$ and $F= (f_1,f_2,\dots, f_{m'})$ are two holomorphic maps defined on some relevant open set of $\CC^n$. If $H$ maps $\mathbb{S}^{2n-1}$ into $\partial\tfourM$, then $F$ maps $\mathbb{S}^{2n-1}$ into $\mathbb{S}^{2m'+1}$ and vice versa, if $F$ maps $\mathbb{S}^{2n-1}$ into $\mathbb{S}^{2m'+1}$, then $H$ maps $\mathbb{S}^{2n-1}$ into the real variety $\{\vr_{\tfourM} = 0\}$ which contains $\partial\tfourM$. Equation \eqref{xy} implies that $H|_{\mathbb{S}^{2n+1}}$ and $F|_{\mathbb{S}^{2n+1}}$ have the same Ahlfors tensor (defined in Section 2). If we take $H=P_1$, then \eqref{xy} holds for $F(z,w) = (z^2, \sqrt{2} zw, w^2)$, which maps $\mathbb{S}^3$ into $\mathbb{S}^5$. Both maps have Ahlfors tensor equal to $1/2$ (when evaluated on an appropriate frame). If $H=P_{-1}$, then there is no map $F\colon \mathbb{S}^3\to \mathbb{S}^5$ such that \eqref{xy} holds, since $P_{-1}$ has a negative Ahlfors tensor, cf. \cite[Example 2.8]{xiao2020holomorphic}, which can never be satisfied by any sphere map, see \cite{lamel2019cr}.
\end{rem}
Each of the isometric embeddings $R_0$ and $I$ restricts to a CR map on an open dense subset of $\mathbb{S}^3$ having the property that a certain coefficient in its partial normal form (given below in \eqref{eq:normalform}) vanishes. In this case we say that the map has vanishing \emph{geometric rank} at the centered point. The notion of geometric rank in our present setting is similar to that of sphere maps, first used by Huang \cite{Huang99}, see also \cite{Huang03}. It also appears recently in \cite{huang2020boundary} as a boundary characterization for isometric embeddings into indefinite hyperbolic space. We define this invariant notion for CR transversal maps into the tube over the future light cone precisely in Definition~\ref{def:gr}. By inspecting the special case of vanishing geometric rank in the proof of Theorem~\ref{thm:main1}, we can observe a phenomenon that is similar to the main result in \cite{huang2020boundary}.
\begin{cor}
	Let $H\colon \mathbb{B}^2 \to \tfour$ be a proper holomorphic map. If $H$ extends smoothly to a neighbourhood of a point $p\in \mathbb{S}^3$ and has vanishing geometric rank at $p$, then $H$ is a holomorphic isometry. Conversely, each isometric embedding from $\mathbb{B}^2$ into $\tfour$ extends to a real-analytic CR map from an open dense proper subset of the sphere $\mathbb{S}^3$ to the boundary of $\partial \tfour$ with vanishing geometric rank at every point and is everywhere CR transversal in its domain.
\end{cor}

We conclude this introduction by describing our approach in the proof of Theorem~\ref{thm:main1} and the main difficulties we need to overcome. Similarly to, e.g. \cite{Reiter16a,xiao2020holomorphic}, we analyze the ``mapping equation'' for the components of a map $H$ which expresses the fact that $H$ sends $\heis$ into $\lc$. But there are differences to the positive but ``low'' codimension case, as treated in \cite{xiao2020holomorphic}, when ``rigidity'' holds, i.e., when there are only ``few'' equivalence classes of maps: An essential ingredient for the rigidity is Huang's lemma \cite{Huang99} and the fact that the maps are all of vanishing ``geometric rank''. In our present setting (as well as in the case of sphere maps from $\mathbb{S}^3$ to $\mathbb{S}^5$, cf. \cite{faran1982maps,Reiter16a}) Huang's lemma is not available and rigidity \emph{fails}. Thus, the solution set in the present setting is expected to be rather ``large'' and we need to solve the mapping equation for an \emph{a priori} unknown number of (discrete or continuous families of) solutions.

To solve the mapping equation, we first normalize the map (the ``unknown'') to a specific form, given in \eqref{eq:normalform}, using explicit formulas for the local CR automorphisms of the source and target (\cref{lem:stab}). In this partial normal form, the mapping equation can be viewed as either a system of infinitely many (linear and nonlinear) equations for the coefficients in the Taylor series expansions of the components the map (the jets of the map at the origin) or a system of infinitely many equations arising from differentiation and evaluation along the first Segre variety $\Sigma:=\{(z,w)\in \mathbb{C}^2\colon w=0\}$ of $\mathbb{H}^3$ at the origin. By the $2$-nondegeneracy of $\lc$, the jets of the map of all orders are determined once the $4$-jet is determined. Such a behavior is also observed in the sphere case, see \cite{Reiter16a}.

A bulk of the present paper is to identify the $4$-jets at the origin that arise from maps sending $\heis$ into $\lc$. This process involves analyzing the mapping equation and various derived ones. It turns out that the $4$-jets of the maps can only be fully identified by analyzing all equations arising from the Taylor series expansions upto weighted order $16$. There are examples of maps which satisfy all but three of weight $14$ (\cref{ex:holEqUpTo14}). This shows that considering such high orders is in some sense necessary. Once we determine the $4$-jets of the maps at the origin, and at the same time, the $2$-jets along the first Segre variety, we can construct a system of three holomorphic functional equations depending on two real parameters whose unique and explicit solutions, as presented in Theorem~\ref{thm:main1}, are genuine maps sending $\heis$ into $\lc$.

The computational complexity also poses some difficulties. To overcome this issue, we use the computer algebra system Mathematica when doing some formal differentiations, collecting terms in polynomial expressions, and computing determinants of matrices of some (relatively small) size. But we want to point out, that the operations and expressions involved in the proofs are in principle simple and elementary enough to be checked by hand.

The paper is organized as follows: In \cref{sec:prelim} we provide the local representations of the tube over the future light cone, its automorphisms and an Ahlfors-type tensor, which allows us to distinguish equivalence classes of maps. In \cref{sec:normalform} we give a partial normal form for the maps and introduce the geometric rank in this setting. In \cref{sec:proofs} we give the proof of our main result \cref{thm:main1} and in \cref{sec:examples} we discuss additional examples of maps and provide explicit biholomorphisms to show equivalence to the maps listed in \cref{thm:main1}.

\section{Preliminaries}
\label{sec:prelim}
\subsection{Models for the tube over the future light cone}
Following \cite{fels2007cr} we would like to describe representations of the tube over the future light cone. A model for the three-dimensional space time is the space $V$ of real symmetric $2\times 2$-matrices where the time coordinate is the normalized trace of the matrices. In this model, the \emph{future cone} is
\[
	\Omega:=\left\{v = \begin{pmatrix}
		t + x_1 & x_2 \\ x_2 & t -x_1
	\end{pmatrix} \colon v \text{ positive definite }\right\}.
\]
The smooth boundary part of $\Omega$ is the \emph{future light cone} $\mathcal{C}$ is given explicitly by
\[
	t^2 = x_1^2 + x_2^2, \quad t> 0.
\]
To construct the tube over $\mathcal{C}$, one identifies $V \oplus iV$ with the space of complex symmetric $2\times 2$-matrices. The \emph{tube over the future light cone} is given by
\[
	\mathcal{T} : = \mathcal{C}\oplus iV = \left\{ z\in V \oplus iV \colon \det(z + \zba) = 0, \ \Re \trace (z+ \zba) >0\right\}.
\]
The tube $\mathcal{T}$ is the smooth boundary part of a tube domain with everywhere degenerate Levi-form. In fact, it is everywhere $2$-nondegenerate as a CR manifold; see \cite[Example~4.2.1]{ebenfelt1998normal}. It has been studied in many papers \cite{ebenfelt1998normal,kaup2006local,fels2007cr}. Our interest in this tube is that it is \emph{holomorphically homogeneous}, i.e., its CR automorphism group acts transitively on it, and has quite ``large'' stability group at each point. Moreover, from Fels--Kaup  \cite{fels2007cr}, $\mathcal{T}$ has a local rational model similar as in the case of the sphere. Namely, at every point it is locally CR equivalent to the real hypersurface $\lc$ near the origin in $\mathbb{C}_{z,\zeta,w}$ given by
\[
v = \frac{|z|^2 + \Re (z^2 \bar{\zeta})}{1 - |\zeta|^2}, \quad w = u+ iv.
\]
In fact, the rational map
\[
	\Phi: (z,\zeta, w)\mapsto
	\frac{1}{1+ \zeta}\begin{pmatrix}
		z^2-i \zeta  w-i w & \sqrt{2} w \\
		\sqrt{2} w & 1-\zeta  \\
	\end{pmatrix}
\]
sends $\lc$ into $\mathcal{T}$ \cite[Proposition 4.16]{fels2007cr} and is locally biholomorphic on a dense set of $\mathbb{C}^3$. 
In this model, the automorphism group and the stability group have quite simple representations. In this paper, we shall work extensively with this model.

For the tube model, the maps are given as follows.
\begin{thm} Let $U\subset \mathbb{H}^3$ be a connected open set and $H\colon U \to \mathcal{T}$ be a smooth CR map. Then
\begin{enumerate}[(a)]
	\item If $H$ is CR transversal at some point $p\in U$, then for all $q\in U$, the germs $(H,q)$ are mutually equivalent and is equivalent to one of the following germs at the origin:
\begin{enumerate}[(i)]
\item $\mathcal{R}_0(z,w) = \begin{pmatrix}
	z^2-i w & \sqrt{2} z \\
	\sqrt{2} z & 1  \\
\end{pmatrix}$,
\item $\mathcal{R}_1(z,w) = \dfrac{1}{1 - w^2 + 2 z^2}
\begin{pmatrix}
	z^2-i w & \sqrt{2}z (1+i w) \\
	\sqrt{2}z (1+i w) & 1-w^2-2 z^2 \\
\end{pmatrix}$,
\item $\mathcal{R}_{-1}(z,w) = \dfrac{1}{1 - w^2 - 2 z^2}
\begin{pmatrix}
	z^2-i w & \sqrt{2}z (1-i w) \\
	\sqrt{2}z (1-i w) & 1-w^2+2 z^2 \\
\end{pmatrix}$,
\item $\mathcal{I}(z,w) = \dfrac{1}{1+2w+\varepsilon}\begin{pmatrix}
-i (1+2 w-\varepsilon)& 2\sqrt{2}z \\
2\sqrt{2}z & 1-2 w+\varepsilon \\
\end{pmatrix}$, $\varepsilon := \sqrt{1-4 w^2-4 i z^2}$.
\end{enumerate}
	\item If $H$ is nowhere CR transversal, then for each $p \in U$, $(H,p)$ is equivalent to a germ at the origin of the form $\phi(z,w) = \begin{pmatrix}
		0 & 0 \\
		0 & \varphi \\
	\end{pmatrix}$
	for some  smooth CR function $\varphi$.
\end{enumerate}
\end{thm}
It is well-known that $\mathcal{T}$ is locally CR equivalent to the smooth boundary part of Cartan's classical domain of type IV. Recall from \cite[Section V]{cartan1935domaines} that the symmetric domain of type IV of $m$ dimension, denoted by $\tfourM$, is the domain of $m$-dimensional vectors 
\[
Z = (z_1,z_2,\dots, z_m ) \in \mathbb{C}^m
\] 
satisfying the conditions
\[
1- 2 Z \Zba^t + \left|Z Z^t\right|^2 > 0, \quad \left| ZZ ^t \right| < 1,
\]
where $Z^t$ is the $m\times 1$ transpose of $Z$. This domain is also called the \emph{Lie ball} or the \emph{complex sphere} in the literature. The complex 3-dimensional case (i.e. $m=3$) is rather special as $\tfour$ is also equivalent to the classical domain of type II, denoted by $D^{\mathrm{II}}_2$, consisting of  
symmetric $2\times 2$-matrices $Z$ such that $I_{2\times 2} - Z \bar{Z}$ is positive definite. Explicitly,
\[
	\vr_{\tfour}
	=
	\det \left(I_{2\times 2} - \begin{pmatrix}
		z_1+ i z_3 & i z_2\\
		iz_2 & z_1-i z_3
	\end{pmatrix} \cdot \begin{pmatrix}
	\zba_1- i \zba_3 & -i \zba_2\\
	-i\zba_2 & \zba_1+i \zba_3
\end{pmatrix} \right)
\]
while the trace of the matrix in the parenthesis is just $1- |Z Z^t|^2$.
The biholomorphisms of $\tfour$ and $D^{\mathrm{II}}_2$ are well-known.
\subsection{Automorphisms of the tube over the future light cone}
\label{sec:autom}
The local CR automorphisms of the tube over the future light cone are well-studied. They are the restrictions of birational but not necessarily biholomorphic transformations of $\mathbb{C}^3$; a local CR equivalence of $\mathcal{T}$ does not necessarily extend to a global CR equivalence \cite{kaup2006local}. By homogeneity, the local CR equivalences of $\mathcal{T}$ are determined once the automorphism group $\Aut(\mathcal{T},p)$ of the germs at an arbitrary point $p$ is determined. In fact, $\Aut(\mathcal{T},p)$ is a solvable real Lie group of dimension $5$ (isomorphic to the stability of a point on a sphere in $\mathbb{C}^2$ \cite{kaup2006local}). For our purpose, we shall describe this group explicitly in the local rational model of Fels--Kaup \cite{fels2007cr}. The result of Fels--Kaup is presented in \cite{kolar2019complete} by writing the Lie algebra $\mathfrak{g}$ as
\[
	\mathfrak{g} = \mathfrak{g}_{-2} \oplus \mathfrak{g}_{-1} \oplus \mathfrak{g}_0 \oplus \mathfrak{g}_1 \oplus \mathfrak{g}_2.
\]

The stability group $G_0^c$ generated by the component 
\[
	\mathfrak{g}_0^c: = \mathrm{span} \left\{ z\frac{\partial }{\partial z} + 2 w \frac{\partial }{\partial w}, \ iz \frac{\partial }{\partial z} + 2i \zeta \frac{\partial }{\partial \zeta} \right\}
\]
is parametrized by $\lambda>0$ and $\varphi \in \RR$ as follows
\[
(z,\zeta,w) \mapsto \left(\lambda e^{i\varphi} z, e^{2i\varphi} \zeta, \lambda^2 w\right).
\]
The component $\mathfrak{g}_2$ is 
\[
\mathfrak{g}_2 = \mathrm{span}\left\{zw \frac{\partial}{\partial z} - i z^2 \frac{\partial}{\partial \zeta} + w^2 \frac{\partial}{\partial w} \right\}.
\]
Integrating yields 1-parameter subgroup, we get
\[
(z,\zeta,w) \mapsto \bigl(z, \zeta - s w\zeta - i s z^2, w \bigr)/(1 - sw), \quad s \in \RR.
\]
The component $\mathfrak{g}_1$ is given by
\begin{multline*}
	\mathfrak{g}_1
	=
	\mathrm{span} \biggl\{
	(z^2+iw(\zeta +1)) \frac{\partial}{\partial z} + 2 z(\zeta+1) \frac{\partial}{\partial \zeta} + 2zw \frac{\partial }{\partial w}, \\ (iz^2-w(\zeta-1)) \frac{\partial}{\partial z} + 2i z(\zeta - 1) \frac{\partial}{\partial \zeta} + 2izw \frac{\partial }{\partial w}
	\biggr\}.
\end{multline*}
Integrating the first vector field, we obtain
\[
(z,\zeta,w) \mapsto \bigl(z(1- tz) + i t (1+ \zeta)w,-\delta_1+\zeta+1,w \bigr)/\delta_1, \ \delta_1 = (1-t z)^2-i (\zeta +1) t^2 w.
\]
Integrating the second vector field, we have
\[
(z,\zeta,w) \mapsto \bigl(z(1- irz) - r (\zeta-1)w,\delta_2+\zeta-1,w \bigr)\bigl/\delta_2, \quad  \delta_2 = (1-i r z)^2+i(\zeta - 1)r^2 w.
\]
Composing the automorphisms obtained above, we obtain
\begin{lem}[Stability group]\label{lem:stab} Let $a\in \CC, t \in \RR, u\in \mathbb{C}, |u|=1, \lambda >0$ and
\[
\delta = 1-2 i \bar a z - (t + i |a|^2) w  + i\bar a^2 (w\zeta + iz^2).
\]
Then the stability group $\Aut(\lc,0)$ is the 5 dimensional group consisting of the automorphisms 
\begin{align}\label{eq:stab}
\psi_{r,t,u,\lambda}\colon 
\begin{cases}
	\tilde{z} =  \lambda u \left(z+aw - \bar{a} (w\zeta + iz^2)\right) \delta^{-1} \\
	\tilde{w} = \lambda^2 w \delta^{-1} \\
	\tilde{\zeta} = u^2 \left(\zeta - 2i a z - i a^2 w - (t - i|a|^2) (w\zeta + iz^2)\right) \delta^{-1}
\end{cases}
\end{align}
\end{lem}

Infinitesimal automorphisms, which do not vanish at the origin, give rise to automorphisms moving points. We provide them in the following:

The vector field $\partial / \partial w$ spanning $\mathfrak g_{-2}$ integrates to the 1-parameter family
\begin{equation}\label{eq:trans1}
	 (z,\zeta,w)\mapsto (z,\zeta,w+s)\quad \text{for} \quad s \in \RR.
\end{equation}
The component $\mathfrak g_{-1}$ is given by
\[
\mathfrak{g}_{-1} = \mathrm{span} \left\{(1-\zeta) \frac{\partial}{\partial z} + 2 i z \frac{\partial}{\partial w}, i (1+\zeta) \frac{\partial}{\partial z} + 2 z \frac{\partial}{\partial w} \right\},
\]
whose vector fields integrate to 
\begin{align*}
	(z,\zeta,w) & \mapsto (z + t'(1-\zeta),\zeta,w + 2 i t' z + i t'^2 (1-\zeta)), \qquad t' \in \RR,\\
	(z,\zeta,w) & \mapsto (z + i s'(1+\zeta), \zeta, w + 2 s' z + i s'^2(1+\zeta)), \qquad s' \in \RR.
\end{align*}
Combining the three above automorphisms gives the following 3-parameter family which generates the Lie group corresponding to $\mathfrak{g}_{-1} \oplus \mathfrak{g}_{-2}$, 
\begin{align}\label{eq:transAuto1}
	\mathbf{\tau}_{b,r} \colon	(z,\zeta,w) \mapsto (z + b - \bar{b}  \zeta,\ \zeta,\ w + r + i |b|^2 + 2 i \bar b z + i\bar b^2/2-i b^2/2 - i \bar b^2 \zeta),
\end{align}
for $b \in \CC$ and $r \in \RR$, which maps $0$ to $(b,0,r+ i |b|^2 +i\bar b^2/2-ib^2/2)$.

The component $\mathfrak{g}_0^s$ is given by
\begin{align*}
	\mathfrak{g}_0^s = \mathrm{span} \left\{-z\zeta \frac{\partial}{\partial z} + (1-\zeta^2)\frac{\partial}{\partial \zeta} + i z^2 \frac{\partial}{\partial w},\ iz\zeta \frac{\partial}{\partial z} + i(1+\zeta^2)\frac{\partial}{\partial \zeta} + z^2 \frac{\partial}{\partial w} \right\}.
\end{align*}
If we write
\begin{align*}
	\varepsilon_1 = 1+ \lambda' + (\lambda'-1)\zeta, \qquad \lambda'> 0,
\end{align*}
then the first vector field integrates to
\begin{align}\label{eq:transAuto2}
	(z,\zeta,w) & \mapsto \left(2 \sqrt{\lambda'} z, \varepsilon'_1+2(\zeta-1), (\lambda'+1)w+(\lambda'-1)(w\zeta + i z^2) \right)/\varepsilon_1.
\end{align}
If we write
\begin{align*}
	\varepsilon_2 = 1+ \mu' -i (\mu'-1)\zeta, \qquad \mu'> 0,
\end{align*}
then the second vector field integrates to
\begin{align*}
	(z,\zeta,w) & \mapsto \left(2 \sqrt{\mu'} z, i(\mu'-1-i(\mu'+1)\zeta), (\mu'+1)w-i(\mu'-1)( w\zeta+iz^2)\right)/\varepsilon_2.
\end{align*}

If we compose the two 1-parameter families of automorphisms and write,
\begin{align*}
	c' = 1 + i(\mu'+\lambda') + \mu' \lambda', \qquad d' = 1+i(\mu'-\lambda') -\mu'\lambda', \qquad \gamma'= c' - d' \zeta,
\end{align*}
we obtain the following 2-parameter family of automorphisms
\begin{align*}
	\psi\colon (z,\zeta,w) \mapsto \left(2 \sqrt{\mu'\lambda'}(1+i) z, i(\bar c'\zeta-\bar d'), c' w - d' (w \zeta + i z^2) \right)/\gamma',
\end{align*}
which sends $0$ to $(0,- i \bar d'/c',0)$.

As an application of these explicit representations of the transitive automorphisms of $\lc$, we explicitly show in \cref{ex:linmap} below, that the rational map $R: \mathbb{B}^2 \rightarrow \tfour$ obtained in \cite[Theorem 1.4]{xiao2020holomorphic} corresponds to the linear map $\ell$, listed in \cref{thm:main1}, as a map from $\heis$ into $\lc$.

To deduce Corollary~\ref{cor:balld4} from Theorem \ref{thm:main1}, we shall use the local equivalence of the rational model $\lc$ and the smooth boundary part $\mr$ of $\tfour$ together with an well-known fact that (in contrast with the local CR equivalences of the rational model) each local CR equivalence of $\mr$ extends to a birational transformation of $\mathbb{C}^3$ that restricts to an automorphisms of $\tfour$ \cite[Section~6]{fels2007cr}, \cite{kaup2006local}. In the case of the sphere and Heisenberg hypersurface, this extension phenomenon has been known since the works of Poincar\'e, Tanaka, and Alexander. For general classical domains of rank $\geqslant2$, the same result was established by Mok--Ng \cite{mok2012germs}.
\begin{thm}\label{thm:alex} Let $p \in \mr$ be a point in the smooth boundary part $\mr$ of $\tfour$ and let $U$ be a neighbourhood of $p$ in $\mr$. If $H\colon U \to \mr$ is a local CR diffeomorphism, then there exists a rational map $\widetilde{H}$ such that $\widetilde{H}|_{\tfour}$ is an isometry 
of $\tfour$ with respect to the Bergman metric and $H = \widetilde{H}|_{U}$.	
\end{thm}
We should also note a well-known theorem of Tumanov--Khenkin \cite{tumanov1982local} which gives a similar conclusion when the smooth boundary part is replaced by the Shilov boundary, the ``skeleton'' in the terminology of \cite{tumanov1982local}, of the domain. 
For the sake of completeness, we give an elementary proof that is based on Lemma~\ref{lem:stab} for the special case of $\tfour$.
\begin{proof}
Since the automorphisms of $\tfour$ act transitively on $\mr$, we may assume that $H(p) = p = (0,\frac{1}{2},\frac{i}{2})$. The rational map $\Phi$ 
\begin{equation}\label{eq:phi}
	\Phi(z,\zeta,w)
	=
	\left(\frac{2i z}{2i + w}, \ \frac{2i  - w - 2i \zeta -(w\zeta + iz^2)}{2(2i+w)},\ \frac{i\left(2i - w + 2i \zeta +(w\zeta + iz^2) \right)}{2(2i+w)}\right)
\end{equation}
sends a neighbourhood $V$ of the origin biholomorphically onto some neighbourhood $U$ of $p$ with $\Phi(0,0,0) = p$ and locally $\lc$ into $\mr$. Then $\gamma:=\Phi^{-1} \circ H\circ \Phi$ defines a local CR diffeomorphism of $\lc$ fixing the origin, i.e., $\gamma$ represents an element of the stability group $\Aut(\lc,0)$ and has the form \eqref{eq:stab} for an appropriate quadruple of parameters $a, u, t$, and $\lambda$. Therefore, $H = \Phi \circ \gamma \circ \Phi^{-1}$ agrees with a rational automorphism (in fact, an isometry) of $\tfour$ as can be checked directly. Indeed, explicitly one can take
\begin{align}\label{eq:phiinv}
	\Phi^{-1}(z_1,z_2,z_3)
	=
	\left(\frac{2z_1}{1+z_2 -i z_3},\ -\frac{z_2 + i z_3 + z_1^2+z_2^2+z_3^2}{1+z_2 -i z_3}, \ \frac{2i(1-z_2 + i z_3)}{1+z_2 -i z_3}\right)
\end{align}
Then $\Phi \circ \gamma \circ \Phi^{-1}$ is a second degree rational mapping of the form
\begin{equation}\label{eq:t4local}
\widetilde{H}=\Phi \circ \gamma \circ \Phi^{-1}
=
\left(\frac{\psi_1}{\Delta},\ \frac{\psi_2}{2\Delta},\ \frac{\psi_3}{2\Delta}\right)
\end{equation}
where the components of the map are
\begin{align*}
	\psi_1 & = 2u\lambda (ia + z_1 + 2z_2 \Im(a) -z_3 \Re(a) -i \bar{a} (z_1^2 + z_2^2+z_3^2) ), \\
	\psi_2 & = 1+2|a|^2 -\lambda^2 -2it -2a^2 u^2 + 4iz_1 \left(a u^2-\bar{a}\right)\\ 
	& \quad + z_2\left( u^2 \left(2 a^2-2 | a|^2 +\lambda ^2-2 i t+1\right)+
	1-2| a|^2 +\lambda ^2+2 \bar{a}^2+2 i t\right)\\
	& \quad -i z_3\left(u^2 \left(2 a^2+2| a|^2-\lambda ^2+2 i t-1\right)+1-2| a|^2+\lambda ^2-2 \bar{a}^2+2 i t\right)\\
	& \quad + \left(u^2 \left(1+2 |a|^2-\lambda ^2+2 i t\right)-2 \bar{a}^2\right)(z_1^2+z_2^2+z_3^2),\\
	\psi_3 & = i\bigl(1+2 |a|^2-\lambda ^2-2 i t + 2 a^2 u^2 - 4i z_1 (au^2 + \bar{a}) \\
	& \quad 		+ z_2(1-2 |a|^2+\lambda ^2+2 \bar{a}^2+2 i t - u^2(1-2|a|^2 +\lambda^2 + 2a^2 -2i t)) \\
	& \quad -i z_3\left(u^2 \left(1-2 |a|^2+\lambda ^2-2 a^2-2 i t\right)+1-2|a|^2+\lambda ^2-2 \bar{a}^2+2 i t\right)\\
	& \quad -\left(u^2 \left(1+2 |a|^2-\lambda ^2+2 i t+1\right)+2 \bar{a}^2\right)(z_1^2+z_2^2+z_3^2)\bigr),
\end{align*}
and the common denominator is
\begin{align*}
	\Delta 
	=
	(1+2 |a|^2+\lambda ^2-2 i t)-2\bar{a}^2(z_1^2+z_2^2+z_3^2) - 4i \bar{a}z_1\\
	+ (1-2 |a|^2- \lambda ^2 + 2 \bar{a}^2 + 2 i t)z_2
	-i(1-2 |a|^2 -\lambda ^2 -2 \bar{a}^2 + 2 i t) z_3.
\end{align*}
Equation \eqref{eq:t4local} gives the most general parametric formula for the CR automorphisms of the germ $(\mr,p)$.

To show that the rational map $\widetilde{H}$ restricts to an automorphism of $\tfour$, it suffices to show that $\Delta$ does not vanish on $\tfour$. 
If we write
\[
	W = \left(\frac{-2ia}{1+2|a|^2+\lambda^2+2it},\ \frac{\lambda^2-1+2|a|^2-2a^2+2it}{2(1+2|a|^2+\lambda^2+2it)},\ \frac{i(\lambda^2-1+2|a|^2+2a^2 +2it)}{2(1+2|a|^2+\lambda^2+2it)}\right)
\]
then $W\in \tfour$ and $\widetilde{H}(W) =(0,0,0)$ and further more
\[
	\Delta
	=
	(1+2 |a|^2+\lambda ^2-2 i t)\left(1 - 2 \overline{W}^t Z + \overline{(W^t W)}\,(Z^t Z)\right).
\]
But from the explicit formula for the Bergman kernel $K_{\tfour}(Z,W)$ of $\tfour$ (see \cite{hua1963harmonic}), we have
\[
K_{\tfour}(Z,W) = \mathrm{volume}(\tfour)^{-1} \left(\frac{\Delta}{1+2 |a|^2+\lambda ^2-2 i t}\right)^{-3}
\] 
and the nonvanishing of $\Delta$ on $\tfour$ is equivalent to the well-known fact that $\tfour$, being a homogeneous complete circular domain, has nonvanishing Bergman kernel (i.e, is a Lu-Qi Keng domain).
\end{proof}

\subsection{An invariant for CR maps into the tube over the future light cone}\label{sec:ahlfors}
In this section, we analyze a tensor attached to each transversal CR map into the tube over the future light cone or into $\mr$ that is similar to the CR Ahlfors derivative of CR immersions studied recently by Lamel and the second author. We refer the reader to \cite{lamel2019cr} for some details about the origin, motivation, and a general differential geometric construction of the CR Ahlfors tensor for CR immersions in the strictly pseudoconvex setting. In our current setting, the construction is rather ad-hoc, yet the resulting tensor is still useful for our purpose. Namely, it has an invariant property and gives an easy way to distinguish inequivalent maps into the tube over the future light cone. Indeed, we shall use it in the proof of Corollary~\ref{cor:balld4}.

Let $M\subset \mathbb{C}^{n+1}$ be a real hypersurface defined by $\vr_M =0$ and $H\colon M \to \mr$ a CR transversal CR map extending holomorphically to a neighbourhood of $M$. By the CR transversality,
there is a real-valued smooth $v$ such that 
\[
	\vr_{\tfourM} \circ H = \pm e^{v} \vr_M, \quad 
	\vr_{\tfourM}: = 1 - 2 Z \overline{Z}^t + |Z Z^t|^2.
\]
Without lost of generality, we assume that the positive sign occurs. If $M$ is a sphere we take $\vr_M := 1 - |z|^2$ and fix the defining function of the target $\mr$ to be $\vr_{\tfourM} $. Then we define the \emph{Ahlfors-type tensor} $\mathcal{A}(H)$ by
\begin{equation}\label{def:ahlfors}
	\mathcal{A}(H) = v_{Z\overline{Z}}\bigl|_{T^{(1,0)}M \times T^{(0,1)}M},
\end{equation}
where $v_{Z\overline{Z}}$ is the $(1,1)$-Hessian of $v$. For any $\psi \in \Aut(\tfourM)$, we have
\[
	\vr_{\tfourM} \circ \psi = |q|^2 \vr_{\tfourM}
\]
for some rational function $q$ having no pole and zero on $\mr$. Thus, arguing as in \cite{lamel2019cr}, we see that $\mathcal{A}(H)$ is invariant with respect to composing with automorphisms of $\tfourM$, i.e., on the smooth boundary part $\mr$,
\[
	\mathcal{A}(H) = \mathcal{A}(\psi\circ H),\quad \forall \ \psi \in \Aut(\tfourM).
\]
Similarly, on the sphere
\[
	\mathcal{A}(H \circ \gamma^{-1} ) = |\tilde{q}|^2\mathcal{A}(H)
\]
holds for every automorphism $ \gamma$ of the ball with a nonvanishing CR function $\tilde{q}$. These invariant properties do not hold for general defining functions of the source and target.

Let $\{Z_{\alpha}\}$ be a local frame for the $T^{(1,0)}M$, then we say that $H$ has \emph{Ahlfors rank} $k$ at $p$ if the Hermitian form in \eqref{def:ahlfors} has rank $k$. The Ahlfors rank is invariant with respect to composition with the source and target automorphisms. Precisely, if $\mathrm{Rk}_{\mathcal{A}}(H)$ the Ahlfors rank of $H$ at $p$ (with respect to the defining functions of the source and target as above), then 
\[
	\mathrm{Rk}_{\mathcal{A}}(H)|_p = \mathrm{Rk}_{\mathcal{A}}(\psi\circ H \circ \gamma^{-1})\bigl|_{\gamma(p)}.
\]

If $M$ is the Heisenberg hypersurface and  $H\colon \heis \to \lc$ is a smooth CR map, then we can define the Ahlfors tensor of $H$  which has the desired invariant property by using the defining function $\vr_{\heis} = \Im w - |z|^2$ for the source and the defining function
\[
	\widetilde{\vr}_{\lc}(z,\zeta,w)
	=
	\frac{(1-|\zeta|^2) \Im w - |z|^2- \Re (\zeta  \zba^2)}{|w+2 i|^2}
\]
the target. In fact, with this choice of the defining function for $\lc$, we have from the explicit formulas for the local CR automorphisms in Section 2, that
\[
	\widetilde{\vr}_{\lc} \circ \psi = |q|^2 \widetilde{\vr}_{\lc}
\]
for some holomorphic (in fact, rational) function $q$. The corresponding Ahlfors tensor is invariant with respect to the action of the CR autormophisms groups of both source and target.

\section{A partial normalization and the notion of geometric rank}
\label{sec:normalform}

Let $H \colon \heis \to \lc$ be a smooth CR map from the Heisenberg hypersurface into $\lc$. We use the automorphisms of the stability groups and normalize $H$ so that its 2-jet has a specific form.

\begin{rem}
Throughout this and the preceding section we use the following notation: Let $h$ be a germ at $(0,0)$ of a holomorphic function in $\CC^2$, depending on variables $(z,w)$ in $\CC^2$. Instead of writing $\frac{\partial^{k+\ell} h}{\partial z^k\partial w^\ell}$ for $k,\ell \in \NN$, we write $h_{z \cdots z w \cdots w}$ and its evaluation at $(0,0)$ we denote by $h^{(k,\ell)}$. 
\end{rem}

\begin{thm}
	\label{thm:normalform}
	Let $p \in \heis$ and $H=(f,\phi,g)$ be a germ at $p$ of a smooth CR map, sending $\heis$ into $\lc$. Assume that $H$ is CR transversal at $p$. Then the germ $(H,p)$ is equivalent to the germ at the origin of a CR map $\widetilde H =(\widetilde f, \widetilde \phi, \widetilde g)$ satisfying the following properties:
	\begin{align}
		\nonumber
		\widetilde f(z,w) & = z + \frac{i}{2} \alpha z w + \nu w^2 + O(3),\\
		\label{eq:normalform}
		\widetilde \phi(z,w) & = \lambda w + \alpha z^2 + \mu z w + \sigma w^2 + O(3),\\
		\nonumber
		\widetilde g(z,w) & = w + O(3),
	\end{align}
	where $\alpha \in \RR$ and $\lambda,\nu,\mu, \sigma \in \CC$. Furthermore, we may assume $\lambda\in \{0,1\}$, and if $\lambda = \mu = 0$, we may assume $\alpha \in \{-1,0,1\}$.
\end{thm}

\begin{proof}
By the transitivity of the automorphisms of $\heis$ and $\lc$ we can assume $H(0)=0$.
It is easy to conclude directly from the mapping equation, that $g(z,0) = 0$ and $g^{(0,1)} = |f^{(1,0)}|^2$. By the transversality of $H$ we have $f^{(1,0)} \neq 0$ and $g^{(0,1)}>0$.
We write $H_{k+1} \coloneqq \varphi'_k \circ H_k \circ \varphi^{-1}_k$, where $\varphi_k^{-1}$ and $\varphi'_k$ belong to $\Aut(\heis,0)$ and $\Aut(\lc,0)$ respectively, $k \in \NN$ and $H_0 \coloneqq H$. We have
\begin{align*}
	H_{1 z}(0) = \left(u u' \lambda \lambda' f^{(1,0)}, u u'^2 \lambda (-2 i c' f^{(1,0)} + \phi^{(1,0)}),0\right).
\end{align*}
Since $f^{(1,0)} \neq 0$, by choosing
\begin{align*}
	u= \frac{\overline{f^{(1,0)}}}{u' |f^{(1,0)}|}, \quad \lambda = \frac{1}{\lambda' |f^{(1,0)}|}, \qquad c' = -\frac{i \phi^{(1,0)}}{2 f^{(1,0)}},
\end{align*}
we can assume that $f_1^{(1,0)}=1$ and $\phi^{(1,0)}_1=0$. This implies $g^{(0,1)}_1=1$.
Considering $H_2 = \varphi'_1 \circ H_1 \circ \varphi_1^{-1}$ with $u=1/u',\lambda = 1/\lambda'$ and $c'=0$, we compute
\begin{align*}
	f^{(0,1)}_2 = c + \frac{u' f^{(0,1)}_1}{\lambda'},
\end{align*}
and set 
\begin{align*}
	c = - \frac{u' f^{(0,1)}_1}{\lambda'},
\end{align*}
to obtain $f^{(0,1)}_2=0$. Differentiating the mapping equation for $H_2$ twice with respect to $\bar w$ gives $\Im g^{(0,2)}_2 =0$. Differentiating the mapping equation with respect to $z$ and $\bar w$ gives $g^{(1,1)}_2=0$ and with respect to $z$ twice and $\bar z$ yields $f^{(2,0)}_2=0$. 
If we take $H_3 = \varphi'_2 \circ H_2 \circ \varphi_2^{-1}$ with $u=1/u',\lambda = 1/\lambda'$ and $c=c'=0$, we obtain
\begin{align*}
	g^{(0,2)}_3 = -2 r + (2 t + g^{(0,2)}_2)/\lambda'^2, 
\end{align*}
which, after setting 
\begin{align*}
	r = \frac{2t + g^{(0,2)}_2}{2\lambda'^2},
\end{align*}
implies that $g^{(0,2)}_3 = 0$. If we plug $H_3$ into the  mapping equation and differentiate with respect to $z, \bar z$ and $\bar w$ we obtain that $\Re f^{(1,1)}_3 = 0$.
Moreover, if we differentiate the mapping equation twice with respect to $z$ and $\bar z$ we get $\Re \phi^{(2,0)}_3 = 2 \Im f^{(1,1)}_3$.
In the last step, we consider $H_4= \varphi'_3 \circ H_3 \circ \varphi_3^{-1}$ with $u=1/u',\lambda = 1/\lambda', c=c'=0$ and $r=t/\lambda'$ to obtain
\begin{align*}
	\phi^{(2,0)}_4 = \frac{-2 i t + \phi^{(2,0)}_3}{\lambda'^2}, 
\end{align*}
which, after setting 
\begin{align*}
	t = \frac{\Im \phi^{(2,0)}_3} 2,
\end{align*}
implies that $\Im \phi^{(2,0)}_4 = 0$. Finally, we compute
\begin{align*}
	\phi^{(0,1)}_4 = \frac{u'^2 \phi^{(0,1)}_3}{\lambda'^2}, 
	\qquad \phi^{(2,0)}_4 = \frac{4 \Im f^{(1,1)}_3}{\lambda'^2},  
\end{align*}
which implies the remaining normalization conditions, when setting $\alpha = \Re  \phi^{(2,0)}_4/2$ and choosing $\lambda'$ and $u'$ accordingly.
\end{proof}
In the partial normal form \eqref{eq:normalform} of a map germ at a center point $p$, the vanishing of the coefficient $\alpha$ is an invariant property of the map germ. Thus, in analogy with the case of sphere maps \cite{Huang99,Huang03}, we make the following definition.
\begin{defn}\label{def:gr}
Let $H\colon U\subset \heis \to \lc$ be a smooth transversal CR map and $p\in U$. We say that $H$ has \emph{geometric rank} zero at $p$ if $H$ can be brought into the partial (formal) normal form \eqref{eq:normalform} with $\alpha =0$. Otherwise we say that $H$ has \emph{geometric rank} $1$ at $p$.
\end{defn}
The geometric rank is precisely the rank of the Ahlfors tensor defined in Section~\ref{sec:ahlfors} but the latter can be computed easily (without the use of explicit formulas for the CR automorphisms). This relation was first noticed by Lamel and the second author in the case of sphere maps in \cite{lamel2019cr}. Thus the invariant property of the geometric rank also follows from that of the Ahlfors tensor. Indeed, if $H$ is given in \eqref{eq:normalform} and if
\[
	\widetilde{\vr}_{\lc} \circ H = \widetilde{Q}\cdot \vr_{\heis},
\]
then at the origin,
\[
	\widetilde{Q} = 2, \quad \widetilde{Q}_{z} = 0,\quad \widetilde{Q}_{z\zba} = 2\alpha.
\]
Hence, for the defining functions as above and for $L: = \partial_z - 2i \zba \partial_w$ a section of $T^{1,0} \heis$, we have
\[
	\mathcal{A}(H)(L,\Lba)\bigl|_0 = \left(\log \widetilde{Q}\right)_{Z\Zba}(L,\Lba)\biggl|_0
	=
	\widetilde{Q}^{-2} \widetilde{Q}_{z\zba}\bigl|_0 = \frac{\alpha}{2}.
\]
The relation also holds at an arbitrary point $p\in \heis$ by the invariant property of the left hand side.

\section{Proofs of Theorem~\ref{thm:main1} and its corollaries}
\label{sec:proofs}
 We denote by $\Sigma$ the first Segre set of $\heis$ at the origin. Precisely,
\[
	\Sigma = \{(z,0) \in \mathbb{C}^2\}.
\]
We can compute the map $H$ along $\Sigma$ as follows:
\begin{lem}\label{lem:first} Let $H = (f,\phi,g)$ be a map given in the normal form \cref{eq:normalform}. Let $P = P_H$ be the rational map depending on 2-jet of $H$ given by
	\[
	P(z;X)
	= P_{\alpha,\lambda,\sigma,\mu,\nu}(z;X) 
	=
	\left(X,\frac{2z(4 \bar{\nu} z+\alpha) X +\left(4 \bar{\sigma} z^2+2 i \bar{\mu} z -\alpha\right)X^2}{1 - 4i \bar{\lambda} z^2},0\right).
	\]
	Then
	\begin{equation}\label{eq:1jetSegre}
	H\bigl|_{\Sigma} = P\left(z;\frac{2z}{1+\sqrt{1-4i\lba z^2}}\right).
	\end{equation}
\end{lem}
Thus, $H$ is uniquely determined along $\Sigma$ by its 2-jet at the origin.
\begin{proof}
Consider the following defining function of the rational model $\lc$ of the tube over the future light cone:
\[
	\widetilde{\vr} (z,\zeta, w)
	=
	-\frac{i}{2}(w - \overline{w})(1 - |\zeta|^2) - |z|^2 - \Re \left(z^2 \overline{\zeta}\right).
\]
If $H$ sends $\mathbb{H}^3$ into $\lc$, then there is a function $Q(z,w,\zba,\wba)$ near $0$ such that $	\widetilde{\varrho}(H) = Q \varrho$. Explicitly
\begin{equation}\label{eq:mapeq}
	(g - \gba)(1 - |\phi|^2) - 2i|f|^2 -2i \Re \left(f^2 \pba\right)
	=
	Q(z,w,\zba,\wba) (w - \wba - 2i|z|^2). 
\end{equation}
If $H$ is real-analytic at $0$, then $Q$ can be taken to be real-analytic. If $H$ is smooth, then we can view $Q$ as a formal power series in $z,w,\zba, \wba$.
In both cases, we can treat $z$ and $\zba$ as separate variables and \cref{eq:mapeq} holds as an identity of formal power series.

Setting $\zba = \wba =0$ in \cref{eq:mapeq}, we have
\[
	g(z,w)=  w Q(z,w,0,0).
\]
We introduce the following auxiliary holomorphic functions (or formal power series)
\begin{equation} \label{eq:rsdef}
	r(z,w) = \frac{\partial Q}{\partial \wba}(z,w,0,0), \quad
	s(z,w) = \frac{i}{2}\frac{\partial Q}{\partial \zba}(z,w,0,0).
\end{equation}
Differentiating the mapping equation in $\wba $ we obtain \[
-(1 - |\phi|^2)\gba_{\wba} - \phi \pba_{\wba} (g-\gba) - 2i f \fba_{\wba} - 2i \fba \fba_{\wba} \phi - i f^2\pba_{\wba} = Q_{\wba} \vr - Q.
\] 
Evaluating at $\zba = \wba  =0$, using $\bar{H}_{\wba}(0,0) = (0,\lba, 1)$, and solving for $Q(z,w,0,0)$ we find that
\begin{equation}\label{eq:Qsol}
	g = wQ(z,w,0,0) = w(1+  wr + \lba \invar).
\end{equation}
Similarly, differentiating  \cref{eq:mapeq} in $\zba$ and evaluating at $\zba = \wba = 0$,
\begin{equation}\label{eq:f}
	f = ws + z Q(z,w,0,0).
\end{equation}
Along the first Segre set $\Sigma := \{w = 0\}$, we have
\[
	z-f + i \lba zf^2\bigl|_{\Sigma} = 0.
\]
Solving for $f\bigl|_{\Sigma}$, we have that
\begin{equation}\label{eq:lambdaNeq0FFirstSegre}
	f\bigl|_{\Sigma} = \frac{2z}{1+\sqrt{1-4i\lba z^2}} = z+i \lba  z^3-2 \lba ^2 z^5+O\left(z^6\right),
\end{equation}
where we choose a holomorphic branch of $\sqrt{1-4i\lba z^2}$ that is equal 1 at $z=0$ so that the right hand side is holomorphic at $z=0$.

Applying the $2^{\mathrm{nd}}$-order differential operators $\partial_{\zba}^j \partial_{\wba}^{2-j}$, $j=0,1,2$, to the mapping equation \cref{eq:mapeq} and evaluating at $w = \zba = \wba = 0$, we obtain 3 linear equations of 3 unknowns $r,s$ and $\phi$ along the first Segre set. Precisely, the following system holds along $\Sigma$.  
\begin{equation}\label{eq:rsphi}
\begin{pmatrix}
	0 & 4 i z & 1 \\
	-z & 1 & 0 \\
	i & 0 & i \lba  \\
\end{pmatrix}
\cdot
\begin{pmatrix}
	r\\
	s\\
	\phi
\end{pmatrix}
=
f\begin{pmatrix}
	-\alpha f \\
	\dfrac{i\alpha }{2} - \dfrac{\bar{\mu}}{2} f\\
	2 \bar{\nu}-\bar{\sigma} f
\end{pmatrix}.
\end{equation}
This system is solvable since the coefficient $3 \times 3$-matrix in the left hand side, denoted by $D$, is invertible with the inverse
\[
	D^{-1} = \frac{1}{1 - 4i \bar{\lambda} z^2}\begin{pmatrix}
		-\lba  & -4 i \lba  z & 1 \\
		-z\lba & -1 & z \\
		1 & 4 i z & -4 i z^2 \\
	\end{pmatrix}.
\]
Explicitly,
\begin{align}\label{eq:phiz0}
	\phi\bigl|_{\Sigma} & = \frac{2 z (\alpha +4 \bar{\nu} z)f+\left(4 \bar{\sigma} z^2+2 i \bar{\mu} z -\alpha \right)f^2}{1- 4 i\lba  z^2}\biggl|_{\Sigma},\\
	r\bigl|_{\Sigma} & = \frac{2(i\bar{\nu} - \alpha  \bar{\lambda} z) f +(\alpha  \bar{\lambda} +i\bar{\sigma} -2i \bar{\lambda}  \bar{\mu}   z)f^2}{1-4i \bar{\lambda}  z^2}\biggl|_{\Sigma}, \\
	s\bigl|_{\Sigma} & = \frac{(-\bar{\mu}  +2 \alpha  \bar{\lambda}  z+2 i \bar{\sigma}  z)f^2+i(\alpha +4 \bar{\nu}   z)f}{2(1-4i\bar{\lambda}  z^2)}\biggl|_{\Sigma}.
\end{align}
The proof is complete.
\end{proof}
\begin{lem}\label{lem:3.3} If $H$ is a map given in the normal form \eqref{eq:normalform}, then the following hold
\begin{align}
	g_w(z,0) & = 1 + i\lba f(z,0)^2, \\
	f_w(z,0) & = \frac{\lba  z g_w(z,0) \phi (z,0)+z r(z,0)+s(z,0)}{\sqrt{1-4 i \lba  z^2}}, \\
	g_{ww}(z,0) & = 2 \left(2 i \lba  f(z,0) f_w(z,0)+\lba  g_w(z,0) \phi (z,0)+r(z,0)\right)
\end{align}
In particular, the following components of the third-order derivative of $H$ are expressed in terms of its lower order derivatives
\begin{align}\label{eq:3jetfg}
	f^{(3,0)} = 6i \lba,\quad f^{(2,1)} = -\bar{\mu} + 8i \bar{\nu},
	\quad 
	g^{(3,0)} = 0, \\
	\quad 
	g^{(2,1)} = 2i \lba,\quad
	g^{(1,2)} = 4i \bar{\nu},\quad 
	\phi ^{(3,0)}=6 (8 \bar{\nu}+2 i \bar{\mu}).
\end{align}
\end{lem}
\begin{proof} Differentiating \cref{eq:Qsol} in $w$, setting $w=0$, and substituting $g(z,0) = 0$, we find that 
\[
	g_w(z,0) = Q(z,0,0,0) = 1 + i\lba f(z,0)^2.
\]
Substituting $Q(z,w,0,0)$ from \eqref{eq:Qsol} into \eqref{eq:f} and applying $\partial_w|_{w=0}$ we obtain 
\[
	f_w(z,0) = \frac{\lba  z g_w(z,0) \phi (z,0)+z r(z,0)+s(z,0)}{\sqrt{1-4 i \lba  z^2}}.
\]
Differentiating \eqref{eq:Qsol} in $w$ twice and evaluating along $w = 0$, we have
\[
	g_{ww}(z,0) = 2 \left(2 i \lba  f(z,0) f_w(z,0)+\lba  g_w(z,0) \phi (z,0)+r(z,0)\right)
\] 
Expanding the formula for $f(z,0)$ in Lemma~\ref{lem:first} as Taylor series at $z=0$ we have $f^{(3,0)} = 6i\lba$.
\end{proof}
\begin{lem}\label{lem:cond3} Let $H = (f,\phi,g)$ be a map given in the normal form \cref{eq:normalform}. Then the following holds along $\Sigma$:
\begin{equation}\label{eq:com1}
\Delta:=\det 
\begin{pmatrix}
	-z & 0 & 0 & 2 i \lambda  f+2 i (\mu +4 i \nu ) f^2 \\
	i & -4 z & 0 & 2 \lambda+ 2 (\mu +8 i \nu ) f-\bar{\phi}^{(2,1)} f^2  \\
	0 & 2i & -2z & 4 \nu-2\bar{f}^{(1,2)} f-\bar{\phi}^{(1,2)} f^2-2(i \bar{\mu}+2\bar{\nu}) \phi\\
	0 & 0 & 3 i &i \bar{g}^{(0,3)}-2 \bar{f}^{(0,3)} f-\bar{\phi}^{(0,3)} f^2-6 i \bar{\sigma} \phi
\end{pmatrix}
=0.
\end{equation}
\end{lem}
\begin{proof}
Applying the $3^{\mathrm{rd}}$-order differential operators $\partial_{\zba}^j \partial_{\wba}^{3-j}$, $j=0,1,2,3$, to the mapping equation \cref{eq:mapeq}, evaluating at $w = \zba = \wba = 0$, and using Lemma~\ref{lem:3.3}, we obtain 4 linear equations of the following auxiliary holomorphic functions
\[
p(z,w) := Q_{\zba\zba}(z,w,0,0),\
t(z,w) := Q_{\zba\wba}(z,w,0,0),\
q(z,w) := Q_{\wba\wba}(z,w,0,0)
\]
along the first Segre set. For instance, for $j=0$, we can set $w = \wba = 0$ in the mapping equation and using $g(z,0)=\gba(\zba,0) = 0$, we have
\[
	f(z,0) \bar{f} (\zba ,0)+\frac{1}{2} f(z,0)^2 \bar{\phi} (\zba ,0)+\frac{1}{2} \bar{f} (\zba ,0)^2 \phi (z,0) =\zba  z Q(z,0,\zba ,0).
\]
Differentiating this three times in $\zba$ and setting $\zba = 0$ we have 
\[
	6 i \lambda  f(z,0)+6 (i\mu -4  \nu ) f(z,0)^2+3 z p(z,0) = 0.
\]
Proceeding similarly for $j=1,2,3$, we have three more equations. Omitting the detailed calculation, we conclude that
\begin{equation}\label{eq:zw3}
\begin{pmatrix}
	-z & 0 & 0 \\
	i & -4 z & 0 \\
	0 & 2 i & -2z \\
	0 & 0 & 3 i
\end{pmatrix}
\cdot
\begin{pmatrix}
	p \\
	t \\
	q
\end{pmatrix}
= 
\begin{pmatrix}
	2 i \lambda  f+2 i (\mu +4 i \nu ) f^2 \\
	2 \lambda+ 2 (\mu +8 i \nu ) f-\bar{\phi}^{(2,1)} f^2  \\
	4 \nu-2\bar{f}^{(1,2)} f-\bar{\phi}^{(1,2)} f^2-2(i \bar{\mu}+2\bar{\nu}) \phi  \\
	i \bar{g}^{(0,3)}-2 \bar{f}^{(0,3)} f-\bar{\phi}^{(0,3)} f^2-6 i \bar{\sigma} \phi
\end{pmatrix}.
\end{equation}
Observe that the $3\times 3$-matrix formed by the last 3 rows of the coefficient matrix in the left hand side of \cref{eq:zw3} is invertible. Thus, by the Kroneker-Capelli theorem, \cref{eq:zw3} has unique solution for $p$, $q$, and $t$ along the first Segre set if and only if the determinant of the augmented matrix \cref{eq:com1} vanishes identically along $\Sigma$.
\end{proof}
To reduce computational complexity, we shall divide into two cases, depending on whether $\lambda = 0$ or $\lambda \ne 0$. We first treat the case $\lambda \ne 0$.
\subsection{Case 1: \texorpdfstring{$\lambda \neq 0$}{lambda neq 0}}
The purpose of this section is to compute $H_w$ and $H_{ww}$ along $\Sigma$ for the case $\lambda \ne 0$.
\begin{lem}\label{lem:caselambdanonzero}
	If $H=(f,\phi,g)$ is given as in \cref{eq:normalform} and sends $\heis$ into $\lc$. If $\lambda\ne 0$, then
	\[
	H_w\bigl|_{\Sigma} = \left(0, \frac{2\lambda}{1+\sqrt{1- 4i \lba z^2}}, \frac{2}{1+\sqrt{1-4 i \lba z^2}}\right),
	\]
	and
	\[
	H_{ww}\bigl|_{\Sigma} = \left(\frac{8|\lambda|^2z\sqrt{1-4 i \lba  z^2}}{\left(1-4i \lba  z^2\right) \left(1+\sqrt{1-4 i \lba  z^2}\right)^2}, 0, 0\right).
	\]
\end{lem}
To prove this lemma, we first identity the coefficient $\mu, \nu$, and $\sigma$.
\begin{lem}\label{lem:5} Let $H$ be a map given in the form \cref{eq:normalform}. Assume that \cref{eq:1jetSegre} and \cref{eq:com1} hold. If $\lambda \ne 0$, then
\begin{equation}\label{eq:case1a}
	\mu = \nu = \sigma =0.
\end{equation}	
Moreover, the following equalities between 3rd-order derivatives at the origin hold:
	 \begin{equation}\label{eq:case1b}
	 	\phi^{(0,3)} = 3 \lambda f^{(1,2)}, 
	 	\phi^{(1,2)} = 0,
	 	\phi^{(2,1)} = 2i |\lambda|^2,
	 	g^{(0,3)} = 3f^{(1,2)}.
	 \end{equation}
\end{lem}
\begin{proof} To reduce our computations, we note that for $\lambda \ne 0$, we can write
\[
	f|_{\Sigma} = z \eta(z), \ f^2|_{\Sigma} = \frac{i}{\lba}(1-\eta(z)), \text{ for } \ \eta(z) := 2\left(1+ \sqrt{1- 4i \lba z^2}\right)^{-1},
\] 
where we choose the holomorphic branch of the squared root having value $1$ at $z=0$. Plugging these into the formula for $\phi|_{\Sigma}$ in \cref{eq:phiz0}, we find that
\[
\phi|_{\Sigma}
	=
	u(z) + v(z) \eta(z)
\]
for some rational functions $u(z)$ and $v(z)$ holomorphic at $z=0$. Expanding the determinant in \cref{eq:com1} along the first row, we have
\[
	\Delta
	=
	-z \Delta_{1,1} - (2 i \lambda  f+2 i (\mu +4 i \nu ) f^2)\Delta_{1,4},
\]
where $\Delta_{1,1}$ and $\Delta_{1,4}$ are the corresponding minors. Observe that $\Delta_{1,4} = -6i$ and, by the formulas for $f, f^2$ and $\phi$ along the first Segre set above, $\Delta_{1,1} = C(z) + D(z) \eta(z)$ for some rational functions $C(z)$ and $D(z)$ holomorphic at the origin. Therefore,
\[
	\Delta
	=
	-\frac{12i}{\lba}(\mu +4 i \nu ) - z C(z) + \left(\frac{12i}{\lba}(\mu +4 i \nu ) -12 \lambda z - z D(z)\right) \eta(z).	
\]
Since $\eta(z)$ is irrational when $\lambda \ne 0$, the vanishing of $\Delta$ is equivalent to
\[
	\nu = \frac{i \mu }{4}, \quad C(z) = 0, \quad D(z) = -12 \lambda.
\]
The last identity is expressed explicitly as
\[
\det \begin{pmatrix}
	-4 z & 0 & 2 \lambda-2 \mu  f-\bar{\phi}^{(2,1)}f^2 \\
	2 i & -2 z & i \mu  -2 \bar{f}^{(1,2)} f-\bar{\phi}^{(1,2)} f^2-i \bar{\mu} \phi\\
	0 & 3 i & i \bar{g}^{(0,3)}-2 \bar{f}^{(0,3)} f-\bar{\phi }^{(0,3)}f^2-6 i \bar{\sigma} \phi \\
\end{pmatrix}
= -12 \lambda \eta(z).
\]
Expanding the determinant along the first row, cancelling the right hand side, and dividing resulting equation by $-2$, we find that
\begin{multline}
	2f \left(4 z^2 \bar{f}^{(0,3)}+6 i z \bar{f}^{(1,2)}-3 \mu \right)+f^2 \left(4 z^2 \bar{\phi}^{(0,3)}+6 i z \bar{\phi }^{(1,2)}-3 \bar{\phi}^{(2,1)}\right)\\+6\phi\left(4 i\bar{\sigma} z^2-\bar{\mu} z\right) -4 i z^2 \bar{g}^{(0,3)}  +6 \mu  z +6 \lambda (1-\eta)  = 0.
\end{multline}
Substituting $f$ and $\phi$ along $\Sigma$ from \cref{eq:1jetSegre} and collecting the term $\sqrt{1-4i\bar{\lambda} z^2}$, we can rewrite the equation as
\[
	\frac{A(z) + B(z) \sqrt{1-4 i \lba z^2}}{\left(4 \lba  z^2+i\right) \left(1+\sqrt{1-4 i \lba  z^2}\right)^2} = 0,
\]
where $A$ and $B$ are polynomials in $z$. The vanishing of $\Delta$ implies that $A$ and $B$ must vanish identically. Explicitly,
\begin{align*}
	B(z) 
	= 
	4z^2\{8\left(2\lba  \bar{f}^{(0,3)}+3 i \bar{\mu} \bar{\sigma}\right)z^3+2\left(12 i \lba  \bar{f}^{(1,2)}-4 i \lba  \bar{g}^{(0,3)}-12 \alpha  \bar{\sigma}-3 \bar{\mu }^2\right)z^2 \\
	+2\left(2 i \bar{f}^{(0,3)}-3 i \alpha  \bar{\mu}\right)z + 2 \bar{g}^{(0,3)}-6 \bar{f}^{(1,2)}\}.
\end{align*}
Since $B(z)$ must vanish identically, equating the coefficient of the lowest degree in $z$, that of $z^2$,  yields
\[
	\bar{g}^{(0,3)} = 3 \bar{f}^{(1,2)}.
\]
The vanishing of the coefficient of $z^3$ yields
\[
	\bar{f}^{(0,3)} = 3\frac{\alpha  \bar{\mu}}{2}.
\]
These together with the vanishing of the terms of degree 4 in $z$ yield
\[
	4\alpha  \bar{\sigma}+\bar{\mu }^2 = 0.
\]
Combining these 3 with the vanishing of the degree 5 term yields
\[
	\bar{\mu }(\alpha\lba + i \bar{\sigma}) = 0.
\]
On the other hand, we write
\begin{align*}
	A(z)
	=
	4z^2\{16 z^4 \left(6 \bar{\sigma} ^2+\bar{\phi} ^{(0,3)}\lba -\lba^2  \bar{g}^{(0,3)}\right)+8 z^3 \left(2 \lba  \bar{f} ^{(0,3)}+3 i \lba  \bar{\phi} ^{(1,2)}-3 i \lba ^2 \mu -6 i \bar{\mu} \bar{\sigma} \right)\\
	+2 z^2 \left(12 i \lba  \bar{f} ^{(1,2)}-6 i \lba  \bar{g}^{(0,3)}-6 \lba  \bar{\phi} ^{(2,1)}+2 i \bar{\phi} ^{(0,3)}-12 i \lambda  \lba ^2+3 \bar{\mu}^2\right)\\
	+z \left(4 i \bar{f} ^{(0,3)}-6 \bar{\phi} ^{(1,2)}+6 \lba  \mu \right)-6 \bar{f} ^{(1,2)}+2 \bar{g}^{(0,3)}-3 i \bar{\phi} ^{(2,1)}+6 |\lambda|^2\}. 
\end{align*}
Since $A$ must vanish identically, equating the coefficient of $z^2$ to zero and using the relation $\bar{g}^{(0,3)} = 3 \bar{f}^{(1,2)}$ obtained above, we find that
\[
	\bar{\phi}^{(2,1)}= -2 i |\lambda|^2.
\] 
Putting these together we obtain
\begin{align*}
	A(z) = 4z^2\{16 z^4 \left(6 \bar{\sigma} ^2+\lba(\bar{\phi} ^{(0,3)} -3\lba  \bar{f}^{(1,2)})\right)+ 8 z^3 \left(3 i \lba  \bar{\phi} ^{(1,2)}-3 i \lba ^2 \mu -6 i \bar{\mu} \bar{\sigma} + 3 \alpha  \lba \bar{\mu}\right)\\
	+2 z^2 \left(2 i (\bar{\phi} ^{(0,3)} - 3\lba\bar{f} ^{(1,2)})+3 \bar{\mu}^2\right)+ 6z \left(i \alpha  \bar{\mu}+ \lba  \mu - \bar{\phi} ^{(1,2)}\right)\}.
\end{align*}
Equating the coeficient of $z^3$ to zero we find that
\[
	\bar{\phi}^{(1,2)}= \bar{\lambda} \mu +i \alpha \bar{\mu}.
\]
Plugging this back to the formula for $A(z)$, we find that the coefficient of $z^5$ is a multiple of $\bar{\sigma}\bar{\mu}$. Equating this to be zero, we find that either $\bar{\sigma} = 0$ of $\bar{\mu}=0$. But in either case, the vanishing of the coefficients of $z^6$ and $z^4$ in $A(z)$ together imply that $\bar{\sigma} = \bar{\mu} = 0$. Moreover,
\[
	\bar{\phi} ^{(0,3)} = 3\lba\bar{f} ^{(1,2)}.
\]
That $\bar{\phi}^{(1,2)} = 0$ is obvious since $\mu = 0$.
\end{proof}
\begin{lem}\label{lem:ptqsol2} If $H$ is a map given by \cref{eq:normalform}. If $\lambda\ne0$, then the auxiliary functions $p,t,q$ along the first Segre set $\Sigma$ are given by
	\[
	p(z,0) = 2i \lambda \eta(z), \ t(z,0) = 0, \ q(z,0)= \bar{f}^{(1,2)}\eta(z).
	\]
\end{lem}
\begin{proof} By \cref{lem:5}, the conditions \cref{eq:case1a} and \cref{eq:case1b}. Plugging these into the system \cref{eq:zw3} we have
\[
\begin{pmatrix}
	i & -4 z & 0 \\
	0 & 2 i & -2z \\
	0 & 0 & 3 i
\end{pmatrix}
\cdot
\begin{pmatrix}
	p \\
	t \\
	q
\end{pmatrix}
= 
\begin{pmatrix}
	2 \lambda+2i |\lambda|^2 f^2  \\
	-2\bar{f}^{(1,2)} f\\
	3i \bar{f}^{(1,2)}-3\lba\bar{f}^{(1,2)} f^2
\end{pmatrix}
=
\eta(z)
\begin{pmatrix}
	2 \lambda  \\
	-2\bar{f}^{(1,2)} z\\
	3i \bar{f}^{(1,2)}
\end{pmatrix}.
\]
This can be solved easily to obtain the desired formulas for $p,t$, and $q$.
\end{proof}
\begin{lem}
	Assume that $\lambda\ne 0$. Then $\alpha = 0$ and
	\begin{align*}
		\phi ^{(3,1)} =\phi ^{(2,2)} = f^{(0,4)}= 0,\ g^{(0,4)} =  4 f^{(1,3)},\\ g^{(1,3)} =  \frac{3}{2} f^{(2,2)}, \phi ^{(1,3)}= \frac{3}{2}\lba  f^{(2,2)},\phi ^{(0,4)}= 4\lba  f^{(1,3)}.
	\end{align*}
\end{lem}
\begin{proof} We introduce the following auxiliary holomorphic functions
\begin{align}\label{eq:klmndef}
k(z,w) = Q_{\zba\zba\zba}(z,w,0,0),\
l(z,w) = Q_{\zba\zba\wba}(z,w,0,0), \notag\\
m(z,w) = Q_{\zba\wba\wba}(z,w,0,0),\
n(z,w) = Q_{\wba\wba\wba}(z,w,0,0).
\end{align}
We proceed similarly to the proof of Lemma~\ref{lem:cond3} as follows.
Applying the differential operator $\partial_{\zba}^j\partial_{\wba}^{4-j}$, $j=0,1,2,3,4$, to the mapping equation \cref{eq:mapeq} and evaluating at $w= \zba=\wba = 0$, we obtain an overdetermined system of 5 linear equations of 4 unknowns that must be satisfied by $k, l, m, n$ restricted to the first Segre set $\Sigma = \{w=0\}$. Explicitly, the following holds when $w=0$:
\begin{multline}
\begin{pmatrix}
	-z & 0 & 0 & 0 \\
	i & -6 z & 0 & 0 \\
	0 & 2 i & -4 z & 0 \\
	0 & 0 & 3 i & -2 z \\
	0 & 0 & 0 & 4 i \\
\end{pmatrix}
\cdot
\begin{pmatrix}
	k\\
	l\\
	m\\
	n
\end{pmatrix}
\\
= \begin{pmatrix}
	6 i \lambda (2\alpha f^2 + \phi) \\
	30 \alpha  \lambda  f-\bar{\phi}^{(3,1)} f^2 \\
	-2 \bar{f}^{(2,2)} f-\bar{\phi}^{(2,2)} f^2-4 \bar{f}^{(1,2)} \phi +\alpha ^2 \phi-4 |\lambda|^2\phi-2 i \bar{\phi}^{(2,1)}\phi-4 i \alpha  \lambda  \\
	-2 \bar{f}^{(1,3)}f-\bar{\phi}^{(1,3)} f^2-2 \bar{f}^{(0,3)}\phi+i \bar{g}^{(1,3)}-3 i \bar{\phi}^{(1,2)} \phi \\
	-2 \bar{f}^{(0,4)} f-\bar{\phi}^{(0,4)} f^2-4 i \phi\left(\bar{\lambda} \bar{g}^{(0,3)}+\bar{\phi}^{(0,3)}\right)+i \bar{g}^{(0,4)} \\
\end{pmatrix}
\end{multline}
The $4\times 4$-matrix formed by the last 4 rows of the coefficient matrix is invertible (its determinant is 24). Thus, by Kroneker-Capelli theorem again, the solvability of this system is equivalent to the vanishing (identically along $\Sigma$) of the determinant of the augmented $5 \times 5$-matrix which we denoted by $\Delta$.

Since $\mu = \nu = \sigma = 0$, we have from \cref{e:phiz0}
\begin{equation}\label{e:phiz0}
	\phi\bigl|_{\Sigma} = \alpha\frac{2 zf - f^2}{1- 4 i\lba  z^2}\biggl|_{\Sigma}.
\end{equation}
Therefore, one can write
\[
	6 i \lambda (2\alpha f^2 + \phi)
	=
	-\frac{6 \alpha  \lambda  \left(8 \lba  z^2+i\right)}{\lba  \left(4 \lba  z^2+i\right)}
	+ E(z) \eta(z),
\]
where $E(z)$ is a rational function of $z$ holomorphic at $z=0$. Expanding the determinant $\Delta$ of the $5\times 5$ augmented matrix along the first row, we have
\[
	\Delta
	= -z \Delta_{1,1} + (6 i \lambda (2\alpha f^2 + \phi)) \Delta_{1,5},
\]
where $\Delta_{1,1}$ and $\Delta_{1,5}$ are the corresponding minors.
Arguing as before, we have 
\[
	\Delta_{1,1}
	=
	F(z) + G(z) \eta(z),
\]
where $F(z)$ and $G(z)$ are also rational in $z$ and holomorphic at $z=0$, while $\Delta_{1,5} = 24$. Plugging this back into the formula for $\Delta$, we have
\[
	-z F(z) -24\frac{6 \alpha  \lambda  \left(8 \lba  z^2+i\right)}{\lba  \left(4 \lba  z^2+i\right)} + \eta(z)(24E(z) - z G(z)).
\]
By irrationality of $\eta(z)$, $\Delta$ vanishes identically iff
\[
	-z F(z) -24\frac{6 \alpha  \lambda  \left(8 \lba  z^2+i\right)}{\lba  \left(4 \lba  z^2+i\right)} = 0,\quad
	24E(z) - z G(z) = 0.
\]
Setting $z=0$ in the first equation, we find that $\alpha = 0$, as desired.
\end{proof}
With the corresponding solvability conditions are satisfied, we can solve for $k, l, m$, and $n$ from
\begin{equation}\label{eq:klmn}
		\begin{pmatrix}
			i & -6 z & 0 & 0 \\
			0 & 2 i & -4 z & 0 \\
			0 & 0 & 3 i & -2 z \\
			0 & 0 & 0 & 4 i \\
		\end{pmatrix}
		\cdot
		\begin{pmatrix}
			k\\
			l\\
			m\\
			n
		\end{pmatrix}
		=
		\eta(z)
		\begin{pmatrix}
			0 \\
			-2 \bar{f}^{(2,2)} z\\
			-2 \bar{f}^{(1,3)}z-\dfrac{3i}{2} \bar{f}^{(2,2)} \\
			4i\bar{f}^{(1,3)} 
		\end{pmatrix}
\end{equation}
to obtain the following identity along $\Sigma$:
\begin{align}\label{eq:klmnsol}
	k(z,0) = m(z,0) = 0,\quad 
	m(z,0) = \frac{1}{2} \bar{f} ^{(2,2)} \eta (z), \quad
	n(z,0) = \bar{f} ^{(1,3)} \eta (z).
\end{align}

\begin{proof}[Proof of Lemma \ref{lem:caselambdanonzero}]
Differentiating the mapping equation with respect to $\zba$ and $\wba$ and combining the two equations, we obtain
\begin{equation}\label{e21}
	z+ i \lba  z f(z,w)^2-f(z,w)+\lba  z g(z,w) \phi (z,w)+w z r(z,w)+w s(z,w)=0.
\end{equation}
Differentiating this with respect to $w$ and setting $w=0$ yields
\[
(2 i \lba  z f-1)f_w +z r+s\bigl|_{\Sigma} = 0.
\]
Plugging $\alpha = \sigma = \mu =\nu = 0$ into \cref{eq:rsphi}, we easily find that $r=s=0$ along $w=0$. Thus
\[
f_w\bigl|_{\Sigma} = 0.
\]

Applying $\partial_{\zba}^j \partial_{\wba}^{2-j} \partial_w$, $j=0,1,2$, to the mapping equation and evaluating along $\Sigma$, we obtain
\[
	\begin{pmatrix}
		0 & 4 i z & 1 \\
		-z & 1 & 0 \\
		i & 0 & i \lba  \\
	\end{pmatrix}
	\cdot
	\begin{pmatrix}
		r_w\\
		s_w\\
		\phi_w
	\end{pmatrix}
	=
	\frac{i}{2}
	\begin{pmatrix}
		p\\
		t\\
		q
	\end{pmatrix}.
\]
Plugging the formulas for $p,t$, and $q$ along $\Sigma$ from \cref{lem:ptqsol2}, we find that
\[
	\phi_w\bigl|_{\Sigma}
	=
	\frac{\eta(z)  \left(\lambda-2i z^2 \bar{f} ^{(1,2)} \right)}{1-4 i\lba  z^2}
\]
We claim that
	\[f^{(1,2)}= 2| \lambda|^{2}.\]
To show this, we differentiate \cref{e21} twice in $w$ and setting $w=0$, using $\phi=0$, $f_w = 0$, $g_w=1+\lba f^2$ along $\Sigma$ to obtain
\[
	f_{ww}\bigl|_{\Sigma}
	=
	-\frac{4 z \left(2 |\lambda|^2 -\bar{f} ^{(1,2)} \left(1+\sqrt{1-4 i \lba  z^2}\right)\right)}{\left(1-4i \lba  z^2\right) \left(1+\sqrt{1-4 i \lba  z^2}\right)^2}.
\]
Differentiating this equation in $z$ and evaluating at $z=0$ we find that
\[
	f^{(1,2)} = -2 |\lambda|^2+2\bar{f} ^{(1,2)}.
\]
Hence $f^{(1,2)}=2 |\lambda|^2$ and the claim follows. Thus
\[
	f_{ww}\bigl|_{\Sigma}
	=
	\frac{8|\lambda|^2z\sqrt{1-4 i \lba  z^2}}{\left(1-4i \lba  z^2\right) \left(1+\sqrt{1-4 i \lba  z^2}\right)^2}.
\]

In the next step, we solve for $p_w, t_w$ and $q_w$ along the first Segre set. To this end, we apply $\partial_{\zba}^j \partial_{\wba}^{3-j} \partial_w$, $j=0,1,2,3$, to the mapping equation \cref{eq:mapeq} and evaluating along $\zba = \wba = w =0$ to obtain
\[
\begin{pmatrix}
	-6 z & 0 & 0 \\
	i & -4 z & 0 \\
	0 & 2 i & -2 z \\
	0 & 0 & 3 i \\
\end{pmatrix}
\cdot
	\begin{pmatrix}
	p_w \\
	t_w \\
	q_w
\end{pmatrix}
=
i \begin{pmatrix}
	k \\
	l \\
	m \\
	n
\end{pmatrix}
\]
when $w=0$. The solvability of this overdetermined system, via Kronecker-Capelli theorem, gives
\[
	k(z,0)-6 i z l(z,0)-12 z^2 m(z,0)+8 i z^3 n(z,0) = 0,
\]
which, from the formula \cref{eq:klmnsol} for $k$, $l$, $m$, and $l$ along $\Sigma$, is equivalent to
\[
	12 i z \left(4 z \bar{f}^{(1,3)}+3 i \bar{f}^{(2,2)}\right)f\bigl|_{\Sigma} = 0,
\]
or equivalently,
\[
	 \bar{f}^{(1,3)} =0,\
	  \bar{f}^{(2,2)} = 0.
\]
Hence, from \cref{eq:klmnsol}, we find that $k=l=m=n=0$ along $\Sigma$ and consequently $p_w = t_w = q_w=0$ along $\Sigma$.

Finally, applying $\partial^j_{\zba}\partial^{2-j}_{\wba}\partial_{w}^2$, $j=0,1,2$, and evaluating along $w = \zba = \wba$, we have
\[
\begin{pmatrix}
	0 & 4 i z & 1 \\
	-z & 1 & 0 \\
	i & 0 & i \lba  \\
\end{pmatrix}
\cdot
\begin{pmatrix}
	r_{ww}\\
	s_{ww}\\
	\phi_{ww}
\end{pmatrix}
=
\begin{pmatrix}
	0\\
	0\\
	0
\end{pmatrix}
\]
and we find that 
\[
	\phi_{ww}\bigl|_{\Sigma} = 0.
\]
The proof is complete.
\end{proof}

Denote by $L$ the global $(1,0)$-vector field given by
\[
L = \frac{\partial}{\partial z} +2i \zba  \frac{\partial}{\partial w},
\]
and $\overline{L}$ its conjugate. Observe that $L \vr = 0$ on $\mathbb{C}^2$ (not just along $\heis$.)

We sum up
\begin{lem}
If $\lambda \neq 0$, then 
\begin{align*}
H(z,w) = (1 + i z^2 + w^2)(z,w,w) + O(5),
\end{align*}
and 
\begin{align*}
H(z,0)  & = \left(\frac{2 z}{1+\sqrt{1-4 i z^2}},0,0 \right),\\
H_w(z,0) & = \left(0, \frac{2}{1+\sqrt{1-4 i z^2}}, \frac{2}{1+\sqrt{1-4 i z^2}}\right),\\
H_{ww}(z,0) & = \left(\frac{8 z \sqrt{1-4 i z^2}}{(1-4 i z^2)(1+\sqrt{1-4 i z^2})^2},0,0\right). 
\end{align*}
\end{lem}

In the next step we want to determine the map along the second Segre set, parametrized by $(z,\bar z)\mapsto (z, 2 i z \bar z)$. To this end we evaluate the mapping equation and its first and second CR derivative at $\bar w = 0$. More precisely, we consider the system 
\begin{align*}
L^k\rho'(H(z,\bar w + 2 i z \bar z),\overline{H(z,w)})\bigl|_{\bar w=0}=0, \quad k\in\{0,1,2\}.
\end{align*}

Using the fact that $\phi(z,0)=g(z,0) = f_w(z,0) = \phi_{ww}(z,0) = g_{ww}(z,0) = 0$, the above system reduces to the following, where we skip evaluation along $\bar \Sigma$ in $\bar H$ and its derivatives and along the second Segre set $T=\{(z,2 i z \bar z): z,\bar z \in \CC\}$ in $H$:
\begin{align}
\label{eq:lambdaNeq0System1}
& 2 i \bar f f + i \bar f^2 \phi - g = 0,\\
\label{eq:lambdaNeq0System2}
& i z \bar \phi_{\bar w} f^2 -\bar f_{\bar z} f - (\bar f \bar f_{\bar z} - z \bar \phi_{\bar w} g) \phi + z \bar g_{\bar w} = 0,\\
\label{eq:lambdaNeq0System3}
& 2 i z \bar \phi_{\bar z \bar w} f^2 + (4 z^2 \bar f_{\bar w \bar w} - \bar f_{\bar z \bar z}) f +  ((4 z^2 \bar f_{\bar w \bar w} - \bar f_{\bar z \bar z} ) \bar f - \bar f_{\bar z}^2+ 4 i z^2 \bar g_{\bar w} \bar \phi_{\bar w})\phi \\
\nonumber
&  + 2 z \bar \phi_{\bar z \bar w}  \phi g  + 2 z \bar g_{\bar z \bar w} = 0.
\end{align}

Combining \eqref{eq:lambdaNeq0System2} and \eqref{eq:lambdaNeq0System3} and eliminating $f^2$ lead to an equation of the form:
\begin{align}
\label{eq:lambdaNeq0diff}
A + B f + (C + \bar f B) \phi = 0,
\end{align}
where
\begin{align*}
A & = 2 z(\bar g_{\bar z \bar w} \bar \phi_{\bar w} - \bar g_{\bar w}\bar \phi_{\bar z \bar w}),\\
B & = \bar \phi_{\bar w} \left(4 z^2 \bar f_{\bar w \bar w} - \bar f_{\bar z \bar z} \right) + 2 \bar f_{\bar z} \bar \phi_{\bar z \bar w},\\
C & = -\bar f_{\bar z}^2 \bar \phi_{\bar w} + 4 i z^2 \bar g_{\bar w} \bar \phi_{\bar w}^2.
\end{align*}

In can be checked that $A=0$ and $-2 i \bar z(C+\bar f B) = B$, hence \eqref{eq:lambdaNeq0diff} implies  $\phi =  2 i \bar z f$, such that \eqref{eq:lambdaNeq0System1} shows $g = \phi$. The remaining equation is given by
\begin{align*}
i z(1 + 4 i \bar z^2) f^2 - f + z = 0.
\end{align*}
The solution compatible with \eqref{eq:lambdaNeq0FFirstSegre} is
\begin{align*}
f(z, 2 i z \bar z ) = \frac{1-\sqrt{1 - 4 i z^2(1 + 4 i \bar z^2)}}{z(2 i - 8 \bar z^2)},
\end{align*}
setting $\bar z = w/(2 i z)$ gives
\begin{align*}
f(z,w) = \frac{z(1 - \sqrt{1- 4 w^2 - 4 i z^2})}{2 (w^2 + i z^2)}.
\end{align*}

If we rewrite the denominator via
\begin{align*}
2(w^2 + i z^2) = \frac{1}{2}(1-\sqrt{1-4 w^2 - 4 i z^2})(1+\sqrt{1-4 w^2 - 4 i z^2}),
\end{align*}
we obtain
\begin{align*}
f(z,w) = \frac{2 z}{1+\sqrt{1-4 w^2 - 4 i z^2}},
\end{align*}
which gives the desired map $\iota$ in \cref{thm:main1} and finishes the case $\lambda \neq 0$. 

\subsection{Case 2: \texorpdfstring{$\lambda = 0$}{lambda = 0}}
In this case, \cref{lem:first} implies that
\[
H\bigl|_{\Sigma} =  \left(z,\alpha  z^2  + 2z^3 (4 \bar{\nu}+ i \bar{\mu})+4 \bar{\sigma} z^4,0\right).
\]
We shall show that $\mu$, $\nu$, and $\sigma$ must be zero. 
\begin{lem}\label{lem:49} Assume that $\lambda = 0$, then $\sigma = 0$ and the following holomorphic identity holds:
	\begin{align}\label{eq:hol1b}
		4z^3 g  - 4z^2 w f + w^2 z \phi -  w^2 \Upsilon(z,w)\left(w (\mu +4 i \nu )-\alpha  z\right) = 0,
	\end{align}
where $\Upsilon := g\phi + i f^2$.
\end{lem}
\begin{proof}
Setting $w = \wba + 2i z\zba$ in the mapping equation \cref{eq:mapeq}, we rewrite it in the parametrized form:
\begin{align}\label{eq:mappara}
	(1 - \phi(z,\wba + 2i z\zba)\pba (\zba,\wba)) (g(z,\wba + 2i z\zba) - \gba(\zba,\wba)) - 2i f(z,\wba + 2i z\zba) \fba (\zba,\wba) \notag \\
	- i f(z,\wba + 2i z\zba)^2 \pba(\zba,\wba) - i \fba(\zba,\wba)^2 \phi(z,\wba + 2i z\zba) = 0,
\end{align}
satisfies for $z,\zba,\wba$ in a neighbourhood of the origin in $\mathbb{C}^3$ or as identity of formal power series.

Setting $\wba = 0$ and substituting $\fba(\zba,0) = \zba$, $\gba(\zba,0) = 0$, and $\pba(\zba,0) = \alpha\zba^2 + 2\zba^3(4\nu - i \mu) + 4\sigma \zba^4$, we have
\begin{align*}
	g(z, 2i |z|^2) - \Upsilon(z, 2i |z|^2)\left(\alpha\zba^2 + 2\zba^3(4\nu - i \mu) + 4\sigma \zba^4\right) \\
	-2i \zba f(z, 2i |z|^2) - i \zba^2 \phi(z, 2i |z|^2) = 0.
\end{align*}
From this, we obtain a holomorphic equation for $f$, $g$, and $\phi$. Formally, we substitute $\zba = w/(2i z)$, to get
\begin{align*}
	g(z,w) - \Upsilon(z,w) \left(\frac{w^2 \left(\sigma w^2+w z (\mu +4 i \nu )-\alpha  z^2\right)}{4 z^4}\right)\\
	- \left(\frac{w}{z}\right) f(z,w) + \left(\frac{w^2}{4z^2}\right) \phi(z,w) = 0.
\end{align*}
Clearing the denominator $4z^4$ we have
\begin{align}\label{eq:hol1a}
	4z^4 g -  w^2 \Upsilon(z,w)\left(\sigma w^2+w z (\mu +4 i \nu )-\alpha  z^2\right) - 4z^3 w f + w^2 z^2 \phi = 0.
\end{align}
This holds as an identity of germs of holomorphic functions at the origin if $H$ is holomorphic. If $H$ is only assumed to be smooth or formal CR map, then the above holds as formal power series. Indeed, if $G(z,w)$ denotes the left hand side, then the computation above shows that
\[
	G(z, 2i z\zba) = 0
\]
as an identity of formal power series of $z$ and $\zba$. Since the Segre map $(z, \zba) \mapsto (z, 2i z \zba)$ is of generic full rank (in other words, $\heis$ is ``minimal''), we can apply a well-known result (as stated and proved in \cite[Proposition 5.3.5]{baouendi1999real}) to conclude that $G(z,w)$ must be zero in $\mathbb{C}[[z,w]]$.

Setting $z = 0$ in \cref{eq:hol1a}, we obtain
 \[
 	\sigma w^4  \Upsilon(0, w) = 0.
 \]
Thus, either $\sigma = 0$, or $\Upsilon(0, w)=0$. If the latter statement holds, we differentiate the identity three times in $w$ and evaluate at the origin to get
\[
	\sigma = \frac{1}{2}\phi_{ww}  = \frac{1}{2}(g\phi + if^2)_{www}  = \Upsilon_{www} = 0.
\]
Hence in all cases it holds that $\sigma = 0$ and \cref{eq:hol1a} reduces to \cref{eq:hol1b}, as desired. Setting $z = 0$ in the last equation, we see that either $\mu +4 i \nu = 0$, or $\Upsilon(0,w) = 0$.
\end{proof}

\begin{lem}\label{lem:3jetcase1}
	If $\lambda = 0$, then $\sigma = 0$ and
	\begin{align}
	\phi ^{(0,3)} & = 6(6\mu  \nu -i\mu ^2+8i \nu ^2),\\
	\phi ^{(1,2)} & = -\frac{4}{3} i f^{(0,3)}+2 \alpha(i\mu -2\nu) ,\\
	\phi ^{(2,1)} & = \frac{4}{3} i g^{(0,3)}-4i f^{(1,2)},\\
	\phi ^{(3,0)} & =12 (4 \bar{\nu}+i \bar{\mu}).
	\end{align}
\end{lem}
\begin{proof}
By Lemma \ref{lem:cond3} and the assumption $\lambda=0$ we have
\begin{equation}\label{eq:com1a}
	\Delta:=\det 
	\begin{pmatrix}
		-z & 0 & 0 & 2 i (\mu +4 i \nu ) f^2 \\
		i & -4 z & 0 & 2 (\mu +8 i \nu ) f-\bar{\phi}^{(2,1)} f^2  \\
		0 & 2i & -2z & 4 \nu-2\bar{f}^{(1,2)} f-\bar{\phi}^{(1,2)} f^2-2(i \bar{\mu}+2\bar{\nu}) \phi\\
		0 & 0 & 3 i &i \bar{g}^{(0,3)}-2 \bar{f}^{(0,3)} f-\bar{\phi}^{(0,3)} f^2-6 i \bar{\sigma} \phi
	\end{pmatrix}
	=0,
\end{equation}
along $\Sigma$. On the other hand, when $\lambda = 0$ we have
\[
H\bigl|_{\Sigma} =  \left(z,\alpha  z^2  + 2z^3 (4 \bar{\nu}+ i \bar{\mu})+4\bar{\sigma} z^4,0\right).
\]
If $\sigma\ne 0$, then $\phi \bigl|_{\Sigma} $ has degree exactly equal 4. Observe that $f\bigl|_{\Sigma}  = z$ and hence the last column in the matrix on the right of \cref{eq:com1a} has degree at most 4 in $z$. It is immediately that $\Delta$ has degree exactly 7. In fact, expanding the determinant dropping all terms of degree less than 7 we find that the degree 7 term in $\Delta$ is $192i\bar{\sigma}^2 z^7$. Thus, the vanishing of $\Delta$ along $\Sigma$ implies that $\sigma = 0$. 

With $\sigma = 0$, substituting $f(z,0) = z$ and $\phi(z,0) = \alpha  z^2  + 2z^3 (4 \bar{\nu}+ i \bar{\mu})$ into \cref{eq:com1a} and equating coefficients of $z$ to be zero, we can easily find the formulas as above.
\end{proof}

\begin{lem}\label{lem:310}
	If $\lambda = 0$, then $\mu = -3i \nu$ or $\mu = -4i \nu$.
\end{lem}
\begin{proof}
Applying the differential operator $\partial_{\zba}^4$ to the mapping equation \cref{eq:mapeq}, evaluating at $\zba = \wba = 0$, and using $ \phi^{(4,0)} = 96\bar{\sigma}= 0$, we find that $k(z,w)=0$, where $k$ is the auxiliary function defined in \cref{eq:klmndef}. Next, applying the differential operator $\partial_{\wba}^{j}\partial_{\zba}^{4-j}$, $j=1,2,3,4$, to the mapping equation \cref{eq:mapeq} and evaluating at $w= \zba=\wba = 0$, we obtain an overdetermined system of 4 linear equations of 3 unknowns that must be satisfied by $l, m, n$ restricted to the first Segre set $\Sigma = \{w=0\}$. Explicitly, the Kroneker-Capelli theorem implies that, when $w=0$,
\[
\det
\begin{pmatrix}
	-6 z & 0 & 0 & -z^2\bar{\phi} ^{(3,1)}-6 \mu  \phi \\
	2 i & -4 z & 0 & (\alpha ^2-2 i \bar{\phi} ^{(2,1)}-4 \bar{f} ^{(1,2)})\phi-2 z \bar{f} ^{(2,2)}-z^2\bar{\phi} ^{(2,2)} \\
	0 & 3 i & -2 z & (6i \alpha  \bar{\nu}-3i \bar{\phi} ^{(1,2)}-2\bar{f} ^{(0,3)}) \phi -2z \bar{f} ^{(1,3)}+i\bar{g} ^{(1,3)}-z^2 \bar{\phi} ^{(1,3)} \\
	0 & 0 & 4 i & -4\phi(6\bar{\nu}^2-i \bar{\phi} ^{(0,3)}) -2 z \bar{f} ^{(0,4)}+i \bar{g} ^{(0,4)}-z^2 \bar{\phi} ^{(0,4)}\\
\end{pmatrix}
	=0.
\]
Denote this determinant by $\Delta$.
Observe that the entries in the last column of the matrix are polynomials in $z$ of degree at most 3 and hence $\Delta$ has degree at most 6. Moreover, all the terms in the determinant expansion for $\Delta$ has degree at most 5 except the product of the diagonal entries. Thus, collecting the term of degree 6 in $z$ we have
\[
	\Delta =  384 z^6 (\bar{\mu}-4 i \bar{\nu}) \left(\bar{\phi} ^{(0,3)}-6 i \bar{\nu}^2\right) + \cdots
\]
where the dots represent terms of degree 5 or less in $z$. Thus, the vanishing of $\Delta$ implies that either $\mu = -4i \nu$ or $\bar{\phi} ^{(0,3)}-6 i \bar{\nu}^2 = 0$. In the latter case, we use \cref{lem:3jetcase1} to conclude that $\mu = - 3i\nu$, as desired.
\end{proof}
\begin{lem}\label{lem:311}
	If $\lambda = 0$ and $\mu = -4i\nu$, then $\mu = \nu = 0$.
\end{lem}
\begin{proof}
Since $\mu = -4i \nu$, we have $\phi(z,0) = \alpha z^2$, and $\bar{\phi}^{(0,3)} = 0$, $\bar{\phi}^{(1,2)} = -\frac{4}{3}i \bar{f}^{(0,3)} + 4 \alpha \bar{\nu}$, $\bar{\phi}^{(2,1)} = -\frac{4}{3}i \bar{g}^{(0,3)} + 4i \bar{f}^{(1,2)}$. Thus
\[
\det
\begin{pmatrix}
	-6 z & 0 & 0 & (24 i \alpha \nu -\bar{\phi} ^{(3,1)})z^2 \\
	2 i & -4 z & 0 & (\alpha ^3- \frac{8}{3}\alpha\bar{g}^{(0,3)} + 4 \alpha \bar{f}^{(1,2)}-\bar{\phi} ^{(2,2)})z^2-2 z \bar{f} ^{(2,2)}\\
	0 & 3 i & -2 z & (2\alpha \bar{f} ^{(0,3)}-6i \alpha ^2\bar{\nu}- \bar{\phi} ^{(1,3)})z^2 -2z \bar{f} ^{(1,3)}+i\bar{g} ^{(1,3)} \\
	0 & 0 & 4 i & -(24\alpha\bar{\nu}^2+\bar{\phi} ^{(0,4)}) z^2 -2 z \bar{f} ^{(0,4)}+i \bar{g} ^{(0,4)} \\
\end{pmatrix}
=0.
\]
Expanding the determinant on the left, we obtain a degree 5 polynomial in $z$. Equating the terms of degree 5, we obtain
\[
	\bar{\phi} ^{(0,4)}=- 	24\alpha\bar{\nu}^2.
\]
Similarly, equating coefficients of degree 4 terms, we have
\[
	\bar{\phi} ^{(1,3)}= 2 \alpha  \bar{f} ^{(0,3)}+i \left(\bar{f} ^{(0,4)}-6 \alpha ^2 \bar{\nu}\right).
\]
Equating the terms of degree 3 yields
\[
	\bar{\phi} ^{(2,2)}= 4 \alpha  \bar{f} ^{(1,2)}-\frac{2}{3} i \left(\bar{g} ^{(0,4)}-4 \bar{f} ^{(1,3)}\right)-\frac{8}{3} \alpha  \bar{g} ^{(0,3)}+\alpha ^3.
\]
Equating the terms of second degree
\[
	\bar{\phi} ^{(3,1)}= 6 i \bar{f} ^{(2,2)}-4 i \bar{g} ^{(1,3)}-6 \alpha  \mu.
\]
With these conditions satisfied, the system is solvable with unique solution which can be found by solving the system
\[
	\begin{pmatrix}
		2 i & -4 z & 0 \\
		0 & 3 i & -2 z \\
		0 & 0 & 4 i \\
	\end{pmatrix}
	\cdot
	\begin{pmatrix}
		l \\
		m \\
		n \\
	\end{pmatrix}
	=
	\begin{pmatrix}
	-2 z \bar{f} ^{(2,2)}+\dfrac{2}{3} i z^2 \left(\bar{g} ^{(0,4)}-4 \bar{f} ^{(1,3)}\right) \\
	-i z^2 \bar{f} ^{(0,4)}-2 z \bar{f} ^{(1,3)}+i \bar{g} ^{(1,3)} \\
	-2 z \bar{f} ^{(0,4)}+i \bar{g} ^{(0,4)}
	\end{pmatrix}
\]
We observe that, when restricted to $\Sigma$, $n$ is linear, $m$ is at most quadratic, and $l$ is at most cubic in $z$ (Actually, the unique solution $l, m$, and $n$ along $\Sigma$ are linear in $z$.)

Applying $\partial_w\partial_{\zba}^3$ to the mapping equation and evaluating at $w = \zba=\wba = 0$, using $k(z,0)=0$ and $\mu+4i \nu=0$,   we find that 
\[
	p_w\bigl|_{\Sigma} = 0.
\]
Next, applying $\partial_w\partial_{\zba}^j\partial_{\wba}^{3-j}$, $j=0,1,2$, to the mapping equation \cref{eq:mapeq} and evaluating at $w = \zba=\wba = 0$ we find that

\[
	\begin{pmatrix}
		 -4 z & 0 \\
		  2 i & -2 z \\
		  0 & 3 i \\
	\end{pmatrix}
	\cdot
	\begin{pmatrix}
		t_w \\
		q_w
	\end{pmatrix}
	=
	\begin{pmatrix}
			i l-4 i \bar{\nu} z^3 \bar{\phi} ^{(2,1)}-16 |\nu|^2 z^2-4 \alpha  \nu  z \\
			-4 i \bar{\nu} z^2 \bar{f} ^{(1,2)}-i \alpha  z \bar{f} ^{(1,2)}+i m-4 i \bar{\nu} z^3 \bar{\phi} ^{(1,2)}+4 \bar{\nu} \phi_w \\
			-5 i \bar{\nu} z^2 \bar{f} ^{(0,3)}-i \alpha  z \bar{f} ^{(0,3)}+i n+18 \bar{\nu}^3 z^3 
	\end{pmatrix}.
\]
Again, by the Kroneker-Capelli theorem, this overdetermined system is solvable iff
\[
	\Delta:=
	\det \begin{pmatrix}
		-4 z & 0 & i l-4 i \bar{\nu} z^3 \bar{\phi} ^{(2,1)}-16 |\nu|^2 z^2-4 \alpha  \nu  z\\
		2 i & -2 z & -4 i \bar{\nu} z^2 \bar{f} ^{(1,2)}-i \alpha  z \bar{f} ^{(1,2)}+i m-4 i \bar{\nu} z^3 \bar{\phi} ^{(1,2)}+4 \bar{\nu} \phi_w \\
		0 & 3 i & -5 i \bar{\nu} z^2 \bar{f} ^{(0,3)}-i \alpha  z \bar{f} ^{(0,3)}+i n+18 \bar{\nu}^3 z^3 
	\end{pmatrix}
	=0.
\]
On the other hand, applying $\partial_w\partial_{\zba}^j\partial_{\wba}^{2-j}$, $j=0,1,2$, and evaluating at $w = \zba=\wba = 0$, we obtain a system of 3 linear equations for $r_w, s_w$, and $\phi_w$ similarly to the first case. In fact,
\[
\phi_w\bigl|_{\Sigma} = -16 i \bar{\nu}^2 z^4 + \cdots 
\]
where the dots are the terms which are at most cubic in $z$. 
Plugging this into the determinant above and expanding it, the resulting is a polynomial of degree at most 5. Collecting the term of degree 5 in $z$ from the term $18 \bar{\nu}^3 z^3$ of the $(3,3)$-entry and the term $4\bar{\nu}\phi_w$ in the $(2,3)$-entry,  we have
\[
	\Delta = -96 \bar{\nu}^3 z^5 +\cdots
\]
where the dots represent terms of degree 4 or less in $z$. Thus, the vanishing of $\Delta$ implies that $\nu = 0$, as desired.
\end{proof}
Thus, from \cref{lem:310,lem:311} we have $\mu = -3i\nu$ and we use this from now on.
\begin{lem}
	If $\lambda = 0$, then $\mu = -3i \nu$. Moreover,
	\begin{align}
		f_{w}\bigl|_{\Sigma}
		&=
		\frac{5}{2} i \bar{\nu} z^2+\frac{i \alpha  z}{2},\\
		\phi_{w}\bigl|_{\Sigma}
		&=\frac{1}{3} i z^2 \left(2 \bar{f} ^{(1,2)}-4f ^{(1,2)}+ \alpha ^2\right)+\frac{1}{3} i z^3 \left(18 \alpha  \bar{\nu}-4 i \bar{f} ^{(0,3)}\right)+6 i \bar{\nu}^2 z^4-3 i \nu  z,\\
		g_{w}\bigl|_{\Sigma} &= 1,
	\end{align}
and
	\begin{align}
	f_{ww}\bigl|_{\Sigma}
	& =
	2 \nu+f^{(1,2)}z+z^2 \left(\frac{4}{3} i \bar{f}^{(0,3)}-\frac{13 \alpha  \bar{\nu}}{2}\right) -15 \bar{\nu}^2 z^3,\\
	g_{ww}\bigl|_{\Sigma}
	& =
	4i\bar{\nu} z,\\
	g_{www}\bigl|_{\Sigma}
	& = \bar{g} ^{(0,3)}+ 2 z \left(i \bar{f} ^{(0,3)}-3 \alpha  \bar{\nu}\right)-36 \bar{\nu}^2 z^2.
\end{align}
\end{lem}
\begin{proof} Since $g(z,w)=w (w r(z,w)+1)$, we have $g_w\bigl|_{\Sigma}  = 1$. From \cref{eq:rsphi},
\[
	r(z,0) = 2i\bar{\nu} z, \quad s(z,0) = \frac{i}{2} \alpha z + \frac{1}{2} i \bar{\nu} z^2.
\]
Therefore,
\[
	f_w\bigl|_{\Sigma} 
	=
	zr + s\bigl|_{\Sigma} 
	= \frac{5}{2} i \bar{\nu} z^2+\frac{i \alpha  z}{2},
\]
as desired.

Next, we solve the following system for $p$, $t$ and $q$ along $\Sigma$:
\begin{align*}\label{eq:zw3b}
\begin{pmatrix}
	i & -4 z & 0 \\
	0 & 2 i & -2z \\
	0 & 0 & 3 i
\end{pmatrix}
\cdot
\begin{pmatrix}
	p \\
	t \\
	q
\end{pmatrix}
& = 
\begin{pmatrix}
	10 i \nu z + \dfrac{4}{3}i(3 \bar{f}^{(1,2)} - \bar{g}^{(0,3)}) z^2  \\
	4 \nu-2\bar{f}^{(1,2)} z- \dfrac{4i}{3}\bar{f}^{(0,3)} z^2 + 4 \bar{\nu}^2 z^3  \\
	i \bar{g}^{(0,3)}-2 \bar{f}^{(0,3)} z-6i\bar{\nu}^2 z^2
\end{pmatrix}
\end{align*}
We find that 
\begin{align*}
\begin{pmatrix}
p \\
t \\
q
\end{pmatrix}
& =
-\frac{1}{6}\begin{pmatrix}
	6 i & 12 z & -8 i z^2 \\
	0 & 3 i & 2 z \\
	0 & 0 & 2 i \\
\end{pmatrix}
\cdot
	\begin{pmatrix}
	10i \nu z + \dfrac{4}{3}i(3 \bar{f}^{(1,2)} - \bar{g}^{(0,3)}) z^2  \\
	4 \nu-2\bar{f}^{(1,2)} z- \dfrac{4i}{3}\bar{f}^{(0,3)} z^2 + 4 \bar{\nu}^2 z^3  \\
	i \bar{g}^{(0,3)}-2 \bar{f}^{(0,3)} z-6i\bar{\nu}^2 z^2
\end{pmatrix} \\
& =
\begin{pmatrix}
	2 \nu  z \\ 
	\dfrac{1}{3} i z \left(3\bar{f} ^{(1,2)}+\bar{g} ^{(0,3)} \right)-2 i \nu \\ -2 \bar{\nu}^2 z^2 +  \dfrac{1}{3} \bar{g} ^{(0,3)}+\dfrac{2}{3} i z \bar{f} ^{(0,3)}
\end{pmatrix}.
\end{align*}

Applying the differential operators $\partial_{\zba}^j \partial_{\wba}^{2-j} \partial_w$, $j=0,1,2$, to the mapping equation \eqref{eq:mapeq} and evaluating along $\Sigma$, we obtain
\[
\begin{pmatrix}
	0 & 4 i z & 1 \\
	-z & 1 & 0 \\
	i & 0 & 0  \\
\end{pmatrix}
\cdot
\begin{pmatrix}
	r_w\\
	s_w\\
	\phi_w
\end{pmatrix}
=
\frac{i}{2}
\begin{pmatrix}
	p\\
	t\\
	q
\end{pmatrix}
	+
\begin{pmatrix}
	-2 \alpha  z f_w+i \alpha \phi  g_w  \\
	\dfrac{1}{2} i \alpha  f_w -3 i \bar{\nu} z f_w-\dfrac{3}{2}\bar{\nu}\phi g_w\\
	-2 \bar{\nu} f_w
\end{pmatrix}.
\]
The coefficent matrix on the left, denoted by $D_3$, is invertible with the inverse
\[
	D_3^{-1} = 
	\begin{pmatrix}
		0 & 0 & -i \\
		0 & 1 & -i z \\
		1 & -4 i z & -4 z^2 
	\end{pmatrix}
\]
Using the formula for $p, t, q$ and $f_w$, $g_w$ along $\Sigma$, we find that 
\begin{align}
	r_w\bigl|_{\Sigma}
	& = \frac{1}{6} \bar{g} ^{(0,3)}+\frac{1}{6} z \left(-6 \alpha  \bar{\nu}+2 i \bar{f} ^{(0,3)}\right)-6 \bar{\nu}^2 z^2,\\
	s_w\bigl|_{\Sigma} & = z \left(-\frac{1}{2} \bar{f} ^{(1,2)}+\frac{1}{3} \bar{g} ^{(0,3)}-\frac{\alpha ^2}{4}\right)+z^2 \left(-\frac{9 \alpha  \bar{\nu}}{4}+\frac{1}{3} i \bar{f} ^{(0,3)}\right)+\nu -\frac{3}{2} \bar{\nu}^2 z^3, \\
	\phi_w\bigl|_{\Sigma} & = \frac{1}{3} i z^2 \left(6 \bar{f} ^{(1,2)}-4 \bar{g} ^{(0,3)}+3 \alpha ^2\right)+\frac{1}{3} i z^3 \left(18 \alpha  \bar{\nu}-4 i \bar{f} ^{(0,3)}\right)+6 i \bar{\nu}^2 z^4-3 i \nu  z
\end{align}
From the formula for $\phi_w$, we can compute
\[
	\phi^{(2,1)} 
	=
	\frac{2}{3} i \left(6 \bar{f} ^{(1,2)}-4 \bar{g} ^{(0,3)}+3 \alpha ^2\right).
\]
This and \cref{lem:3jetcase1} imply the following ``reflection identity''
\[
	6 f^{(1,2)}+6 \bar{f} ^{(1,2)}-2 g^{(0,3)}-4 \bar{g} ^{(0,3)}+3 \alpha ^2 = 0,
\]
which, in turn, shows that $g^{(0,3)}$ is real and
\[
	g^{(0,3)} = \bar{g}^{(0,3)}= f^{(1,2)}+\bar{f}^{(1,2)}+\frac{\alpha ^2}{2}.
\]
Plugging this back into the formula for $\phi_{w}$
\[
	\phi_w\bigl|_{\Sigma}  = -3 i \nu  z + \frac{1}{3} i z^2 \left(2 \bar{f} ^{(1,2)}-4f ^{(1,2)}+ \alpha ^2\right)+\frac{1}{3} i z^3 \left(18 \alpha  \bar{\nu}-4 i \bar{f} ^{(0,3)}\right)+6 i \bar{\nu}^2 z^4.
\]

 Since $f = z +  w(zr +s)$, we have
\[
f_w(z,0) = zr_w(z,0) + s_{w}(z,0).
\]
Plugging in the formulas for $r_w$ and $s_w$ along $w=0$, we obtain the desired formula for $f_w(z,0)$. The formulas for $g_w$ and $g_{ww}$ along $w=0$ can be deduced similarly from the identity $g = w(1+wr)$. 
The proof is complete.
\end{proof}
\begin{lem}\label{lem:f03}
	If $\nu = 0$, then $f^{(0,3)} = 0$.
\end{lem}
\begin{proof}
Using the formula for $f_w, \phi_w$, and $g_w$ along $\Sigma$, we can produce another holomorphic equation for $f, \phi$, and $g$ by differentiating the mapping equation along the CR vector field. Indeed, we have
\[
	L\left(\vr (H(z,\wba + 2i |z|^2), \overline{H}(\zba,\wba))\right) = 0.
\]
Substituting $\wba = 0$ and using $\fba(\zba,0) = \zba, \pba(\zba,0) = \alpha \zba^2 + 2\nu\zba^3, \bar{g}(\zba,0) = 0$, we find that 
\begin{align}\label{eq:hol2temp}
	0=2i z+ \Upsilon(z,2i|z|^2)\left(2 i z \bar{\phi}_{\wba}(\zba ,0)	-\bar{\phi}_{\zba}(\zba ,0)\right)-f(z,2 i |z|^2) \left(4 z \fba_{\wba}(\zba,0)+2 i\right) \notag \\
	- \phi (z,2 i |z|^2) \left(4 |z|^2 \fba_{\wba}(\zba,0)+2 i z \bar{\phi} (\zba ,0)+2 i \zba \right),
\end{align}
where, as before, $\Upsilon: = \phi g + i f^2$.  Substituting $\zba = w/(2i z)$, we obtain an equation of the form
\[
	0 = 2i z + \Upsilon(z,w) +\cdots 
\]
Since $\nu = 0$, we have $\pba_{\wba}(\zba,0)=\frac43 f^{(0,3)} \zba^3 + $ lower order terms. Clearing the denominator $z^2$, we obtain a holomorphic functional equation for $f(z,w), g(z,w)$, and $\phi(z,w)$. 
Setting $z=0$, we obtain $0 = w^3 f^{(0,3)}\Upsilon(0,w)$. Thus, 
\begin{equation}\label{eq:t1}
	f^{(0,3)}\Upsilon(0,w)=0.
\end{equation}
On the other hand, with $\nu=0$, setting $w=0$ in the first holomorphic functional equation \cref{eq:hol1b} and solving for $g(0,w)$, we find that
\[
\phi (0,w) =  \frac{ \alpha  f(0,w)^2}{-1+i\alpha  g(0,w)}.
\]
Substituting $\phi(0,w)$ as above into \eqref{eq:t1}, we find that
\[
	f^{(0,3)} f(0,w)^2 = 0.
\]
Hence, either $f^{(0,3)} = 0$ or $f(0,w)=0$. But the latter also implies that $f^{(0,3)} = 0$. The proof is complete.
\end{proof}
\begin{lem}\label{lem:cond4A} If $\lambda = 0$, then $\sigma = 0$ and $\mu = -3i\nu$. Moreover,
\begin{align*}
	\phi ^{(3,1)} & = 12i \alpha  \bar{\nu},\\
	\phi ^{(2,2)} & = \frac{2i}{3} \left(g^{(0,4)}-4 f^{(1,3)}\right)-\frac{5\alpha}{3}  f^{(1,2)}-\frac{8\alpha}{3}   \bar{f}^{(1,2)}-\frac{1}{3}\alpha ^3-12 |\nu|^2,\\
	\phi ^{(1,3)} & = 4i\nu  \bar{f}^{(1,2)}-2i\nu  f^{(1,2)}-if^{(0,4)}+\frac{13}{2}i \alpha ^2 \nu ,\\
	\phi ^{(0,4)} &= 24\alpha\nu^2.
\end{align*}
\end{lem}
\begin{proof} The solvability condition for the overdetermined system for $l,m$, and $n$ reads
\begin{align}
\begin{pmatrix}
	-6 z & 0 & 0 & 18 i \nu  \phi -z^2 \bar{\phi} ^{(3,1)} \\
	2 i & -4 z & 0 & -2 z \bar{f} ^{(2,2)}-z^2 \bar{\phi} ^{(2,2)} +(\alpha^2-4 \bar{f} ^{(1,2)}-2 i \bar{\phi} ^{(2,1)} ) \phi\\
	0 & 3 i & -2 z & -2 z \bar{f} ^{(1,3)}+i\bar{g} ^{(1,3)}- z^2 \bar{\phi} ^{(1,3)}+(6 i\alpha  \bar{\nu}-3i \bar{\phi} ^{(1,2)}-2 \bar{f} ^{(0,3)} )\phi \\
	0 & 0 & 4 i & -2 z \bar{f} ^{(0,4)}+i \bar{g} ^{(0,4)}-z^2 \bar{\phi} ^{(0,4)}
\end{pmatrix}
=0.
\end{align}
Here we used $6\bar{\nu}^2+ i \bar{\phi} ^{(0,3)} = 0$. Recall that $\phi(z,0) =  2 \bar{\nu} z^3+\alpha  z^2$ and $f(z,0) = z$. Plugging these into the determinant and equating the coefficients of $z^k$, $k=2,3,4,5$, we obtain 
	\begin{align*}
	\phi ^{(3,1)} & = -2 i \left(3 f^{(2,2)}-2 g^{(1,3)}+9 \alpha  \bar{\nu}\right) = 4\left(2\bar{f}^{(0,3)} + 3i \alpha  \bar{\nu}\right),\\
	\phi ^{(2,2)} & = \alpha  f^{(1,2)}+\frac{2}{3} i \left(g^{(0,4)}-4 f^{(1,3)}\right)-\frac{8}{3} \alpha  g^{(0,3)}+\alpha ^3-12 |\nu|^2,\\
	\phi ^{(1,3)} & = 2 \alpha  f^{(0,3)}-i \left(6 \nu  f^{(1,2)}+f^{(0,4)}-4 \nu  g^{(0,3)}\right)-\frac{3}{2}i \alpha ^2 \nu ,\\
	\phi ^{(0,4)} &= -8 i \nu  f^{(0,3)}.
\end{align*}
With these equations being satisfied, we can solve for $l,m$, and $n$ along $\Sigma$.
\end{proof}
\begin{lem}\label{lem:jet34mix} 
\begin{align}
	\bar{g} ^{(1,3)} & = \frac{3}{4} \left(\bar{f} ^{(2,2)} +5 \alpha  \nu \right),\\
	\bar{g} ^{(0,4)} & = \frac{1}{2} i \left(4 \alpha  f^{(1,2)} +10 \alpha  \bar{f} ^{(1,2)} -4 i \bar{f} ^{(1,3)} -4 \alpha  \bar{g} ^{(0,3)} +2 \alpha ^3+75 |\nu|^2\right),\\
	\bar{f} ^{(0,4)} & = \frac{1}{2} \bar{\nu} \left(36  f^{(1,2)} +18  \bar{f} ^{(1,2)} -16  \bar{g} ^{(0,3)} -3 \alpha ^2\right),\\
	\bar{f} ^{(0,3)} & = -3 i \alpha  \bar{\nu}
\end{align}
\end{lem}
\begin{proof}
Applying the $4^{\mathrm{th}}$-order differential operators $\partial_w \partial_{\zba}^j\partial_{\wba}^{3-j}$, $j = 0,1,2,3$, to the mapping equation \eqref{eq:mapeq} and evaluating at $w = \zba = \wba = 0$, we obtain a system of 4 linear equations of 3 unknowns that must be satisfied by $p_w, t_w$, and $q_w$ along $\Sigma$. By Kronecker-Capelli theorem again, the following determinant must vanish identically on the first Segre set:
\[
\Delta:=
\det 
\begin{pmatrix}
	-z & 0 & 0 & -6 i |\nu|^2 z^3 \\
	i & -4 z & 0 & i l(z,0)-3 i \bar{\nu} z^3 \bar{\phi} ^{(2,1)} -25 |\nu|^2 z^2-5 \alpha  \nu  z \\
	0 & 2 i & -2 z & -5 i \bar{\nu} z^2 \bar{f} ^{(1,2)} -i \alpha  z \bar{f} ^{(1,2)} +i m-3 i \bar{\nu} z^3 \bar{\phi} ^{(1,2)} +2 \bar{\nu} \phi_w \\
	0 & 0 & 3 i & -5 i \bar{\nu} z^2 \bar{f} ^{(0,3)} -i \alpha  z \bar{f} ^{(0,3)} +i n(z,0)+18 \bar{\nu}^3 z^3 \\
\end{pmatrix} = 0.
\]
Expanding the determinant along the last column, the terms are of degree at most 6. Collecting the terms of kind $z^6$, we easily see that they are cancelled out. Since $n(z,0)$ are linear in $z$, the degree 5 terms only come from the terms of kind $z^2$ in $(4,4)$-entry and the terms of kind $z^3$ in $(3,4)$-entry. Explicitly, equating the coefficient of $z^5$, we obtain
\[
	\bar{\nu} \left(2 i \bar{f} ^{(0,3)} -3 \bar{\phi} ^{(1,2)} +12 \alpha  \bar{\nu}\right) =0.
\]
Using the formula for $\phi^{(1,2)}$ in \cref{lem:3jetcase1}, we obtain 
\[
	\bar{\nu} \left(3 \alpha  \bar{\nu}-i \bar{f} ^{(0,3)} \right) = 0.
\]
Thus, 
\[
	i \bar{f} ^{(0,3)} = 3 \alpha  \bar{\nu},
\]
provided that $\nu \ne 0$.  On the other hand, if $\nu = 0$, then this also holds by Lemma~\ref{lem:f03}. Consequently, $g^{(1,3)} = 0$, $f^{(2,2)} = -5 \alpha\bar{\nu}$, $\phi^{(3,1)} = 12i\alpha \bar{\nu}$.
\end{proof}
Putting these calculations above together, we obtain
\begin{lem}\label{lem:fwwgwww2} If $\lambda = 0$, then
\begin{align}
	\phi_w\bigl|_{\Sigma}  
	& = -3 i \nu  z+\frac{1}{3} i z^2 \left(2 \bar{f} ^{(1,2)}-4f ^{(1,2)}+ \alpha ^2\right)+2i\alpha \bar{\nu} z^3+6 i \bar{\nu}^2 z^4,\\
	f_{ww}\bigl|_{\Sigma}
	& =
	2 \nu+f^{(1,2)}z-\frac{5}{2}\alpha \bar{\nu}z^2 -15 \bar{\nu}^2 z^3,\\
	g_{www}\bigl|_{\Sigma}
	& = 2\Re f^{(1,2)}+\frac{\alpha ^2}{2}-36 \bar{\nu}^2 z^2.
\end{align}
\end{lem}
When the equalities in \cref{lem:cond4A} hold, the overdetermined system for $l, m, n$ can be solved uniquely.
Explicitly,
\begin{align}
	l\bigl|_{\Sigma} & = -\frac{1}{3} i z \left(2 \bar{g} ^{(1,3)}-3 \bar{f} ^{(2,2)}\right)-6 i |\nu|^2 z^2 =  -5i \alpha\nu z -6 i |\nu|^2 z^2, \\
	m\bigl|_{\Sigma} & = \frac{1}{6} \bar{\nu} z^2 \left(8 f^{(1,2)}-4 \bar{f} ^{(1,2)}+\alpha ^2\right)-\frac{1}{6} i z \left(\bar{g} ^{(0,4)}-4 \bar{f} ^{(1,3)}\right)+\frac{1}{3} \bar{g} ^{(1,3)} \notag \\
	& = \frac{1}{6} \bar{\nu} z^2 \left(8 f^{(1,2)}-4 \bar{f} ^{(1,2)}+\alpha ^2\right)-\frac{1}{6} i z \left(\bar{g} ^{(0,4)}-4 \bar{f} ^{(1,3)}\right),\\
	n\bigl|_{\Sigma} & = -2 \bar{\nu} z^2 \bar{f} ^{(0,3)}+\frac{1}{2} i z \bar{f} ^{(0,4)}+\frac{1}{4} \bar{g} ^{(0,4)} \notag  \\
	 & = 6i \alpha \bar{\nu}^2 z^2+\frac{1}{2} i z \bar{f} ^{(0,4)}+\frac{1}{4} \bar{g} ^{(0,4)}.
\end{align}
Recall that $k\bigl|_{\Sigma} = 0$.

In the next step, we solve for $p_w, t_w$, and $q_w$ along $\Sigma$. To this end, we apply $\partial_{\zba}^j \partial_{\wba}^{3-j} \partial_w$, $j=0,1,2,3$, to the mapping equation \cref{eq:mapeq}, evaluating along $\zba = \wba = w =0$, and plugging in the formulas above for $l,m,$ and $n$, to obtain
\begin{align*}
\begin{pmatrix}
	-z & 0 & 0 \\
	i & -4 z & 0 \\
	0 & 2 i & -2 z \\
	0 & 0 & 3 i \\
\end{pmatrix}
&\cdot
\begin{pmatrix}
	p_w \\
	t_w \\
	q_w
\end{pmatrix}\\
&=
\begin{pmatrix}
	-6i |\nu|^2 z^3 \\
	il + 2\bar{\nu} z^3(4\bar{f} ^{(1,2)}-2f^{(1,2)}-\alpha ^2)-25 |\nu|^2 z^2-5 \alpha  \nu  z \\
	im+2\bar{\nu} z^3 (2\bar{f} ^{(0,3)}-3 i \alpha  \bar{\nu})-5 i \bar{\nu} z^2 \bar{f} ^{(1,2)}-i \alpha  z \bar{f} ^{(1,2)}+2 \bar{\nu} \phi_w \\
	in-5 i \bar{\nu} z^2 \bar{f} ^{(0,3)}-i \alpha  z \bar{f} ^{(0,3)}+18 \bar{\nu}^3 z^3
\end{pmatrix}\\
 &= 
 \begin{pmatrix}
 	-6i |\nu|^2 z^3 \\
 	2\bar{\nu} z^3(4\bar{f} ^{(1,2)}-2f^{(1,2)}-\alpha ^2)-19 |\nu|^2 z^2 \\
 	im+18i\alpha \bar{\nu}^2 z^3-5 i \bar{\nu} z^2 \bar{f} ^{(1,2)}-i \alpha  z \bar{f} ^{(1,2)}+2 \bar{\nu} \phi_w \\
 	-\frac{i}{4}\bar{g}^{(0,4)}-\left(\frac{1}{2}\bar{f}^{(0,4)} + 3\alpha^2 \bar{\nu}\right)z-21\alpha\bar{\nu}^2 z^2+18 \bar{\nu}^3 z^3
 \end{pmatrix}
\end{align*}
when $w=0$. 
Here we have used various formulas for the $3^{\mathrm{rd}}$-order derivatives of $\phi$ at $z=w=0$ obtained earlier in \cref{lem:3jetcase1}.
The resulting overdetermined system is solvable if and only if the corresponding augmented matrix is degenerate. By direct computation the determinant is equal to
\[
	2 z^4 \left(2 \bar{\nu} \left(10 f^{(1,2)}+\bar{f} ^{(1,2)}\right)-2 \bar{f} ^{(0,4)}-11 \alpha ^2 \bar{\nu}\right)+2 z^3 \left(6 \alpha  \bar{f} ^{(1,2)}+2 i \left(\bar{g} ^{(0,4)}-2 \bar{f} ^{(1,3)}\right)+75 |\nu|^2\right).
\]
Equating the coefficients of powers of $z$ to be zero, we have

\begin{align}\label{eq:g4}
	\bar{g} ^{(0,4)} &= 3i \alpha  \bar{f} ^{(1,2)} + 2\bar{f} ^{(1,3)} + \frac{75i}{2} |\nu|^2,\\
	\bar{f} ^{(0,4)} &=  \bar{\nu} f^{(1,2)} + 10\bar{\nu} \bar{f} ^{(1,2)}-\frac{11}{2} \bar{\nu} \alpha ^2.
\end{align}
On the other hand, substituting $g^{(0,3)}$ and $g^{(0,4)}$ into \cref{lem:cond4A} we find that
\begin{equation}\label{eq:p221}
\phi ^{(2,2)} = \frac{1}{3} \left(10 \alpha  f^{(1,2)}-8 \alpha  \bar{f} ^{(1,2)}-4 i f^{(1,3)}-\alpha ^3+39   |\nu|^2\right).
\end{equation}
A ``reflection identity'' for $\phi^{(2,2)}$ is used in the following lemma.
\begin{lem}\label{lem:f12f13}
	Assume that $\lambda = 0$, then 
	\begin{equation}\label{eq:360}
		3 \alpha  \left(f^{(1,2)}+\bar{f} ^{(1,2)}\right)+2 i \left(f^{(1,3)}-\bar{f} ^{(1,3)}\right)+123 |\nu|^2 =0.
	\end{equation}
Moreover
\begin{align*}
\phi_{ww}\bigl|_\Sigma = & 6\alpha \nu z + \frac{1}{6}(285 |\nu|^2- \alpha^3 + 16 \alpha f^{(1,2)} -2 \alpha \bar f^{(1,2)}-4 i \bar f^{(1,3)}) z^2   \\
& - \frac{1}{3} \bar \nu (5 \alpha^2 - 68 f^{(1,2)}+46 \bar f^{(1,2)}) z^3 + \frac{25}{2} \alpha \bar \nu^2 z^4 - 9 \bar \nu^3 z^5.
\end{align*}
\end{lem}
\begin{proof}
	With all the identities above, we can solve for $p_w, t_w$ and $q_w$. Precisely
	\begin{align*}
		p_w & = 6 i |\nu|^2 z^2,\\
		t_w & = \frac{1}{2} \bar{\nu} z^2 \left(2 f^{(1,2)}-4 \bar{f} ^{(1,2)}+\alpha ^2\right)+\frac{13 |\nu|^2 z}{4},\\
		q_w & = -6 i \bar{\nu}^3 z^3+7 i \alpha  \bar{\nu}^2 z^2+\frac{1}{12} i \bar{\nu} z \left(20 f^{(1,2)}+2 \bar{f} ^{(1,2)}+\alpha ^2\right)-\\
		& \quad \frac{1}{24} i \left(-6 \alpha  \bar{f} ^{(1,2)}+4 i \bar{f} ^{(1,3)}-75 | \nu|^2\right).
	\end{align*}
Applying $\partial_w^2\partial_{\wba}^j\partial_{\zba}^{2-j}$, $j=0,1,2$, we obtain a system of the form:
\[
\begin{pmatrix}
	0 & 4 i z & 1 \\
	-z & 1 & 0 \\
	i & 0 & 0  \\
\end{pmatrix}
\cdot
\begin{pmatrix}
	r_{ww}\\
	s_{ww}\\
	\phi_{ww}
\end{pmatrix}
=
i \begin{pmatrix}
	p_w\\
	t_w\\
	q_w
\end{pmatrix} + 
\begin{pmatrix}
V_1\\
V_2\\
V_3
\end{pmatrix},
\]
where 
\begin{align*}
	V_1 & = -2 \alpha  z f_{ww}+ \alpha  z g_{ww}(4s-i \alpha z)+2 i \alpha  \phi_w+\frac{25}{2} \alpha  \bar{\nu}^2 z^4+5 \alpha ^2 \bar{\nu} z^3+\frac{\alpha ^3 z^2}{2}, \\
	V_2 & = \frac{1}{2} i \alpha  f_{ww}-3 i \bar{\nu} z f_{ww}-\frac{3}{2} \bar{\nu} g_{ww} \phi-3 \bar{\nu} \phi_w+\frac{75}{4} i \bar{\nu}^3 z^4+\frac{15}{2} i \alpha  \bar{\nu}^2 z^3+\frac{3}{4} i \alpha ^2 \bar{\nu} z^2,\\
	V_3 & = -2 \bar{\nu} f_{ww}.
\end{align*}
Solving for $\phi_{ww}$ we find that
\[
	\phi^{(2,2)} = \frac{1}{3} \left(16 \alpha  f^{(1,2)}-2 \alpha  \bar{f} ^{(1,2)}-4 i \bar{f} ^{(1,3)}-\alpha ^3+285 |\nu|^2\right).
\]
Compare this with formula \eqref{eq:p221} for $\phi^{(2,2)}$ as above, we are done.
\end{proof}
\begin{lem}\label{lem:holeq2}
	If $\lambda = 0$, then $\nu = \mu = \sigma = 0$. The components $f, \phi$, and $g$ of the map satisfy the following three holomorphic  functional equations:
	\begin{equation}\label{eq1}
	-4 w z f +i w^2 \phi+4 z^2 g +\alpha  w^2 \Upsilon= 0.
	\end{equation}
	\begin{equation}\label{eq:hol2}
	6 z (\alpha  w-2 i) f-w \phi + 12 i z^2 + \left(6 i \alpha w- w^2 \left(2f^{(1,2)}-4 \bar{f} ^{(1,2)}+ \alpha ^2\right)\right)\Upsilon = 0,
	\end{equation}
	\begin{equation}\label{eq:hol3}
	A_0(z,w) f(z,w) + B_0(z,w) \phi(z,w) + C_0(z,w) \Upsilon(z,w)=0,
	\end{equation}
	for
	\begin{align}
	A_0(z,w) 
	&=
	24 w z \bar{f} ^{(1,2)} + 24 i \alpha  z, \\
	B_0(z,w) 
	& =
	4 i w^2 \bar{f} ^{(1,2)} -8 i w^2 f^{(1,2)} - i \alpha ^2 w^2 -12 i,\\
	C_0(z,w)
	& =
	w^2 \left(2 \alpha  f^{(1,2)} -4i f^{(1,3)} -16 \alpha  \bar{f} ^{(1,2)} + \alpha ^3\right)\notag\\
	&\quad -8iw \left(2 f^{(1,2)} -16 \bar{f} ^{(1,2)}+\alpha ^2\right)-12 \alpha.
	\end{align}
\end{lem}
\begin{proof}
	Using the formula for $f_w, \phi_w$, and $g_w$ along $\Sigma$, we can produce another holomorphic equation for $f, \phi$, and $g$ by differentiating the mapping equation along the CR vector field. Indeed, we have
	\[
	L\left(L\left(\vr (H(z,\wba + 2i |z|^2), \overline{H}(\zba,\wba))\right)\right)\biggl|_{\wba=0} = 0.
	\]
	Expanding the derivatives and substituting the formulas above, we have
	\begin{equation}\label{eq:hol30}
		0
		=
		A(z,w) f(z,w) + B(z,w) \phi(z,w) + C(z,w) \Upsilon(z,w)  -192 \nu  w z^4,
	\end{equation}
	where, by direct calculation,
	\begin{align}
		A(z,w) 
		&=
		96 w z^4 \bar{f} ^{(1,2)} +360 \nu ^2 w^3 z^2+120 i \alpha  \nu  w^2 z^3	\notag \\
		& \quad +480 \nu  w z^3+384 i \bar{\nu} z^5+96 i \alpha  z^4, \\
		B(z,w) 
		& =
		16 i w^2 z^3 \bar{f} ^{(1,2)} -32 i w^2 z^3 f^{(1,2)} -135 i \nu ^2 w^4 z+24 \alpha  \nu  w^3 z^2 \notag \\
		& \quad -4 i \alpha ^2 w^2 z^3-216 i \nu  w^2 z^2-96 \bar{\nu} w z^4-48 i z^3,\\
		C(z,w)
		& =
		4w^2 z^3 \left(2 \alpha  f^{(1,2)} -4i f^{(1,3)} -16 \alpha  \bar{f} ^{(1,2)} + \alpha ^3-285 |\nu|^2 \right)\notag\\
		&	\quad -4iw^3 z^2 \left(46 \nu  f^{(1,2)}-68\nu  \bar{f} ^{(1,2)} +5\alpha ^2 \nu \right) \notag \\
		&\quad -16iw z^3 \left(4 f^{(1,2)} -32 \bar{f} ^{(1,2)}+2\alpha ^2\right)-48 \alpha  z^3 -144 \alpha  \nu  w^2 z^2+144 i \nu  w z^2\notag \\
		&\quad +75 \alpha  \nu ^2 w^4 z+288 i \nu ^2 w^3 z+27 i \nu ^3 w^5-288 i \alpha  \bar{\nu} w z^4-288 \bar{\nu} z^4.
	\end{align}
	On the other hand, from \cref{eq:hol1b}, we have
	\[
	g(z,w)= \frac{w \left(4 i z^2 f(z,w)+w f(z,w)^2 (\alpha  z-i \nu  w)+w z \phi (z,w)\right)}{w^2 \phi (z,w) (\nu  w+i \alpha  z)+4 i z^3}.
	\]
	Substituting this into $\Upsilon$, we have
	\[
	\Upsilon(z,w) = \frac{-z(2 z f(z,w)-i w \phi (z,w))^2}{\nu  w^3 \phi (z,w)+i \alpha  w^2 z \phi (z,w)+4 i z^3}.
	\]
	Plugging this into \cref{eq:hol30}, cancelling the nonzero term $z/(\nu  w^3 \phi (z,w)+i \alpha  w^2 z \phi (z,w)+4 i z^3)$, and setting $z=0$, we obtain
	\[
	108 i \nu ^3 w^7 \phi (0,w)^2 = 0.
	\]
	(In this step, the only term of degree 1 in $z$ in $B(z,w)$ and the term containing no $z$ in $C(z,w)$ are important.)
	Thus, either $\nu = 0$ or $\phi (0,w) = 0$. But in the latter case, we have $0=\phi^{(0,3)} = -6i \nu^2$ and hence $\nu=0$, as desired.
	
	Plugging these into \eqref{eq:hol1b} and cancelling $z$ we obtain \eqref{eq1}. Similarly, we get \eqref{eq:hol2} from \cref{lem:f03} and the third equation \eqref{eq:hol3} from \eqref{eq:hol30}. The proof is complete.
\end{proof}
From now on, we use the fact that $\nu = \mu = \sigma = 0$. Then, collecting from above,
\begin{align*}
	H(z,0) & = (z,\alpha z^2, 0),\\
	H_w(z,0) & = \left(\frac{i \alpha  z}{2}, \frac{1}{3} i z^2 \left(2 \bar{f} ^{(1,2)}-4f ^{(1,2)}+ \alpha ^2\right), 1\right),\\
	H_{ww}(z,0) & =\left(f^{(1,2)}z, -\frac{1}{6}\left(\alpha^3 - 16 \alpha f^{(1,2)} +2 \alpha \bar f^{(1,2)}+4 i \bar f^{(1,3)} \right) z^2, 0\right).
\end{align*}

For each triple of parameters $(\alpha,f^{(1,2)},f^{(1,3)})\in \mathbb{R}\times \mathbb{C}\times \mathbb{C}$ satisfying  \eqref{eq:360} with $\nu = 0$, the system of holomorphic functional equations \eqref{eq1}, \eqref{eq:hol2}, and \eqref{eq:hol3} above has a unique algebraic solution, which may have a singularity at the origin. Even in the case the solution is holomorphic near the origin, the resulting map does not necessarily send $\heis$ into $\lc$, see \cref{ex:holEqUpTo14} below. In the next step, we determine all the triples of parameters $\left(\alpha, f^{(1,2)}, f^{(1,3)}\right)$ such that the solution of this system is a genuine map sending $\heis$ into $\lc$. We write $f^{(1,2)} = \eta + i\xi$, with $\eta,\xi \in \mathbb{R}$. From \cref{lem:f12f13}, we can write $f^{(1,3)} = \gamma + (3i/2) \alpha \eta$ with $\gamma \in \mathbb R$. The mapping equation \cref{eq:mappara} gives the following constrains on these parameters.
\begin{lem}\label{lem:remeq}
\begin{equation}\label{eq:remeq1}
	\alpha^5 - 4 \alpha^3 \eta + 4 \alpha \left(90 \xi ^2 + \eta^2 \right) - 72 \gamma \xi  = 0,
\end{equation}
\begin{equation}\label{eq:remeq2}
	23 \alpha ^4 \xi  -4 \alpha ^3 \gamma -56 \alpha ^2 \xi   \eta +8 \alpha  \gamma  \eta -1116 \xi  ^3+20 \xi   \eta ^2=0,
\end{equation}
\begin{equation}\label{eq:remeq3}
	7 \alpha ^6-240 \alpha ^4 \eta +2304 \alpha ^2 \left(142 \xi  ^2+\eta ^2\right)-290304 \alpha  \gamma  \xi  +62208 \gamma^2-4096 \left(333 \xi  ^2 \eta +\eta ^3\right) =0.
\end{equation}
\end{lem}

The following will be useful for the proof of \cref{lem:remeq}.

\begin{lem}\label{lem:f0w} If $\lambda = 0$, then 
	\[
	f(0,w) = \phi(0,w) = 0, \quad \phi_z(0,w) = 0.
	\]
\end{lem}

\begin{proof}
Setting $z=0$ in the first holomorphic functional equation \eqref{eq1}, and solving for $\phi(0,w)$ we have
\[
\phi (0,w)= -\frac{ \alpha  f(0,w)^2}{1-i\alpha  g(0,w)}.
\]
Setting $z=0$ into the second holomorphic functional equation \eqref{eq:hol2} and plugging in $\phi(0,w)$ formula above we have
\[
w^2 f(0,w)^2 \left(2 f^{(1,2)}-4 \bar{f}^{(1,2)}+\alpha ^2\right) = 0.
\]
Assume, for contradiction, that $f(0,w)\ne 0$, then $f^{(1,2)} = \alpha^2/2$. 

Setting $z=0$ into the third holomorphic functional equation \eqref{eq:hol3} and plugging in $\phi(0,w)$ formula above we have
\[
-2 w^2 f(0,w)^2 \left(5 \alpha  f^{(1,2)}-2 i  f^{(1,3)}-10 \alpha   \bar{f}^{(1,2)}+\alpha ^3 \right) = 0.
\]
Solving for $f^{(1,3)}$ we have 
\[
f^{(1,3)} = \frac{3 i \alpha ^3}{4}.
\]
Plugging these two values into the holomorphic equations, we obtain precisely the map which satisfies $f(0,w) = 0$. Thus, in any case, $f(0,w) = 0$ and hence $\phi(0,w) = 0$. The proof is complete. 
\end{proof}

For the proof of \cref{lem:remeq} we associate weight $1$ to $z$ and weight $2$ to $w$ and consider a weighted homogeneous expansion of $H$, i.e.,
\begin{align*}
f(z,w)  = z + \sum_{i\geq 2}f_i(z,w), \quad 
\phi(z,w) =  \sum_{j\geq 2} \phi_j(z,w), \quad 
g(z,w)  = w+ \sum_{k\geq 3} g_k(z,w),
\end{align*}
where $h_m(z,w)$ is a weighted homogeneous polynomial of order $m\geq 2$. Collecting weighted homogeneous terms of weight $k+1$ for $k\geq 3$ in the mapping equation gives:
\begin{align*}
\text{Re}\bigl(i g_ {k+1}(z,w) +2 \bar z f_k(z,w) +\bar z^2 \phi_{k-1}(z,w)\bigr) + \cdots = 0,
\end{align*}
for $(z,w) \in \heis$. The dots $\cdots$ denote terms, which involve parts of the weighted homogeneous expansion of the map of lower order, which appear for $k\geq 5$. 

The lower weight components are computed above as follows:
\begin{quote}
\begin{multicols}{2}
\noindent
$H_1  = \left(z,0,0\right)$ \\
$H_2  = \left(0,\alpha z^2,w\right)$ \\
$H_3  = \left(\frac{i\alpha}{2}zw,0,0\right)$ \\
$H_4  = \left(0,\frac{1}{2}\phi^{(2,1)} z^2 w,0\right)$ \\
$H_5  = \left(\frac{1}{2} f^{(1,2)}z w^2,0,0\right)$\\
$H_6  = \left(0,\frac{1}{4}\phi^{(2,2)} z^2 w^2,\frac{1}{6} g^{(0,3)} w^3\right)$\\
$H_7  = \left(\frac{1}{6} f^{(1,3)}zw^3,0,0\right),$
\end{multicols}
\end{quote}
\noindent
where $\phi^{(2,1)} = \frac{2i}{3} \left(2 \bar{f} ^{(1,2)}-4f ^{(1,2)}+ \alpha ^2\right)$, $\phi^{(2,2)} = \frac{1}{3} \left(16 \alpha  f^{(1,2)}-2 \alpha  \bar{f} ^{(1,2)}-4 i \bar{f} ^{(1,3)}-\alpha ^3\right)$, and $g^{(0,3)} = \alpha^2/2 +2\Re f ^{(1,2)}$.
\begin{lem}
	For every $j\geqslant 2$, $g_{2j+1} = f_{2j} = \phi_{2j-1}= 0$ and $g_{2j} = g_{2j,j} w^j$ with $g_{2j,j}$ is real.
\end{lem}
\begin{proof} Observe that the conclusion holds for $j=0,1,2,3$. We argue by induction as follows. Setting $z = 0$ in the mapping equation of weight $2j+1$, noticing that $g_{2j+1}(0, \wba) = 0$ (as $\wba$ has weight 2), and using \cref{lem:f0w} above ($f(0,w) = \phi(0,w) = 0$), we immediately conclude that $g_{2j+1}$ must vanish. Plugging this back into the mapping equation of weight $2j+1$, the remaining terms are divisible by $z$. Moreover, $f_{2j}$ is divisible by $z^2$ while $f_{2j-2} = 0$ by induction assumption. Dividing the equation by $z$ and setting $z=0$, we obtain $f_{2j} = 0$. Plugging these two into the mapping equation and arguing similarly, we find that $\phi_{2j-1} =0$.
\end{proof}
In particular,
\[
	g_{8} = \frac{1}{24}g^{(0,4)} w^4 = \frac{1}{24}\left(2f ^{(1,3)}-3i \alpha f^{(1,2)}\right) w^4.
\]

Comparing the terms of weight 10 in the mapping equation, we have
\begin{align*}
	0= & ig_{10,5}\left(w^5 - \wba^5\right)  + 2 \sum_{j=0}^4 f_{2j+1} \overline{f}_{9-2j} + 2 \Im \sum_{l=1}^3 \phi_{2l} \sum_{j=1}^{4-l}  \bar \phi_{10-2l-2j} g_{2 j}
	\\ &
	+ \Re \sum_{l=1}^4 \bar{\phi}_{2l} \sum_{j=0}^{4-l} f_{2j+1} f_{9-2l-2j}\biggl|_{w = \wba + 2i z\zba}.
\end{align*}
Observe that in each monomial appearing in the right hand side, the highest degree in $\wba$ is $4$. Collecting the terms $z\zba \wba^4$ appearing in $g_{10}$, $f_1 \fba_9$, $\fba_1 f_9$, and $f_5 \fba_5$, we have
\[
	2 \Re f_{9,4} - 5 g_{10,5} + \alpha \Im f_{7,3} + |f_{5,2}|^2 = 0.
\]
Collecting the terms of kind $z \zba^3 \wba^3$ appearing only in $f_1 \fba_9$ we find that $f_{9,3} = 0$. Thus, $f_9$ is monomial, that is $f_9 = f_{9,4} z w^4$. Collecting terms of kind $z^2 \zba^2 \wba^3$ we have
\begin{align*}
48 f_{9,4} - 120 g_{10,5} - 3 i \Re \phi_{8,3} -10\alpha ^2 \Im f_{5,2} + 24 |f_{5,2}|^2 -16 i  \Im f_{5,2}^2 - 6 i \alpha (3 f_{7,3}-\bar f_{7,3}) =0.
\end{align*}
The coefficient of $z^5 \bar z^5$ gives:
\begin{align*}
4 f_{9,4} - i \phi_{8,3} - 4 g_{10,5}  - 2 i \alpha^2 \Re f_{5,2}=0.
\end{align*}
The coefficient of $z^4 \bar z^4 \bar w$ gives:
\begin{align*}
576 f_{9,4} - 108 i \phi_{8,3} - 720 g_{10,5}  & - \alpha^4 - 128 |f_{5,2}|^2 - 72 i \alpha f_{7,3} \\
& - 104 \alpha^2 f_{5,2} + 100 \alpha^2 \bar f_{5,2} + 32 f_{5,2}^2 + 128 \bar f_{5,2}^2=0.
\end{align*}
And finally, the coefficient of $z^3 \bar z^3 \bar w^2$ gives:
\begin{align*}
864 f_{9,4} - 108 i \phi_{8,3} - 1440 g_{10,5} &  - \alpha^4 - 16 |f_{5,2}|^2 - 216 i \alpha f_{7,3} \\
& - 92 \alpha^2 f_{5,2} + 88 \alpha^2 \bar f_{5,2} - 64 f_{5,2}^2 +  224 \bar f_{5,2}^2=0.
\end{align*}
Solving for $(f_{9,4},\phi_{8,3}, g_{10,5})$ and using \eqref{lem:f12f13} in the above system of five equations we obtain:
\begin{align*}
f_{9,4} & = \frac{1}{144}\left(5 \alpha^4 + 928 |f_{5,2}|^2  +72 i \alpha f_{7,3} + 4 \alpha^2 f_{5,2}  + 16 \alpha^2 \bar f_{5,2} - 352 f_{5,2}^2- 448 \bar f_{5,2}^2 \right),\\
\phi_{8,3} & =  -\frac{i}{12}\left(\alpha^4 + 176 |f_{5,2}|^2 + 24 i \alpha f_{7,3} -12 \alpha^2 f_{5,2} + 16 \alpha ^2 \bar f{5,2} - 64 f_{5,2}^2 -96 \bar f_{5,2}^2  \right), \\
g_{10,5} & = \frac{1}{72} \left(\alpha^4 + 200 |f_{5,2}|^2 + 4 \alpha^2 \Re f_{5,2} - 160 \Re  f_{5,2}^2  \right).
\end{align*}

Considering terms of weight $12$ of different kinds in the mapping equation, we obtain from $z^8 \bar z^4$ that $f_{11,3}=0$, from $z^2 \bar z^4 \bar w^3$ that $\phi_{10,3} = 0$ and from $z^5 \bar z^3 \bar w^2$ that $f_{11,4} = 0$, which shows that $f_{11}= f_{11,5} z^2w^5$ and $\phi_{10}= \phi_{10,4} z^2 w^4$ are monomials. 

Furthermore, the coefficients of $z,\bar z$ and $\bar w$ allow to solve for the triple $(f_{11,5},\phi_{10,4},g_{12,6})$ and provide an additional equation for $(\alpha, f_{5,2}, f_{7,3})$. More precisely, the coefficient of $z \bar z w^5$ gives
\begin{align*}
24   \Re f_{11,5} - 72 g_{12,6} -  \alpha^3 \Im f_{5,2} + 6 \alpha^2 \Re f_{7,3} + 8  \alpha \Im f_{5,2}^2 + 24 \Re(f_{7,3}\bar f_{5,2})=0.
\end{align*}
The coefficient of $z^5 \bar z^5 \bar w$ gives
\begin{align*}
& 160 i f_{11,5} - 192 i g_{12,6} + 32 \phi_{10,4}  + \alpha^5 + 20 \alpha^3 f_{5,2} - 24 \alpha ^3 \bar f_{5,2}  - 32 \alpha f_{5,2}^2 \\
& + 64 \alpha |f_{5,2}|^2 - 8 i \alpha^2 f_{7,3}    = 0.
\end{align*}
The coefficient of $z^3 \bar z^3 \bar w^3$ gives 
\begin{align*}
&  720 i f_{11,5} + 72 \phi_{10,4}- 1440 i g_{12,6} + 8 \alpha ^5 + 1312 \alpha |f_{5,2}|^2 +  43 \alpha^3 f_{5,2} - 29 \alpha^3 \bar f_{5,2} - 496 \alpha f_{5,2}^2 \\  & - 568 \alpha \bar f_{5,2}^2   + 264 i \bar f_{5,2} f_{7,3}  + 24 i f_{5,2}\bar f_{7,3} +192 \Im(f_{5,2} f_{7,3})   + 66 i \alpha^2 f_{7,3}  + 6 i \alpha^2 \bar f_{7,3}  =0.
\end{align*}
The coefficient of $z^6 \bar z^6$ gives
\begin{align*}
& 288 i f_{11,5} + 72 \phi_{10,4}  - 288 i g_{12,6} + \alpha^5 - 160 \alpha |f_{5,2}|^2 + 50 \alpha^3 f_{5,2}   - 64 \alpha^3 \bar f_{5,2} + 40 \alpha f_{5,2}^2 \\
& + 160 \alpha \bar f_{5,2}^2  - 36 i \alpha^2 f_{7,3} = 0.
\end{align*}
Solving the above system for $(f_{11,5},\bar f_{11,5},\phi_{10,4}, g_{12,6})$ we obtain
\begin{align*}
f_{11,5} & = \frac{i}{48}(\alpha^5 + 320 \alpha  |f_{5,2}|^2 -11 \alpha^3 f_{5,2} + 15 \alpha ^3 \bar f_{5,2} - 120 \alpha  f_{5,2}^2  - 184 \alpha  \bar f_{5,2}^2  + 14 i \alpha^2 f_{7,3} \\
&  \qquad - 2 i \alpha ^2 \bar f_{7,3} - 8 i f_{5,2} \bar f_{7,3} - 88 i \bar f_{5,2}f_{7,3}   -64  \Im(f_{5,2} f_{7,3})),\\
\bar f_{11,5} & = -\frac{i}{96} (\alpha^5 - 64 \alpha | f_{5,2}|^2 + 28 \alpha ^3 f_{5,2} - 24 \alpha ^3  \bar f_{5,2} + 128 \alpha \bar f_{5,2}^2 - 8 i \alpha ^2 f_{7,3} - 16 i \alpha ^2 \bar f_{7,3}  \\
& \qquad -64 i f_{5,2} \bar f_{7,3 } + 256 i  \bar f_{5,2}f_{7,3} + 256 \Im(f_{5,2} f_{7,3})),\\
\phi_{10,4} & = \frac{1}{48} (3 \alpha^5 + 1152 \alpha |f_{5,2}|^2 - 62 \alpha ^3 f_{5,2} + 86 \alpha ^3 \bar f_{5,2} - 416 \alpha f_{5,2}^2 - 688 \alpha \bar f_{5,2}^2 + 76 i \alpha ^2 f_{7,3} \\
& \qquad - 4 i \alpha^2 \bar f_{7,3} - 16 i f_{5,2} \bar f_{7,3} - 176 i \bar f_{5,2} f_{7,3} -128 \Im(f_{5,2} f_{7,3})),\\
g_{12,6} & = \frac{i}{576} (\alpha^5 + 704 \alpha |f_{5,2}|^2 - 46 \alpha^3 f_{5,2} + 50 \alpha^3 \bar f_{5,2} - 272 \alpha  f_{5,2}^2 - 464 \alpha  \bar f_{5,2}^2 \\
& \qquad - 24  \alpha ^2 \Im f_{7,3} - 48 i f_{5,2} \bar f_{7,3} - 528 i \bar f_{5,2} f_{7,3}  -384 \Im(f_{5,2} f_{7,3})).
\end{align*}
The remaining equation (or equivalently the difference of the conjugate of the formula of $f_{11,5}$ and $\bar f_{11,5}$ above) using \eqref{lem:f12f13} is given by \eqref{eq:remeq1}.

Considering terms of weight $14$ of various kinds, we obtain from $z^9 \bar z^3 \bar w$ that $f_{13,3}=0$ and $z^2 \bar z^6 \bar w^3$ gives $\phi_{12,3}=0$. The coefficient of $z^9 \bar z^5$ implies $f_{13,4}=0$, the coefficient of $z^3 \bar z \bar w^5$ shows $f_{13,5} = 0$ and the coefficient of $z^5 \bar z^3 \bar w^3$ implies $\phi_{12,4} = 0$. Hence, $f_{13} = f_{13,6} z^2 w^6$ and $\phi_{12} = \phi_{12,5} z^4 w^5$ are monomials.

Furthermore, the coefficient of $z \bar z \bar w^6$ gives
\begin{align*}
& 1152 \Re f_{13,6} - 4032 g_{14,7} + 9 \alpha^6 + 1760 \alpha^2 |f_{5,2}|^2 + 76 \alpha^4 \Re f_{5,2}  -1312 \alpha^2 \Re f_{5,2}^2 \\
& - 3584 \Re f_{5,2}^3 + 4608 |f_{5,2}|^2 \Re \bar f_{5,2}  + 576 |f_{7,3}|^2 - 120 \alpha^3 \Im f_{7,3} + 384 \alpha  \Im (f_{5,2} f_{7,3}) \\
& + 1056 \alpha \Im (f_{5,2} \bar f_{7,3}) = 0,
\end{align*}
the coefficient of $z^7 \bar z^7$ gives
\begin{align*}
& 576 f_{13,6} -144 i \phi_{12,5} -576 g_{14,7} +\alpha ^6 -1072 \alpha ^2 |f_{5,2}|^2 + 92 \alpha^4 f_{5,2} - 136 \alpha^4 \bar f_{5,2} + 352 \alpha^2 f_{5,2}^2 \\
& + 736 \alpha^2 \bar f_{5,2}^2 - 72 i \alpha^3 f_{7,3} = 0,
\end{align*}
the coefficient of $z^3 \bar z^3 \bar w^4$ gives
\begin{align*}
& 4320 f_{13,6} - 360 i \phi_{12,5} - 10080 g_{14,7}  + 31 \alpha^6 + 7312 \alpha^2 |f_{5,2}|^2 - 154 \alpha^4 f_{5,2} + 246 \alpha ^4  \bar f_{5,2} \\
& - 2576 \alpha^2 f_{5,2}^2 - 4056 \alpha^2 \bar f_{5,2}^2 + 768 f_{5,2}^3 - 5792 \bar f_{5,2}^3 + 15296 f_{5,2} \bar f_{5,2}^2 - 8384 f_{5,2}^2 \bar f_{5,2} \\
& + 252 i \alpha^3 f_{7,3} - 72 i \alpha ^3 \bar f_{7,3} + 864 |f_{7,3}|^2  - 720 i \alpha  f_{5,2} \bar f_{7,3} - 2160 i \alpha \bar f_{5,2} f_{7,3} \\
& - 1728  \alpha  \Im (f_{5,2} f_{7,3}) = 0,
\end{align*}
and the coefficient of $z^4 \bar z^4 \bar w^3$ gives
\begin{align*}
& 1440 f_{13,6} - 180 i \phi_{12,5} - 2520 g_{14,7} + 8 \alpha^6 + 1236 \alpha ^2 |f_{5,2}|^2 + 11 \alpha ^4 f_{5,2} - 20 \alpha^4 \bar f_{5,2} - 470 \alpha^2 f_{5,2}^2 \\
&  - 646 \alpha^2 \bar f_{5,2}^2 + 400 f_{5,2}^3 - 672 \bar f_{5,2}^3 + 2416 f_{5,2}\bar f_{5,2}^2 - 1840 f_{5,2}^2 \bar f_{5,2} + 24 i \alpha^3 f_{7,3} - 18 i \alpha ^3 \bar f_{7,3} \\
&  + 72 |f_{7,3}|^2  -624 i \alpha  \bar f_{5,2}f_{7,3} - 144 i \alpha  f_{5,2} \bar f_{7,3}  + 204 i \alpha  f_{5,2} f_{7,3} - 180 i \alpha  \bar f_{5,2} \bar f_{7,3}  = 0.
\end{align*}
We then solve the above system of four equations for the quadruple $(f_{13,6},\bar f_{13,6},\phi_{12,5}, g_{14,7})$ to obtain
\begin{align*}
f_{13,6} & =  \frac{1}{5760}(-39 \alpha^6 + 4 \alpha^4 (288 f_{5,2}-335 \bar f_{5,2})-16 \alpha^2 (1483 |f_{5,2}|^2 - 570 f_{5,2}^2 - 1013 \bar f_{5,2}^2) \\
& + 64(|f_{5,2}|^2(70 f_{5,2} + 578 \bar f_{5,2}) - 180 f_{5,2}^3 - 401 \bar f_{5,2}^3) - 144 i \alpha^3 (8 f_{7,3}-\bar f_{7,3})  \\
&+ 5184 |f_{7,3}|^2  - 288 i \alpha  (f_{5,2} (4 f_{7,3} + \bar f_{7,3}) - 29  \bar f_{5,2} f_{7,3}) ),\\
\bar f_{13,6} & = \frac{1}{5760}(-9 \alpha ^6 + \alpha ^4 (568 \bar f_{5,2} - 636 f_{5,2}) - 16 \alpha^2 (163 |f_{5,2}|^2 - 162 f_{5,2}^2 - 41 \bar f_{5,2}^2) \\
& - 64(|f_{5,2}|^2(878 f_{5,2} - 1526\bar f_{5,2}) - 96 f_{5,2}^3 + 677 \bar f_{5,2}^3)- 24 i \alpha^3 (5 f_{7,3} - 19 \bar f_{7,3})  \\
&  + 5184 |f_{7,3}|^2 + 192 i \alpha (f_{5,2} (23 f_{7,3} + 8 \bar f_{7,3}) - \bar f_{5,2} (22 f_{7,3} + 31 \bar f_{7,3}))),\\
\phi_{12,5} & = \frac{i}{1440}(35 \alpha^6 -4 \alpha ^4 (486 f_{5,2}-661 \bar f_{5,2}) + 16 \alpha^2 (2075 |f_{5,2}|^2 - 744 f_{5,2}^2 - 1381 \bar f_{5,2}^2) \\
& - 64 (|f_{5,2}|^2(134 f_{5,2}+ 226 \bar f_{5,2}) - 128 f_{5,2}^3 - 207 \bar f_{5,2}^3) + 48 i \alpha ^3 (37 f_{7,3}-3 \bar f_{7,3}) \\
& - 2880 |f_{7,3}|^2 + 96 i \alpha (f_{5,2} (14 f_{7,3} - 3 \bar f_{7,3})-\bar f_{5,2} (73 f_{7,3} + 6 \bar f_{7,3}))),\\
g_{14,7} & = \frac{1}{2880}(3 \alpha^6 + 4 \alpha ^4 (16 f_{5,2} - 7 \bar f_{5,2}) - 16 \alpha^2 (39 f|_{5,2}|^2 - 23 f_{5,2}^2 - 46 \bar f_{5,2}^2 ) \\
& - 64 (|f_{5,2}|^2(32 f_{5,2} -176 \bar f_{5,2}) + 26 f_{5,2}^3 + 97 \bar f_{5,2}^3) - 48 i \alpha^3 f_{7,3} + 1152 |f_{7,3}|^2 \\
& + 96 i \alpha  (f_{5,2} (f_{7,3}-3 \bar f_{7,3})+\bar f_{5,2} (7 f_{7,3}-3 \bar f_{7,3}))).
\end{align*}

Using the above expressions and \eqref{lem:f12f13}, the remaining equations reduce to \eqref{eq:remeq2} and \eqref{eq:remeq3}. This finishes the proof of \cref{lem:remeq}.

\begin{lem}\label{lem:solAlpha}
	The solutions of \eqref{eq:remeq1}--\eqref{eq:remeq3} are given as follows: If $\alpha \neq 0$, then the only solution is $\xi =\gamma=0$ and $\eta={\alpha^2}/{2}$. 
	If $\alpha = 0$, then a solution satisfies $\xi = 0$ and $6 \gamma^2 - \eta^3 = 0$.
\end{lem}

\begin{proof}
If $\alpha = 0$, then \eqref{eq:remeq1} shows that $\gamma \xi  = 0$. When $\gamma = 0$, we have $\eta = \xi  = 0$. If $\xi  = 0$, then \eqref{eq:remeq2} reduces to $6 \gamma^2 - \eta^3 = 0$.\\
Consider the case $\alpha \neq 0$. Assume $\xi  \neq 0$, then \eqref{eq:remeq1} shows
\begin{align*}
	\gamma = \frac{1}{72 \xi } \alpha(1+360 \xi ^2-4 \eta(1-\eta)).
\end{align*}
The remaining two equations then become
\begin{align}\label{eq:remeq4a}
	18\xi ^2(3-4 \eta(4-5 \eta))-(1-2 \eta)^3-20088\xi ^4 = 0,\\
	\label{eq:remeq4b}
	(1-2 \eta)^4-36 \xi ^2(1-2\eta)^2(1+6 \eta)+2592(53- 111 \eta)\xi ^4 = 0.
\end{align}
Note that $\eta = 53/111$ does not lead to a solution. So we can assume $\eta \neq 53/111$. Combing \eqref{eq:remeq4a} and \eqref{eq:remeq4b}, to eliminate $\xi ^4$, we obtain
\begin{align}\label{eq:remeq5}
	(1-2 \eta)((1-2 \eta)^2(181 - 382 \eta) - 36 \xi ^2(287 - 2 \eta(925-1296 \eta)))=0.
\end{align}
Observe that $\eta = 1/2$ leads to $\xi  = 0$, which is not allowed. Also, if the coefficient of $\xi ^2$ is zero, this gives a contradiction to the fact that $\xi  \in \RR$. Solving \eqref{eq:remeq5} for $\xi ^2$ gives
\begin{align}\label{eq:remeq6}
	\xi ^2 = \frac{(1-2\eta)^2(182-382\eta)}{36(287-2 \eta(925-1296 \eta))}.
\end{align}
Substituting this into the equations \eqref{eq:remeq4a} and \eqref{eq:remeq4b} implies that $\eta \in\{53/111, 1/2\}$, which we already excluded, or $\eta = 51/22$. Assuming the latter leads to a negative expression in the right-hand side of \eqref{eq:remeq6}. Thus the case $\xi  \neq 0$ does not lead to any solution. Assuming $\xi  = 0$, directly leads to $\gamma = 0$ and $\eta = \alpha^2/2$, which completes the proof.
\end{proof}

When $\alpha \neq 0$, the $4$-jet of the map at $0$ is completely fixed, while for $\alpha = 0$ there is still one parameter to be determined. In the latter case we need to consider weight $16$ in the weighted homogeneous expansion of the mapping equation.

Assuming $\alpha = 0$, we obtain from the coefficient of $z^{12} \bar z^4$ that $f_{15,3} = 0$, from $z^7 \bar z \bar w^4$ that $f_{15,4} = 0$, from $z^{11} \bar z^5$ that $\phi_{14,3}=0$, from $z^5 \bar z \bar w^5$ that $f_{15,5}=0$, from $z^9 \bar z^5 \bar w$ that $\phi_{14,4} = 0$, from $z^3 \bar z \bar w^6$ that $f_{15,6}=0$, and from $z^7 \bar z^5 \bar w^2$ that $\phi_{14,5} = 0$. This shows that $f_{15}=f_{15,7} z^2 w^7$ and $\phi_{14}= \phi_{14,6} z^4 w^6$ are both monomials.

Furthermore, the coefficient of $z \bar z \bar w^7$ gives
\begin{align*}
& 36 \Re f_{15,7} - 144 g_{16,8} - 8 f_{5,2}^2 (10 f_{7,3}+7 \bar f_{7,3}) - 11 \bar f_{5,2}^2 (f_{7,3} + 4 \bar f_{7,3}) \\
& + |f_{5,2}|^2 (152 f_{7,3} + 143 \bar f_{7,3}) =0,
\end{align*}
the coefficient of $z^7 \bar z^7 \bar w$ gives
\begin{align*}
14 f_{15,7}  - 3 i \phi_{14,6} - 16 g_{16,8} =0,
\end{align*}
the coefficient of $z^5 \bar z^5 \bar w^3$ gives
\begin{align*}
& 630 f_{15,7} - 90 i \phi_{14,6} -1008 g_{16,8} + 32 f_{5,2}^2 (f_{7,3}-\bar f_{7,3}) + \bar f_{5,2}^2(209 f_{7,3}  + 100 \bar f_{7,3}) \\
& + |f_{5,2}|^2 (23 \bar f_{7,3} - 164 f_{7,3}) =0,
\end{align*}
and the coefficient of $z^8 \bar z^8$ gives
\begin{align*}
4 i f_{15,7} + \phi_{14,6} - 4 i g_{16,8} =0.
\end{align*}

Solving the above system for $(f_{15,7},\bar f_{15,7}, \phi_{14,6},g_{16,8})$, we obtain
\begin{align*}
f_{15,7} & = \frac{1}{54} \left(32 f_{5,2}^2 (f_{7,3}-\bar f_{7,3}) + \bar f_{5,2}^2 (209 f_{7,3}+100 \bar f_{7,3}) + |f_{5,2}|^2(23 \bar f_{7,3}- 164 f_{7,3})\right),\\
\bar f_{15,7} & = \frac{2}{9} \left(f_{5,2}^2 (28 f_{7,3} + 6 \bar f_{7,3}) + \bar f_{5,2}^2 (55 f_{7,3} + 36 \bar f_{7,3}) - |f_{5,2}|^2 (79 f_{7,3} + 30 \bar f_{7,3}) \right),\\
\phi_{14,6} &= -\frac{i}{27} \left(32 f_{5,2}^2 (f_{7,3}-\bar f_{7,3}) + \bar f_{5,2}^2 (209 f_{7,3} + 100 \bar f_{7,3}) + |f_{5,2}|^2 (23 \bar f_{7,3} - 164 f_{7,3}) \right),\\
g_{16,8} & = \frac{1}{108} \left(32 f_{5,2}^2 (f_{7,3}-\bar f_{7,3}) + \bar f_{5,2}^2 (209 f_{7,3}+100 \bar f_{7,3}) + |f_{5,2}|^2 (23 \bar f_{7,3} - 164 f_{7,3}) \right). 
\end{align*}

When using $\xi=0$, the remaining coefficients reduce to the condition $\gamma \eta^2 = 0$, which, together with \cref{lem:solAlpha}, implies that $f^{(1,2)}=f^{(1,3)}=0$. This fixes the $4$-jet of the map at $0$ in the case $\alpha = 0$.

We sum up and list the $4$-jet of $H$ at $0$ in the case $\lambda = 0$, when setting $\alpha = 2 \beta$, 
\begin{align*}
f(z,w) & = z(1 + i \beta w + \beta^2 w^2 + i \beta^3 w^3) + O(5)\\
\phi(z,w) & = 2 \beta z^2 (1+ \beta w^2) + O(5)\\
g(z,w) & = w (1 + \beta w^2) + O(5)
\end{align*}

Note that we have the following expressions:
\begin{align*}
H(z,0)  = (z,\alpha z^2,0), \quad 
H_w(z,0)  = \left(\frac{i \alpha}{2} z,0,1\right), \quad 
H_{ww}(z,0) = \left(\frac{\alpha^2}{2} z, \frac{\alpha^3}{2} z^2,0\right).
\end{align*}

Using the information about the $4$-jet of $H$ in \cref{lem:holeq2} we obtain the following holomorphic equations for $H$:
\begin{align*}
4 i z(w f - z g) + \alpha w^2  f^2 + w^2(1- i \alpha g) \phi & = 0\\
2 i z^2 - (2 i- \alpha w) z f - \alpha w f^2 -w(1- i\alpha g) \phi & = 0\\
4 i \alpha z^2 f - z (2 i- \alpha w)( \alpha  f^2 + (1- i \alpha g) \phi) & = 0. 
\end{align*}

Solving this system of equations we get the family of maps, when writing $\alpha = 2 \beta$,
\begin{align*}
r_\beta(z,w) = \frac{1}{1-\beta^2 w^2}\left(z(1+ i \beta w), 2 \beta z^2, w \right).
\end{align*}
By the partial normal form in \cref{thm:normalform} we can assume $\beta\in \{-1,0,1\}$. Taking $\beta\in \{-1,1\}$ gives $r_1$ and $r_{-1}$ in \cref{thm:main1} and $r_0$ corresponds to $\ell$. We finish Case 2 when $\lambda =0$. Together with Case 1 treated in the last section, this completes the proof of part (a) of Theorem \ref{thm:main1}, giving a complete classification of CR transversal maps from $\heis$ into $\lc$.

The pairwise inequivalences of these 4 maps are clear from their partial normal forms.
\subsection{Nowhere CR transversal maps}

In this section we discuss nowhere CR transversal maps. This proves part (b) of Theorem~\ref{thm:main1}. Indeed, if a CR map is CR transveral at one point, it is CR transversal at any point in its domain, as can be checked for the maps listed in part (a) of \cref{thm:main1}, see also \cref{rem2}.

\begin{lem}\label{lem:nontrans}
Any smooth CR map $H$ from $\heis$ into $\lc$, which is not CR transversal at any point of $\heis$ is equivalent to a map $(z,w) \mapsto (0,\phi(z,w),0)$ for a CR function $\phi$ satisfying $\phi(0,0) = 0$.
\end{lem}
Note that in \cite[Example 2.5]{xiao2020holomorphic}, Xiao--Yuan also gave examples of nowhere CR transversal maps from the sphere into the boundary of the classical domain of type IV.
\begin{proof}
By homogeneity we can assume that $H$ sends $0$ to $0$. If $H$ is nowhere CR transversal in $\heis$ it satisfies the equation:
\begin{align}\label{eq:nowhereME}
	(g(z,w) - \gba(\bar z, \bar w))(1 - |\phi(z,w)|^2) - 2i |f(z,w)|^2 -2i \Re \left(f(z,w)^2 \pba(\bar z, \bar w)\right) = 0,
\end{align}
for all $(z,w,\bar z, \bar w)$ in a neighbourhood of $0$. Setting $\bar z = \bar w = 0$ shows $g= 0$. Note that $\phi= 0$ implies $f= 0$. Assume that $\phi \neq 0$. We write $f= \sum_{\ell \geq \ell_0} f_\ell$ and $\phi=\sum_{k \geq k_0} \phi_k$, where $f_r$ and $\phi_r$ are weighted homogeneous polynomials of order $r$ w.r.t. the weight $(1,2)$ for $(z,w)$ and $\ell_0,k_0 \geq 0$. Collecting terms of weight $m$ in \eqref{eq:nowhereME} we obtain:
\begin{align}\label{eq:nowhereWeighted}
2 \sum_{i=0}^m f_i \bar f_{m-i} + \sum_{j=0}^m \sum_{k=0}^j \bar f_{m-j} \bar f_{j-k} \phi_k + \sum_{l=0}^m \sum_{n=0}^l f_{m-l} f_{l-n} \bar \phi_n = 0,
\end{align} 
where we skip the arguments.
We assume $\phi_{k_0} \neq 0$ and $\phi_k = 0$ for $k < k_0$. Considering $m=k_0$, shows $\sum_{i=0}^{k_0} f_i \bar f_{k_0-i}=0$, which implies $f_r = 0$ for $r\leq k_0$. Considering $m=2 k_0+2$ shows $f_{k_0+1} = 0$. Inductively, assuming $f_{k_0+s} = 0$ for $s<r$ and considering $m= 2 (k_0 + r)$ for $r \geq 2$, we obtain that $f_{k_0+r} = 0$. 
This shows $f=0$ and completes the proof.
\end{proof}
\subsection{Proof of Corollary~\ref{cor:balld4}}
We can transform the maps given in Theorem~\ref{thm:main1}, part (a), into proper holomorphic maps from $\mathbb{B}^2$ into $\tfour$ using the birational transformation $\Phi$. First, let $\gamma_0$ be the automorphism of the ball given by
\[
\gamma_0(z,w) = \left(-\frac{2 \sqrt{2} z}{w-3},\frac{1-3 w}{w-3}\right)
\]
and $\mathcal{C}$ be the Cayley transform
\[
\mathcal{C}(z,w) = \left(\frac{z}{1+w},\ i\left(\frac{1-w}{1+w}\right)\right)
\]
which sends $\mathbb{S}^3\setminus \{(0,-1)\}$ onto the Heisenberg hypersurface $\heis$.
Then we obtain the map $R_0$ as in (i) via
\[
R_0 = -\Phi \circ \ell \circ \mathcal{C}\circ \gamma_0.
\]
Moreover,
\[
\vr_{\tfour} \circ R_0 = \vr_{\mathbb{S}^3}
\]
and we have $R_0$ has vanishing Ahlfors tensor: $\mathcal{A}(R_0) = 0$.

The family $R_{\alpha}$ obtained by composing $\Phi$ with $r_\alpha$ and the Cayley transform, i.e., 
\[
	R_\alpha = -\Phi \circ r_\alpha \circ \mathcal{C} \circ \gamma_0
\]
gives an 1-parameter smooth deformation of $R_0$ by proper holomorphic maps. This family has the same properties as the family $r_\alpha$, that all $R_{\alpha}$'s with $\alpha>0$ are mutually equivalent, and a similar statement holds for the case $\alpha<0$, yet $R_0$ is not equivalent to $R_{\alpha}$ for $\alpha \ne 0$. This shows that, in contrast to the higher dimensional, but small codimensional case, treated in \cite{xiao2020holomorphic}, $R_0$ is \emph{not} rigid. Observe that $R_{\alpha}$ is rational of degree two with a rather simple formula. For each $\alpha \ne 0$, $R_{\alpha}$ is equivalent to a polynomial (quadratic) map.

Clearly, the map $P_{1}(z,w) =(zw, (z^2-w^2)/2, i(z^2+ w^2)/2)$ is a proper polynomial map sending $\mathbb{B}^2$ into $\tfour$. On the boundary, $P_{1}$ sends $\mathbb{S}^3$ into the smooth boundary part $\mathcal{R}$ transversally. Since
\[
\vr_{\tfour} \circ P_1 = (1+|z|^2 + |w|^2)\,\vr_{\mathbb{S}^3},
\]
we obtain, with $Z_{1} = \wba \partial_z - \zba \partial_w \in \Gamma(T^{(1,0)} \mathbb{S}^3)$, 
that the Ahlfors invariant of $P_1$ is
\[
\mathcal{A}(P_1)(Z_1,\Zba_1)
=
\frac{1}{2}.
\]
Similarly, the map $P_{-1}(z,w) = (z, w^2/2, iw^2/2)$ satisfies
\[
\vr_{\tfour} \circ P_{-1} = (1-|z|^2 + |w|^2)\,\vr_{\mathbb{S}^3}
\]
while
\[
\tilde{\rho}_{\tfour} \circ P_{-1} \biggl|_{\mathbb{S}^3}
=
-1 + |z|^2 + \frac{|w|^4}{2} \leqslant 0, 
\]
with the equality appearing if and only if $w=0$. Thus, $P_{-1}$ sends the circle $\{(e^{it},0)\}\subset \mathbb{S}^3$ into the singular locus of $\partial \tfour$ and sends $\mathbb{S}^3\setminus \{(e^{it},0)\}$ into the smooth boundary part $\mr\subset \partial \tfour$ transversally. The Ahlfors invariant of $P_{-1}$ is defined on $\mathbb{S}^3\setminus \{(e^{it},0)\}$:
\[
\mathcal{A}(P_{-1})(Z_1,\Zba_1)
=
-\frac{1}{2|w|^2},
\quad w \ne 0.
\]
Thus, these three maps $R_0, P_1$, and $P_{-1}$ are mutually inequivalent. Moreover, $R_{\alpha}\sim P_1$ for $\alpha >0$ and similarly $R_{\alpha} \sim P_{-1}$ for $\alpha <0$ (in contrast, the trace of the Ahlfors tensor of CR maps into the sphere is always nonnegative \cite{lamel2019cr}). 

The map $I$ satisfies
\[
\vr_{\tfour} \circ I = \vr_{\mathbb{S}^3}
\]
and thus has vanishing Ahlfors tensor. Consequently, $I$ is of geometric rank zero at all points. Clearly, any smooth deformation of $I$ along the sphere must be equivalent to $I$.

None of the rational maps $R_0, P_1$, and $P_{-1}$ can be equivalent to the irrational map $I$, since all the automorphisms of the source and target are rational. Hence, these four maps are pairwise inequivalent. This can also be verified by analyzing elementary geometric properties of the maps as already explained in \cref{rem2}. 

Finally, assume $H\colon \mathbb{B}^2 \to \tfour$ is a proper holomorphic map that extends smoothly to a boundary point $p\in \mathbb{S}^3$. We can assume that $p=(0,1)$ and $H(p) = \left(0,\frac{1}{2},\frac{i}{2}\right)$. Then $\tilde{H}: = \Phi^{-1} \circ H \circ \mathcal{C}^{-1}$ defines a germ at the origin of CR maps sending $\heis$ into $\lc$. Thus, it must belong to one of the five equivalence classes of the germs represented by $r_0, r_{1}, r_{-1}$, $\iota$, and $t$ for some nowhere CR transversal map $t$ as given in \cref{lem:nontrans}. But if $\tilde H$ is equivalent to $t$, then $H$ cannot be proper as can be easily checked. Hence, $\tilde{H}$ must be in one of the equivalence classes represented by the germs of CR transversal maps, which are the four equivalence classes of the germs of $\Phi^{-1} \circ F\circ \mathcal{C}^{-1}$ with $F\in \{R_0, P_1, P_{-1}, I\}$. In fact, the four germs of CR maps $\Phi^{-1} \circ F\circ \mathcal{C}^{-1}$, $F\in \{R_0, P_1, P_{-1}, I\}$ are pairwise inequivalent as can be easily checked using the Ahlfors invariant and the irrationality. Thus, there are local CR automorphisms $\psi\in \Aut (\lc, 0)$ and $\gamma \in \Aut (\heis, 0)$ such that $\psi \circ \tilde H \circ \gamma^{-1} = \Phi^{-1} \circ F\circ \mathcal{C}^{-1}$ near the origin for some $F\in \{R_0, P_1, P_{-1}, I\}$. Thus, if $\tilde \psi : = \Phi \circ \psi \circ \Phi^{-1}$ and $\tilde \gamma := \mathcal{C}^{-1} \circ \gamma \circ \mathcal{C}$, then, as germs at $p$, $\tilde \psi \circ H \circ \tilde \gamma^{-1} =F  \in \{R_0, P_1, P_{-1}, I\}$. But $\tilde \psi$ is a global automorphism of $\tfour$ by Theorem \ref{thm:alex} and $\tilde{\gamma}$ is a global automorphism of the unit ball. This completes the proof of \cref{cor:balld4}.
\begin{rem}
Although $P_{1}$ and $P_{-1}$ have quite simple formulas, to the best of the authors' knowledge, they didn't appear before in the literature. Take the family of rational maps $r_{\beta} \colon \heis \to \lc$ as above and transform it into a family of rational maps $\Phi \circ r_{\beta} \circ \mathcal{C}$ to get various representatives of the equivalence classes of $P_{\pm 1}$. For certain values of $\beta$, the formulas simplify. For example, if $\beta = -1/4$, then the map $\Phi \circ r_{\beta} \circ \mathcal{C}$ sends $\mathbb{B}^2$ into $\tfour \cap \{z_3 - i z_2 =0 \}$. This suggest that we substitute $z_3 = i z_2$ into $\vr_{\tfour}$ to obtain
	\[
		\vr_{\tfour}(z_1, z_2, iz_2)
		=
		1 - 2|z_1|^2 - 4|z_2|^2 + |z_1|^4
		=
		(1- |z_1|^2)^2 - 4|z_2|^2.
	\]
From this, we can easily find the desired map $P_{-1}$. The formula for $P_{1}$ can be constructed easily by analyzing the map $\Phi \circ r_{\beta} \circ \mathcal{C}$ with $\beta =1/4$.
\end{rem}

\section{Examples}
\label{sec:examples}
Observe that $\lc \cap \{\zeta = 0\}$ is CR equivalent to the Heisenberg hypersurface $v = |z|^2$. In fact, $\lc \cap \{\zeta = 0\}$ is the image of $\heis$ under the linear map $\ell(z,w)$ in part (i) of Theorem~\ref{thm:main1}. The formula for the stability group of $\Aut(\lc,0)$ suggests we consider the variety $V = \{w \zeta + i z^2 =0\} \subset \CC^3$ which has a singularity at the origin. On $V$, we have $\Re (z^2 \bar{\zeta})= - |\zeta|^2 v$ and thus on $\lc \cap V$,
\[
v = \frac{|z|^2 - v |\zeta|^2}{1 - |\zeta|^2},
\]
implying $v = |z|^2$. The real subvariety $V \cap \lc$ contains a complex variety $(0,\zeta,0)$, and a graph over the Heisenberg hypersurface $v = |z|^2$ given by $(z,- i z^2/w,w)$, $w\ne 0$. All germs of nontranversal CR maps at the origin send an open set in $\heis$ into this complex variety, see Lemma~\ref{lem:nontrans}. From Lemma~\ref{lem:stab}, we find that $V \cap \lc$ is ``invariant'' in the following sense: For $(\tilde{z},\tilde{\zeta},\tilde{w})$ in the stability group, we have
\[
\tilde{w}\tilde{\zeta} + i \tilde{z}^2
=
-\lambda^2 u^2 (w \zeta + i z^2)/\delta.
\]
That is $\gamma (\lc \cap V) \subset \lc \cap V$ for all $\gamma \in \Aut(\lc,0)$. From the observation above, we can construct a examples of maps from $\heis$ into $\lc$ as follows.  

\begin{example}
Each map in the family of rational maps $H_t(z,w) = (z,-iz^2/(w+t),w+t)$, $t\in \mathbb{R}$, sends $\mathbb{H}^3\setminus \{(0,-t) \}$ into $\lc\cap V$ transversally, but does not extend holomorphically to the point $(0,-t)$. It turns out that at any point $p \in \mathbb{H}^3\setminus \{(0,-t) \}$, $H_t$ represents a germ at $p$ that is equivalent to the germ of $\ell$ at the origin. In fact,
\[
	\widetilde{\vr}_{\lc} \circ H_t = |w+t+2i|^{-2} \vr_{\heis},
\]
and hence outside the pole of $H_t$, the Ahlfors tensor $\mathcal{A}(H_t)$ vanishes identically. Thus, $H_t$ must be equivalent to either $\ell$ or $\iota$. But the latter possibility is ruled out by the rationality and $H_t$ is equivalent to $\ell$. Transferring this map into a map from the unit ball into the type IV model, we obtain a map that is equivalent to $R_0$. Explicitly, with $t=0$, we have
\[
	\Phi \circ H_0\circ \mathcal{C} \circ \gamma_1 = - R_0,
\]
for 
\[
	\gamma_1(z,w) = \left(-\frac{2 \sqrt{2} z}{w+3},-\frac{3 w+1}{w+3}\right)
\]
is an automorphism of $\mathbb{B}^2$ exchanging the origin and $(0,-1/3)$.
\end{example}

\begin{example}\label{ex:linmap}
As an application of the transitive part of the automorphisms in \cref{sec:autom}, we explicitly verify that the rational map $R: \mathbb{B}^2 \rightarrow \tfour$ from \cite[Theorem 1.4]{xiao2020holomorphic}, given by 
\begin{align*}
R(z_1,z_2) = \left(z_1, \frac{z_1^2/2 -z_2^2 + z_2}{\sqrt{2}(1-z_2)}, \frac{i(z_1^2/2+ z_2^2-z_2)}{\sqrt{2}(1-z_2)} \right),
\end{align*}
corresponds to the linear map $\ell: \heis \rightarrow \lc ,\ell(z,w)=(z,0,w)$ from \cref{thm:main1}. Note that $R$ sends $p=(0,1/2) \in \mathbb{B}^2$ to $q=(0,1/2 \sqrt{2}, -i/2 \sqrt{2}) \in \tfour$. We define $\Psi_1: \CC^3\setminus\{z_2 - i z_3 = \sqrt{2}\} \rightarrow \CC^3$, given by
\begin{align*}
\Psi_1(z_1,z_2,z_3)= \left(\frac{-2 z_1}{\sqrt{2}-z_2+i z_3}, \frac{2 z_2 + 2 i z_3- \sqrt{2}(z_1^2+z_2^2+z_3^2)}{2(\sqrt{2}-z_2+i z_3)}, \frac{2(\sqrt{2}i + i z_2 + z_3)}{\sqrt{2}-z_2+i z_3}\right),
\end{align*}
which locally sends $\mr$ into $\lc$ and $q$ to $(0,1/2,2 i) \in \lc$, and $\Psi_2: \CC^2\setminus\{w = -i\} \rightarrow \CC^2$, given by
\begin{align*}
\Psi_2(z,w)= \left(\frac{i-w}{i+w},\frac{2 i z}{i+w} \right),
\end{align*}
which sends $\heis$ into $\mathbb{S}^3$ and $(1/2,i)$ to $p$. We consider
\begin{align*}
\hat R(z,w)  = \Psi_1 \circ R \circ \Psi_2 = & \left(\frac{2\sqrt{2}(1-2 z-2 i zw+w^2)}{-1+4 z + (6 i - 4 i z+w)w},\frac{1-4 z+2i w(1+2 z)+8 z^2-w^2}{-1+4 z + (6 i - 4 i z+w)w} \right., \\
& \quad \left. 2 i + \frac{4i(i - w)^2}{-1+4 z + (6 i - 4 i z+w)w} \right),
\end{align*}
which locally maps $\heis$ to $\lc$ and $\hat p=(0,1)$ to $\hat q=\frac{1}{3}(-2 \sqrt{2}i,1,2 i)$. Next, we compose $\hat R$ with automorphisms to ensure that it sends $0$ to $0$. More precisely, we define
\begin{align*}
\varphi(z,w)=(z,w+1), 
\end{align*}
which is an automorphism of $\heis$, and 
\begin{align*}
\varphi'_1(z,\zeta,w) & =\left(\frac{2 \sqrt{2} z}{3-\zeta},\frac{3 \zeta-1}{3-\zeta},w-\frac{i z^2}{3-\zeta} \right), \\
\varphi'_2(z,\zeta,w) & =\left(z+i(1+\zeta),\zeta,2 z + w +i (1+\zeta)\right),
\end{align*}
which originate from choosing $\lambda' = 1/2$ in \eqref{eq:transAuto2} and $b=i$ and $r=0$ in \eqref{eq:transAuto1} respectively. If we set 
\begin{align*}
\check R(z,w) & = \varphi_2' \circ \varphi_1' \circ \hat R  \circ \varphi\\
& = \left(\frac{2(2i (z-(1-i))z - 2 i z w +w ^2)}{-2((i-1)+z)^2+((2+ 4 i)-4 i z)w + w^2}, \right.\\
& \qquad \frac{2 z(3 z-(2-2i)) - w(2 - 4 i z)-w^2}{-2((i-1)+z)^2+((2+ 4 i)-4 i z)w + w^2},\\
& \qquad \left.\frac{(4+4 i)w((1+i)-2 i z + w)}{-2((i-1)+z)^2+((2+ 4 i)-4 i z)w + w^2} \right),
\end{align*}
we obtain a map, which, again, locally sends $\heis$ into $\lc$ and $0$ to $0$. For the final step, we define the following automorphisms of $\heis$ and $\lc$ respectively, fixing $0$,
\begin{align*}
\phi(z,w) & = \left(\frac{2 z + w}{(2+ 2 i)(i+z)+w},\frac{(i-1) w}{ (2+ 2 i)(i+z)+w} \right),\\
\phi'(z,\zeta,w) & = \left(-2 i + \frac{8 + 4 i z}{i(z-2 i)^2 + (\zeta-1)w}, \frac{(4 + i z)z - 4 i \zeta + (\zeta-1) w}{i(z-2 i)^2 + (\zeta-1)w}, \right.\\
& \left. \qquad \frac{- 4 i w}{i(z-2 i)^2 + (\zeta-1)w} \right),
\end{align*} 
which are obtained by following the procedure in the proof of \cref{thm:normalform}. It can be verified that the map $L=\phi' \circ \check R \circ \phi$ agrees with $\ell$.
\end{example}

\begin{rem}
Carrying out the analogous steps as in \cref{ex:linmap} it can be shown that the irrational map $I:\mathbb{B}^2 \rightarrow \tfour$ given by $I(z,w) = \left(z,w,1-\sqrt{1-z^2-w^2}\,\right)\bigl/ \sqrt{2}$ found by Xiao--Yuan in \cite[Theorem 1.4]{xiao2020holomorphic} corresponds to the irrational map $\iota$ obtained in \cref{thm:main1}. More precisely, it is possible to obtain a map $\check I$ from $I$, which sends $(\heis,0)$ to $(\lc,0)$. After that, bringing $\check I$ to the partial normal form given in \cref{thm:normalform}, it can be verified that the $4$-jets at the origin of this map and $\iota$ agree. Since the formulas involved are more complicated than in the rational case we refrain from displaying them here.
\end{rem}

\begin{example} 
\label{ex:holEqUpTo14}
The mapping equation \eqref{eq:mapeq} can be regarded as a system of infinitely many linear and nonlinear equations for the Taylor coefficients of the components of the map under consideration. Each solution to this system gives rise to holomorphic or formal CR map sending the germ $(\heis,0)$ into $(\lc, 0)$ as a subset of formally. We give an example of a holomorphic map whose Taylor coefficients solve all \emph{but} three equations of weight 14 as well as three holomorphic functional equations in Lemma \ref{lem:holeq2}, yet the map does \emph{not} send $\heis$ into $\lc$. Indeed, consider the following family of holomorphic maps
\[
H:=(f,\phi,g) = \left(z+ 2t zw^3, -4it z^2w^2, w+  t w^4\right), \ t\in \mathbb{R}.
\]
which satisfies the functional equations in Lemma \ref{lem:holeq2}. 
If $\Sigma_1 \colon (z,\zba, \wba) \mapsto (z, \wba +2 i z\zba )$ is a parametrization of the (complexification) of the Heisenberg hypersurface, then
\[
\widetilde{\vr} \circ H_t \circ \Sigma_1
=
-4 t^2 \wba ^3 z\zba (\wba +2 i z\zba)^3.
\]
Hence, $H_t$ sends $\heis$ into $\lc$ if and only if $t=0$.
\end{example}

%
\end{document}